\documentclass[a4paper,12pt]{amsart}
\input xypic
\xyoption{all}
\usepackage{amssymb}
\newtheorem{thm}{Theorem}[section]
\newtheorem{lemma}[thm]{Lemma}
\newtheorem{cor}[thm]{Corollary}
\newtheorem{prop}[thm]{Proposition}
\newtheorem{defi}[thm]{Definition}
\newtheorem{conj}[thm]{Conjecture}
\theoremstyle{definition}
\newtheorem{rem}[thm]{Remark}

\def\per{{\pi}}

\def\tG{{\tilde{G}}}
\def\tH{{\tilde{H}}}
\def\tR{{\tilde{R}}}
\def\tDelta{{\tilde{\Delta}}}

\def\BB{{\mathbf{B}}}
\def\BF{{\mathbf{F}}}
\def\BG{{\mathbf{G}}}
\def\BL{{\mathbf{L}}}

\def\BT{{\mathbf{T}}}
\def\BU{{\mathbf{U}}}
\def\BZ{{\mathbf{Z}}}
\def\CA{{\mathcal{A}}}
\def\CB{{\mathcal{B}}}
\def\CC{{\mathcal{C}}}
\def\CE{{\mathcal{E}}}
\def\CF{{\mathcal{F}}}

\def\CO{{\mathcal{O}}}
\def\CP{{\mathcal{P}}}
\def\CT{{\mathcal{T}}}
\def\BR{{\mathbf{R}}}
\def\CS{{\mathcal{S}}}
\def\GA{{\mathfrak{A}}}
\def\GS{{\mathfrak{S}}}

\def\Aut{\operatorname{Aut}\nolimits}
\def\Br{\operatorname{Br}\nolimits}
\def\can{{\mathrm{can}}}

\def\Comp{\operatorname{Comp}\nolimits}
\def\en{{\mathrm{en}}}
\def\End{\operatorname{End}\nolimits}
\def\Ext{\operatorname{Ext}\nolimits}
\def\GL{\operatorname{GL}\nolimits}
\newcommand{\GU}{\operatorname{GU}}
\def\Ho{\operatorname{Ho}\nolimits}
\def\Hom{\operatorname{Hom}\nolimits}
\def\id{\operatorname{id}\nolimits}
\def\im{\operatorname{im}\nolimits}
\def\Ind{\operatorname{Ind}\nolimits}
\def\Irr{\operatorname{Irr}\nolimits}
\def\mMod{\operatorname{\!-mod}\nolimits}
\def\opp{{\operatorname{opp}\nolimits}}
\def\Out{\operatorname{Out}\nolimits}
\def\Pic{\operatorname{Pic}\nolimits}
\newcommand{\PGL}{\operatorname{PGL}}
\def\PSL{\operatorname{PSL}\nolimits}
\newcommand{\PSp}{\operatorname{PSp}}
\newcommand{\PSU}{\operatorname{PSU}}
\def\Res{\operatorname{Res}\nolimits}
\def\SL{\operatorname{SL}\nolimits}
\def\soc{\operatorname{soc}\nolimits}
\def\Sp{\operatorname{Sp}\nolimits}

\def\mstab{\operatorname{\!-stab}\nolimits}
\def\TrPic{\operatorname{TrPic}\nolimits}

\newcommand{\Proj}[1]{\mathcal P({#1})}
\newcommand{\Alt}[1]{\GA_{#1}}
\newcommand{\Sym}[1]{\GS_{#1}}
\newcommand{\Cyc}[1]{Z_{#1}}
\newcommand{\F}{\BF}

\newcommand{\Cent}{\operatorname{C}}

\def\eps{\varepsilon}
\def\iso{\buildrel \sim\over\to}
\def\ie{{\em i.e.}}

\title{Perverse equivalences and Brou\'e's conjecture}
\author{David A.~Craven and Rapha\"el Rouquier, University of Oxford}
\date\today

\begin{document}
\maketitle
\setcounter{tocdepth}{2}
\tableofcontents

\section{Introduction}

This paper proposes a new approach to the construction of derived equivalences, based on perverse equivalences. These equivalences, introduced by Joe Chuang
and the second author, aim to encode combinatorially the derived equivalence class of a block with respect to a fixed block
\cite{ChRou2}. The derived equivalences between blocks of symmetric groups constructed in \cite{ChRou1} are
compositions of perverse equivalences, and it is expected that the
conjectural derived equivalences provided by Deligne--Lusztig varieties for finite groups of Lie type in non-describing characteristic are perverse. This has motivated our search for perverse equivalences in the case
of sporadic groups, in the setting of
Brou\'e's abelian defect conjecture (cf.\ \cite{ChRi} for a survey on
 Brou\'e's conjecture). We
also consider certain finite groups of Lie type: in those cases,
we provide equivalences which should coincide with the conjectural
ones coming from Deligne--Lusztig varieties.

We provide lifts of stable equivalences to perverse equivalences. Our method
requires only calculations within the normalizer of a defect group. In the
cases we consider, Brou\'e's conjecture was already known to hold. The
equivalences we provide are different from the known ones, which were
usually not perverse. Also, to obtain perverse equivalences, we sometimes
need to change the ``obvious'' stable equivalence, given by Green
correspondence:
this depends on local data. For those groups of Lie type we consider,
this is dictated by the properties of Deligne--Lusztig varieties. We
consider finite groups with elementary abelian Sylow $3$-subgroups of
order $9$: for these groups, Green correspondence provides a stable equivalence.
In addition, we have a complete description of the local normalized 
self-derived equivalences, which enables us to parametrize splendid
self-stable equivalences.
Note that there are stable equivalences between blocks that can be lifted to
a derived equivalence but not to a
perverse derived equivalence. This occurs for example when all simple modules
of one of the blocks can be lifted to characteristic $0$ (that happens
for the local block in the setting of Brou\'e's conjecture), while the
decomposition matrix of the other block cannot be put in a triangular form.

\medskip
Our approach can also be viewed as an 
attempt to mimic the extra structure carried
by representations of finite groups of Lie type in non-describing
characteristic to the case of arbitrary finite groups. Our equivalences
depend on the datum of a perversity function $\per$, which is related
to Lusztig's $A$-function for finite groups of Lie type.

Extensions of equivalences through $\ell'$-groups of automorphisms are
easy to carry for perverse equivalences, and we devote an important part
of this paper to the study of extensions of equivalences. Our
main point is that checking that a two-sided tilting complex will extend
depends only on the underlying (one-sided) tilting complex. We deduce
from our results new methods to check that equivalences extend. This enables
us to show for instance that Brou\'e's abelian defect conjecture holds
for principal blocks in characteristic $2$.

\smallskip
Section \ref{se:construction} is devoted to constructing equivalences.
We first explain the description of the images of simple modules under perverse
equivalences associated to increasing perversity (\S \ref{se:perverse}), then
we explain our method for lifting stable equivalences to perverse
equivalences (\S \ref{se:lifts}). In \S \ref{se:stablelxl}, we explain
the construction of stable equivalences for principal blocks with elementary abelian Sylow
$\ell$-subgroups of order $\ell^2$. We describe in detail
the images of modules and we describe a family of stable equivalences
dependent on local perversity functions when $\ell=3$.
Finally, we recall in \S \ref{se:Lie} the setting of Brou\'e's conjecture
for finite groups of Lie type, where two-sided tilting complexes are
expected to arise from Deligne--Lusztig varieties. We study in particular
those finite groups of Lie type with a Sylow $3$-subgroup elementary
abelian of order $9$. We also recall Puig's construction of
equivalences for the case ``$\ell\mid(q-1)$''.

Section \ref{se:automorphismes} is devoted to automorphisms and extensions
of equivalences.  In \S \ref{se:stability}, we set up a general formalism that
allows a reduction to
finite simple groups for equivalences of a suitable type between the principal
block of a finite group with an abelian Sylow $\ell$-subgroup $P$, and
the principal block of $N_G(P)$. This is meant to encompass the various
forms of Brou\'e's abelian defect group conjecture. In \S \ref{se:auto}, we
provide extension theorems for equivalences. We give criteria that
ensure that a two-sided tilting complex can be made equivariant for
the action of a group of automorphisms; we recover results of Rickard and
Marcus.  We consider in particular (compositions of) perverse equivalences.
In \S \ref{se:extra}, we apply the general results of the previous sections
to various forms of Brou\'e's conjecture: derived equivalences, Rickard
equivalences, splendid or perverse properties, positivity of gradings and
perfect isometries. This provides us in \S \ref{se:principal}
with a general reduction theorem to simple groups, generalizing a result
of Marcus. We apply this to show that Brou\'e's conjecture can be
solved using perverse equivalences for certain cases when $\ell=2$ or
$3$.

This last result is obtained by a case-by-case study of finite simple groups
with elementary abelian Sylow $3$-subgroups of order $9$ in
Section \ref{se:defect33}.
We provide a perverse equivalence with the normalizer of a Sylow $3$-subgroup
for all such groups except $\GA_6$ and $M_{22}$, for which we need
the composition of two perverse equivalences.
Perverse equivalences are encoded in
global and local perversity functions. Note that these combinatorial
data determine the source algebra of the block up to isomorphism. While
Brou\'e's conjecture was known to hold in all cases considered (work
of Koshitani, Kunugi, Miyachi, Okuyama, Waki), we have
been led to construct a number of new equivalences.
In Section \ref{se:ell2}, we provide an analysis of simple groups with
abelian Sylow $2$-subgroups.

\smallskip
We thank Jean Michel and Hyohe Miyachi for useful discussions and help
with references.

\section{Notation and basic definitions}
\subsection{Algebras}
\subsubsection{Modules}
All modules are finitely generated left modules, unless otherwise specified.

Let $R$ be a commutative ring. We write $\otimes$ for $\otimes_R$. Given $q$ a prime power, $\BF_q$ denotes a finite field with $q$ elements.

Let $A$ be an $R$-algebra. We denote by $A^\opp$ the opposite algebra to $A$ and we put $A^\en=A\otimes A^\opp$. Given an $R$-module $M$, we put $AM=A\otimes M$, an $A$-module. We denote by $\CS_A$ a complete set of representatives of
isomorphism classes of simple $A$-modules.

Let $M$ be a finitely generated module over an artinian algebra. The \emph{head} of $M$ is defined to be its largest semi-simple quotient. We denote by $I_M$ (resp.\ $P_M$ or $\Proj{M}$) an injective hull (resp.\ a projective cover) of $M$. We denote by $\Omega(M)$ the kernel of a surjective map $P_M\to M$ and by $\Omega^{-1}(M)$ the cokernel of an injective map $M\to I_M$. We define by induction $\Omega^i(M)=\Omega\bigl(\Omega^{i-1}(M)\bigr)$ and $\Omega^{-i}(M)=
\Omega^{-1}\bigl(\Omega^{-i+1}(M)\bigr)$ for $i>1$.

Let $\sigma:B\to A$ be a morphism of algebras and let $M$ be an $A$-module. We define a $B$-module ${_\sigma M}$: it is equal to $M$ as a $k$-module, and the action of $b\in B$ given by the action of $\sigma(b)$ on $M$.

The algebra $A$ is \emph{symmetric} if it is finitely generated and projective as an $R$-module, and if $\Hom_R(A,R)\simeq A$ as
$A^\en$-modules.

\subsubsection{Categories}

We denote by $A\mMod$ the category of finitely generated $A$-modules.

Let $\CC$ be an additive category and $\CA$ an abelian category.
We denote by
\begin{itemize}
\item $\Comp(\CC)$ the category of complexes of objects of $\CC$,
\item $\Ho(\CC)$ the homotopy category of $\Comp(\CC)$, and
\item $D(\CA)$ the derived category of $\CA$.
\end{itemize}

A complex in $\CC$ is \emph{contractible} if it is $0$ in $\Ho(\CC)$, and a complex in $\CA$ is \emph{acyclic} if it is $0$ in $D(\CA)$. We write $\Comp(A)$ for $\Comp(A\mMod)$, and so on.

We write $0\to M\to N\to \cdots\to X\to 0$ (or sometimes
$M\to N\to \cdots\to X$) for a complex where $X\not=0$ is
in degree $0$.

Given $M,N\in\Comp(\CC)$, we denote by
$\Hom^\bullet(M,N)$ the complex with degree $n$ term
$\bigoplus_{j-i=n}\Hom(M^i,N^j)$. We denote by $R\Hom^\bullet$ the
derived version.

\smallskip
Given two algebras $A$ and $B$, we say that a functor
$D^b(A)\to D^b(B)$ is {\em standard} if it is of the form
$C\otimes_A -$, where $C$ is a bounded complex of $(B,A)$-modules, 
finitely generated and projective as $B$-modules and as $A^\opp$-modules.

A {\em tilting complex} $C$ for $A$ is a perfect complex of $A$-modules
(\ie, quasi-isomorphic to a bounded complex of finitely generated
projective $A$-modules) such that $A$ is in the thick subcategory of
$D(A)$ generated by $C$ and $\Hom_{D(A)}(C,C[i])=0$ for $i\not=0$.

A {\em two-sided tilting complex} $C$ for $(A,B)$ is a bounded complex of
$(A,B)$-bimodules such that the functor
$C\otimes_B^\BL -:D(B)\to D(A)$ is an equivalence.

A {\em Rickard complex} $C$ for $(A,B)$ is a bounded complex of
$(A,B)$-bimodules, finitely generated and projective as left $A$-modules
and as right $B$-modules, such that the functor
$C\otimes_B -:\Ho(B)\to \Ho(A)$ is an equivalence.

\smallskip
Assume that $A$ is a symmetric $R$-algebra. We denote by
$A\mstab$ the stable category, the triangulated quotient of $D^b(A)$ by the thick subcategory
of perfect complexes. The canonical functor $A\mMod\to A\mstab$ identifies
$A\mstab$ with the additive quotient of $A\mMod$ by the subcategory of
projective modules.

Let $B$ be a symmetric $R$-algebra. A bounded
complex $C$ of $(A,B)$-bimodules {\em induces a stable equivalence} if its
terms are projective as left $A$-modules and as right $B$-modules, and
there are isomorphisms of complexes of bimodules $\End^\bullet_A(C)\simeq B\oplus R_1$
and $\End^\bullet_{B^\opp}(C)\simeq A\oplus R_2$, where $R_1$ and $R_2$ are
homotopy equivalent to bounded complexes of projective bimodules. There is
an equivalence $C\otimes_B -:B\mstab\iso A\mstab$.

We say that a Rickard complex $C$ {\em lifts} a complex $D$ inducing
a stable equivalence if $C$ and $D$ are isomorphic in the quotient of
$\Ho^b(A\otimes B^\circ)$ by its thick subcategory of complexes
of projective modules.

\subsection{Modular setting}
We will denote by $\CO$ a complete discrete valuation ring with 
residue field $k$ of characteristic $\ell>0$ and field of fractions $K$
of characteristic $0$. 

\subsection{Groups}
We denote by $\Cyc{n}$ a cyclic group of order $n$, by $D_n$ a dihedral group of
order $n$, by $SD_n$ a semi-dihedral group of order $n$, and by $\GA_n$ and $\GS_n$ the alternating and symmetric groups of degree $n$ respectively. If $G$ is a finite group, we denote by $G^\opp$ the opposite group
and we set $\Delta G=\{(g,g^{-1}) \mid g\in G\}\le G\times G^\opp$.

We denote by $b_0(G)$ the principal block idempotent of $\BZ_\ell G$ and
by $B_0(G)=b_0(G)\BZ_\ell G$ the principal block algebra. Given an $\CO G$-module $M$ and an $\ell$-subgroup $P$ of $G$,
we denote by $\Br_P(M)$ the image of $M^P$ in the coinvariants $M/\{g(m)-m\}_{m\in M,g\in P}$; this is an
$\CO(N_G(P)/P)$-module.

\smallskip
Let $R$ be either $\CO$ or $k$ and let $H$ be a finite group.  Assume that $H$ and $G$ have a
common Sylow $\ell$-subgroup $P$. We say that a bounded complex
$C$ of $RB_0(H\times G^\opp)$-modules is {\em splendid} if
its terms are direct summands of finite direct sums of modules of the form
$\Ind_{\Delta Q}^{H\times G^\opp}R$, where $Q\le P$.

\section{Constructions of equivalences}
\label{se:construction}
\subsection{Perverse equivalences}
\label{se:perverse}
We explain the constructions of \cite{ChRou2} (see also \cite[\S 2.6]{Rou5}).
\subsubsection{Definition}
Let $A$ and $B$ be two finite-dimensional algebras over a field $k$.
Fix $r\ge 0$, $q:[0,r]\to\BZ$ and fix filtrations $\emptyset=\CS_{-1}\subset\CS_0\subset\cdots\subset\CS_r=\CS_A$
and $\emptyset=\CS'_{-1}\subset\CS'_0\subset\cdots\subset\CS'_r=\CS_B$.
A functor $F:D^b(B)\to D^b(A)$ is \emph{perverse} relative to $(q,\CS_\bullet,\CS'_\bullet)$ if, whenever
$T$ is in $\CS'_i$, the composition factors of $H^{-j}\bigl(F(T)\bigr)$ are in $\CS_{i-1}$ for $j\not=q(i)$ and in $\CS_i$ for $j=q(i)$.

\smallskip
If $F$ is an equivalence, then
given $T\in\CS'_i$, the $A$-module
$H^{-q(i)}\bigl(F(T)\bigr)$ is the extension of
an object with composition factors in $\CS_{i-1}$ by an object $f(T)$ in
$\CS_i$ by an object with composition factors in $\CS_{i-1}$.
The map $f$ gives a bijection $\CS_B\iso \CS_A$ compatible
with the filtrations.

\subsubsection{Increasing perversity}
\label{se:increasingperversity}
Let $A$ be a symmetric $k$-algebra.
Given $M\in A\mMod$ and $E\subset \CS=\CS_A$, we denote by
$M^E$ the largest submodule $N$ of $I_M$ containing $M$ and such that
all composition factors of $N/M$ are in $E$.

Consider a map $\per:\CS\to\BZ_{\ge 0}$, and let $S\in\CS$. We define a complex of $A$-modules $C=C_S
=0 \to C^{-\per(S)}\to\cdots \to C^0\to 0$.

$\bullet\ $If $\per(S)=0$ then we set $C^0=S$.

$\bullet\ $Assume that $\per(S)>0$.
We put $C^{-\per(S)}=I_S$. Let $E=\per^{-1}([0,\per(S)-1])$.
Let $C^{-\per(S)+1}=I_{\Omega^{-1}(S^E)}$. We define
$d^{-\per(S)}$ as the composition of canonical maps
$C^{-\per(S)}=I_{S^E}\twoheadrightarrow\Omega^{-1}(S^E)
\hookrightarrow C^{-\per(S)+1}$.

Fix $0<i<\per(S)$.
Assume that $0\to C^{-\per(S)}\xrightarrow{d^{-\per(S)}}
C^{-\per(S)+1}\to\cdots\xrightarrow{d^{-\per(S)+i-1}} C^{-\per(S)+i}$
has been constructed with the property that
$C^{-\per(S)+i}=I_T$, where $T=\im d^{-\per(S)+i-1}$.
Let $E=\per^{-1}([0,\per(S)-i-1])$ and
$C^{-\per(S)+i+1}=I_{\Omega^{-1}(T^E)}$. We define
$d^{-\per(S)+i}$ as the composition of canonical maps
$C^{-\per(S)+i}=I_{T^E}\twoheadrightarrow\Omega^{-1}(T^E)
\hookrightarrow C^{-\per(S)+i+1}$.

Finally, let $T=\im d^{-2}$, $E=\per^{-1}(0)$ and $C^0=\Omega^{-1}T^E$. We 
define $d^{-1}$ as the canonical map $C^{-1}=I_{T^E}\twoheadrightarrow C^0$.

\smallskip
There is a symmetric $k$-algebra $B$, well defined up to Morita
equivalence, and a standard equivalence
$F:D^b(B)\iso D^b(A)$ such that
$\{F(T)\}_{T\in\CS_B}=\{C_S\}_{S\in\CS_A}$. This 
equivalence is perverse, relative to the filtration $\emptyset\subset \per^{-1}(0)\subset \per^{-1}([0,1])\subset
\per^{-1}([0,2])\subset\cdots\subset\CS_A$ and the corresponding filtration on
$\CS_B$, and relative to the constant perversity function $i\mapsto i$.
Note conversely that given a perversity datum $(q,\CS_\bullet,\CS'_\bullet)$
where $q$ is increasing, the perverse equivalence arises from a 
function $\pi$ where $\pi(S)=\min\{q(n)\mid S\in\CS_n\}_{n\ge 0}$.
We write
$$B\xrightarrow{\per}A$$
to denote the perverse equivalence.

\subsubsection{Elementary equivalences}
\label{se:elementary}
Assume that $\per(\CS)=\{0,1\}$. We will describe a tilting complex for $A$
in this case.

Let $U$ be the smallest submodule of the $A$-module $A$ such that
all composition factors of $A/U$ are in $\per^{-1}(0)$. 
Let $f:P_U\twoheadrightarrow U$ be a projective cover and let
$P_V$ be a projective cover of the largest submodule $V$ of $A/J(A)$ all
of whose composition factors are in $\per^{-1}(1)$.
Let $X=0\to P_V\oplus P_U\xrightarrow{(0,f)}A\to 0$ be a complex
of $A$-modules with $A$ in degree $0$. Let $B=\End_{\Ho(A)}(X)$. Then $B$
is a symmetric algebra and there is a standard perverse equivalence
$D^b(B)\iso D^b(A),\ B\mapsto X$.

\subsubsection{Perverse splendid equivalences}
\label{se:splendidperverse}

Let $G$ be a finite group with an abelian Sylow $p$-subgroup $G$, and let $H=N_G(P)$. Let $R$ be either $\CO$ or $k$.

Given $Q\le P$, let $\per_Q:\CS_{B_0(C_G(Q)/Q)}\to\BZ_{\ge 0}$ be a map. We assume that $\per_Q$ is invariant under $N_H(Q)$ and independent of $Q$ up to $H$-conjugacy.

\begin{defi}
An {\em increasing perverse splendid equivalence} between $RB_0(G)$ and $RB_0(H)$ relative to $\{\per_Q\}_Q$ is
a standard Rickard equivalence of the form $C\otimes_{RB_0(G)}-$, where
$C$ is splendid and such that for every $Q\le P$,
$\Br_{\Delta Q}(C)$ induces a perverse equivalence relative to
$\per_Q$ between $kB_0\bigl(C_G(Q)\bigr)$ and $kB_0\bigl(C_H(Q)\bigr)$.
\end{defi}

This definition generalizes immediately to the case of two arbitrary blocks
of two finite groups, using the general notion of splendid equivalences
\cite{Li3,Rou7}. An important property of splendid Rickard equivalences is that they lift
from $k$ to $\CO$ \cite[Theorem 5.2]{Ri5}.

\begin{thm}
Let $C$ be a splendid Rickard complex for $\bigl(kB_0(G),kB_0(H)\bigr)$. There
is a splendid Rickard complex $\tilde{C}$ for $\bigl(B_0(G),B_0(H)\bigr)$
such that $k\tilde{C}\simeq C$ in $\Comp\bigl(kB_0(G\times H^\opp)\bigr)$.
Furthermore,
$\tilde{C}$ is unique up to isomorphism.
\end{thm}

\begin{rem}
One can normalize the equivalence by assuming $\pi_P=0$ and
$\pi_Q(k)=0$ for all $Q$.

Also, if $\tG$ is a finite group containing $G$ as a normal subgroup of
$\ell'$-index, then one can ask for an equivariant form of the definition above
by requiring the maps $\pi_Q$ to be invariant under the action of $\tG$.
\end{rem}

\subsection{Lifts of stable equivalences}
\label{se:lifts}
\subsubsection{Recognition criteria}
\label{se:recognition}
Let $A$ and $A'$ be two symmetric algebras over a field $k$, with no
simple direct factors, and let
$L:D^b(A')\to D^b(A)$ be a standard functor
inducing a stable equivalence $\bar{L}:A'\mstab\iso A\mstab$.

Let $\per:\CS\to\BZ_{\ge 0}$. There is a symmetric algebra $B$ and
a standard perverse equivalence $F:D^b(B)\iso D^b(A)$. Assume that $\{F(T)\}_{T\in \CS_B}$ coincides, up to isomorphism in
$A\mstab$, with $\{L(S')\}_{S'\in\CS_{A'}}$.

\smallskip
The composition $F^{-1}L:D^b(A')\to D^b(B)$ is given by
tensoring by a complex $X$ of $(B,A')$-bimodules. There is
a $(B,A')$-bimodule $M$ with no non-zero projective direct summand,
projective as a $B$-module and as a right $A'$-module,
that is isomorphic to $X$ in $(B\otimes A^{\prime\opp})\mstab$. The functor
$M\otimes_{A'}-:A'\mstab\to B\mstab$ is an equivalence and it preserves
isomorphism classes of simple modules. Since $M$ has no non-zero projective direct summand, it follows that $M\otimes_{A'}S'$ is indecomposable whenever $S'\in\CS_{A'}$ and we deduce that we have an equivalence $M\otimes_{A'}-: A'\mMod\iso B\mMod$ \cite[Theorem 2.1]{Li2}.
The composition
$G=F\circ (M\otimes_{A'}-):D^b(A')\iso D^b(A)$ is a standard perverse
equivalence lifting $\bar{L}$.

\smallskip
Let $k_0$ be a subfield of $k$ such that the extension $k/k_0$ is separable.
Let $A_0$ and $A'_0$ be two symmetric
$k_0$-algebras such that $A=kA_0$ and $A'=kA'_0$. Assume that
\begin{itemize}
\item  there is a standard 
functor $L_0:D^b(A'_0)\to D^b(A_0)$ with $L=kL_0$, and
\item given 
$S\in\CS_{A_0}$ and $S_1,S_2$ two simple direct summands of $kS$, then
$\per(S_1)=\per(S_2)$. 
\end{itemize}
The second assumption gives a function $\pi_0:\CS_{A_0}\to\BZ_{\ge 0},\
S\mapsto \pi(S_1)$ where $S_1$ is a simple direct summand of $kS$.
There is a symmetric $k_0$-algebra $B_0$ and a standard perverse
equivalence $D^b(B_0)\iso D^b(A_0)$.
As above, we obtain a standard stable
equivalence $A'_0\mstab\iso B_0\mstab$ that preserves semi-simple
modules, and hence simple modules. We deduce that there is a standard
perverse equivalence $G_0:D^b(A'_0)\iso D^b(A_0)$ such that $G_0$ and
$L_0$ induce isomorphic stable equivalences and such that $kG_0\simeq G$.

\subsubsection{Strategy}
\label{se:strategy}
Assume that we are given a stable equivalence $\bar{L}$ as above. Our strategy to lift
$\bar{L}$ to a derived equivalence is to look for a function $\per$ as in
\S \ref{se:increasingperversity} such that the set of $A$-modules
$\{C_S^0\}_{S\in\CS_A}$, coincides with the set
$\{\bar{L}(S')\}_{S'\in\CS_{A'}}$.

In the setting of Brou\'e's conjecture, we take for $A'$ a block with a
normal abelian defect group (for example,
$A'=k(P\rtimes E)$ where $k$ is a field
of characteristic $\ell$, $P$ is a abelian $\ell$-group and $E$ an
$\ell'$-group).
The determination of the $L(S')$ requires the determination of the Green
correspondents of simple modules: this computation is not directly feasible
for larger groups (for example $\PSU_5(4)$).
The calculation of the $C_S$ is a reasonable computational task.
A more tricky matter is the determination of the function $\per$. There are
constraints: the filtration
on $\CS_A$ should make the decomposition matrix of $A$ triangular. Also,
the datum $\per$ modulo $2$ should come from a perfect isometry.
As for the 
specific value of $\per$, we have proceeded by trying systematically all
possibilities, increasing progressively the values of $\per$.

\medskip
Let us explain this more precisely.
Let $G$ be a finite group, $A=\CO B_0(G)$ and $H$ another finite
group with principal block $B=\CO B_0(H)$. Assume that we are given a
standard equivalence $F:D^b(A)\iso D^b(B)$ inducing a perverse
equivalence $D^b(kA)\iso D^b(kB)$ relative to $\pi:\CS_B\to \BZ_{\ge 0}$.
Note that this provides a bijection $\CS_A\iso \CS_B$.
The map $I:K_0(KA)\iso K_0(KB)$ induced by $KF$ is a perfect isometry
\cite{Br2}.
There is a map $\eps:\CS_{KA}\to\{\pm 1\}$ and a bijection
$J:\CS_{KA}\iso\CS_{KB}$ such that
$I(\chi)=\eps(\chi) J(\chi)$ for $\chi\in\CS_{KA}$.
We have $\chi(1)\equiv \eps(\chi)J(\chi)(1)\pmod \ell$, if
$I(K)=K$.

Let $Z$ be a subset of $\CS_{KB}$ whose image in
$K_0(kB)$ by the decomposition map is a basis ($Z$ is a ``basic set'').
In this case, $J^{-1}(Z)$ is a basic set for $A$.

Assume now that the image of $Z$ in $K_0(kB)$ is the basis given by $\CS_B$; this provides a bijection $Z\iso\CS_B$. Given $V\in\CS_B$ corresponding to
$\psi\in Z$, we have $\pi(V)\equiv \eps(J^{-1}(\psi))\pmod 2$.

Define a partial order on
$\CS_B$ by $V>V'$ if $\pi(V)>\pi(V')$. This gives an order on $Z$
and on $\CS_A$.
Then the decomposition
of the irreducible characters in $J^{-1}(Z)$ is given by a
unitriangular matrix.

\subsection{Stable equivalences for $\ell\times \ell$}
\label{se:stablelxl}
\subsubsection{Construction of a complex of bimodules}
\label{se:constructionglue}
We recall the construction of \cite[\S 6.2]{Rou3}.
Let $G$ be a finite group, $\ell$ a prime and $P$ a Sylow
$\ell$-subgroup of $G$. We assume in this subsection that
$K$ contains all $|G|$-th roots of
unity. Let $H=N_G(P)$. 
We assume that $P$ is elementary abelian of order $\ell^2$
and $G$ is not $\ell$-nilpotent.

\smallskip
Let $Q$ be a subgroup of $P$ of order $\ell$.
Let $\bar{N}_H(Q)$ be a complement to $Q$ in $N_H(Q)$, so that 
$N_H(Q)=Q\rtimes \bar{N}_H(Q)$. Let $\bar{C}_H(Q)=C_H(Q)\cap \bar{N}_H(Q)\lhd
N_H(Q)$.
Let $\bar{N}_G(Q)$ be a complement to $Q$ in $N_G(Q)$ containing
$\bar{N}_H(Q)$.  Let $\bar{C}_G(Q)=C_G(Q)\cap \bar{N}_G(Q)\lhd N_G(Q)$.
We have $\bar{N}_G(Q)=\bar{C}_G(Q)\bar{N}_H(Q)$.

\medskip
Let $d$ be the distance from the edge corresponding to $k$
to the exceptional vertex in the Brauer tree of $kB_0\bigl(\bar{C}_G(Q)\bigr)$.
Let $\CE$ be the set of simple modules
(up to isomorphism) of $kB_0\bigl(\bar{C}_G(Q)\bigr)$
whose distance
to the exceptional vertex is $d+1\pmod 2$; hence,
$k{\not\in}\CE$.
Define an injection $\gamma:\CE\hookrightarrow\CS_{kB_0(\bar{C}_H(Q))}$:
$\gamma(S)$ is the unique simple
$k\bar{C}_H(Q)$-module such that
$\Hom_{k\bar{C}_H(Q)\mstab}\bigl(\Res_{\bar{C}_H(Q)}^{\bar{C}_G(Q)}S,
\gamma(S)\bigr)\not=0$.
The set $\CE$ and the map $\gamma$ are $\bar{N}_H(Q)$-stable.

\smallskip
Let $N_\Delta=\bigl(\bar{C}_H(Q)\times\bar{C}_G(Q)^\opp\bigr)\Delta \bar{N}_H(Q)$.
We have a decomposition of
$\BF_\ell N_\Delta$-modules 

$$b_0\bigl(\bar{C}_H(Q)\bigr)\BF_\ell k\bar{C}_G(Q)b_0\bigl(\bar{C}_G(Q)\bigr)=
M_Q\oplus P,$$
where
$P$ is projective and $M_Q$ restricts to an indecomposable 
$\bigl(\BF_\ell B_0\bigl(\bar{C}_H(Q)\bigr)\otimes 
\BF_\ell B_0\bigl(\bar{C}_G(Q)\bigr)^\opp\bigr)$-module
inducing a stable equivalence.
A projective cover of $kM_Q$ is of the form
$$\bigoplus_{S\in\CS_{kB_0(\bar{C}_H(Q))}}P_{\gamma(S)}\otimes P_S^*\to kM_Q$$
The map may be chosen so that its restriction to
$\bigoplus_{S\in\CE}P_{\gamma(S)}\otimes P_S^*$ is defined over $\BF_\ell$
and we obtain a complex
of $\BF_\ell N_\Delta$-modules
$$X=(0\to U_Q\xrightarrow{a} M_Q\to 0)$$
with $kU_Q=\bigoplus_{S\in\CE}P_{\gamma(S)}\otimes P_S^*$.
The restriction of $X$
to $\BF_\ell B_0\bigl(\bar{C}_H(Q)\bigr)\otimes
\BF_\ell B_0\bigl(\bar{C}_G(Q)\bigr)^\opp$ is a
Rickard complex. This is the complex
$C(M_Q,\CE)$ defined in a more general setting in \S \ref{se:Okuyama}.
We put $T_Q=U_Q\oplus P$, $f=a+\id$, and
$$D=
\bigl(0\to T_Q \xrightarrow{f} b_0\bigl(\bar{C}_H(Q)\bigr)\BF_\ell
\bar{C}_G(Q)b_0\bigl(\bar{C}_G(Q)\bigr)\to 0\bigr),$$
a complex of $\BF_\ell N_\Delta$-modules homotopy
equivalent to $X$.

\medskip
Define
$$T'_Q=
\Res_{N_{H\times G^\opp}(\Delta Q)}^{N_{H\times G^\opp}(\Delta Q)
/\Delta Q}\Ind_{N_\Delta}
^{N_{H\times G^\opp}(\Delta Q)
/\Delta Q}T_Q.$$
We have $T'_Q=\BF_\ell Q\otimes T_Q$:
the action of $Q^\en$ is the canonical action
on $\BF_\ell Q$, the action of $\bar{C}_H(Q)\times\bar{C}_G(Q)^\opp$ comes
from the action on $T_Q$ and the action of $\Delta N_H(Q)$ comes from
the tensor product of the actions on $\BF_\ell Q$ and $T_Q$.

The map $f$ provides by induction a morphism of
$\BF_\ell N_{H\times G^\opp}(\Delta Q)$-modules
$$f':T'_Q \to b_0\bigl(C_H(Q)\bigr)\BF_\ell C_G(Q)b_0\bigl(C_G(Q)\bigr)$$
 and the associated complex
gives a Rickard complex by restriction to
$\BF_\ell B_0\bigl(C_H(Q)\bigr)\otimes \BF_\ell B_0\bigl(C_G(Q)\bigr)^\opp$.

Consider finally the morphism of $\BF_\ell (H\times G^\opp)$-modules
$$g_Q:b_0(H\times G)
\Ind^{H\times G^\opp}_{N_{H\times G^\opp}(\Delta Q)}
T'_Q \to b_0(H)\BF_\ell Gb_0(G)$$
deduced from $f'$ by adjunction.
Then,
$$C=0\to \bigoplus_Q b_0(H\times G)
\Ind^{H\times G^\opp}_{N_{H\times G^\opp}(\Delta Q)}
T'_Q \xrightarrow{\sum g_Q} b_0(H)\BF_\ell Gb_0(G)\to 0$$
induces a stable equivalence between the principal blocks of
$\BF_\ell G$ and $\BF_\ell H$ \cite[Theorem 6.3]{Rou3}. Here, $Q$ runs over
subgroups of $P$ of order $\ell$ up to $H$-conjugacy.

\smallskip
Fix a decomposition $b_0(H)\BF_\ell Gb_0(G)=M\oplus R$, where $M$ is an
indecomposable $\BF_\ell(G\times H^\circ)$-module with vertex $\Delta P$ and
$R$ is a direct sum
of indecomposable modules with vertices strictly contained in $\Delta P$.
Then, $\Br_{\Delta Q}(M)\simeq M_Q$
for $Q$ a subgroup of $P$ of order $\ell$ (see \cite[Theorem 3.2]{Br1} for example).
We proceed with the construction above with $T_Q$ replaced by $U_Q$:
define $U'_Q=
\Res_{N_{H\times G^\opp}(\Delta Q)}^{N_{H\times G^\opp}(\Delta Q)
/\Delta Q}\Ind_{N_\Delta}
^{N_{H\times G^\opp}(\Delta Q) /\Delta Q}U_Q$. We obtain a complex
homotopy equivalent to $C$:
$$0\to \bigoplus_Q b_0(H\times G)
\Ind^{H\times G^\opp}_{N_{H\times G^\opp}(\Delta Q)}
U'_Q \to M\to 0.$$

\begin{rem}
\label{re:Moritacyclic}
Note that $U'_Q=0$ if
$N_H(Q)=C_H(Q)$, or if the Brauer tree of $kB_0\bigl(\bar{C}_G(Q)\bigr)$ is a star
with exceptional vertex in the centre (this happens for example if
$\ell=3$). In that case, $M_Q$ induces
a splendid Morita equivalence. If this holds for all subgroups $Q$ of $P$ of order $\ell$, then $M$ induces a splendid stable equivalence.
\end{rem}

\subsubsection{Images}
Let $L$ be a $kB_0(G)$-module. Let $Q$ be a subgroup
of $P$ of order $\ell$. We keep the notation of \S \ref{se:constructionglue}.
Let $\Gamma=N_\Delta \times N_G(Q)$.
We have an embedding $\alpha:N_H(Q) \hookrightarrow\Gamma,\ g\mapsto
\bigl((\bar{g},\bar{g}^{-1}),g\bigr)$, where $\bar{g}\in \bar{N}_H(Q)$
is the image of $g$.
The action of $\Gamma$ on
$T_Q\otimes \Res^G_{N_G(Q)}L$ restricts via $\alpha$ to an action of
$N_H(Q)$ on $T_Q\otimes_{\BF_\ell\bar{C}_G(Q)} L$, and $f$
induces
a morphism of $kN_H(Q)$-modules
$T_Q\otimes_{\BF_\ell\bar{C}_G(Q)}L\to \Res_{N_H(Q)}^G L$. By adjunction, this
provides a morphism of $kH$-modules 
\[h_Q:b_0(H)\Ind_{N_H(Q)}^H (T_Q\otimes_{\BF_\ell\bar{C}_G(Q)}L)\to
b_0(H)\Res_H^G L.\]
Thus
\[C\otimes_{\BF_\ell G}L\simeq 
\bigl(0\to \bigoplus_Q b_0(H)\Ind_{N_H(Q)}^H 
(T_Q\otimes_{\BF_\ell \bar{C}_G(Q)}L)
\xrightarrow{\sum h_Q} b_0(H)\Res_H^G L \to 0\bigr).\]

\subsubsection{Self-derived equivalences for $k\GS_3$}
\label{se:selfderived}
Let $G=\GS_3$, $\ell=3$ and $A=kG$.
Let $P_1$ be the projective cover of the trivial
$A$-module $S_1$ and $P_2$ the projective cover of the non-trivial simple
$A$-module $S_2$. A projective cover of $A$, viewed as an $A^{\en}$-module,
is 
$$P_1\otimes P_1^*\oplus P_2\otimes P_2^*\xrightarrow{b} A.$$
We set $C=C(A,\{S_2\})=
0\to P_2\otimes P_2^*\xrightarrow{b} A\to 0$ (cf.\ \S \ref{se:Okuyama}).
This is a Rickard complex. Given $n\ge 0$, the equivalence induced by
the Rickard complex $C^{\otimes_A n}$
is perverse relative to the function
$\per$ given by $\per(1)=0$ and $\per(2)=n$.

\smallskip
We have $C\otimes_A P_2\simeq P_2[1]$ in $\Ho(A)$.
Assume that $n>0$.
We deduce that 
$$\Hom_{\Ho(A)}(P_2,\Res_A C^{\otimes_An}[i])=
\Hom_{\Ho(A^\opp)}(P_2^*,\Res_{A^\opp} C^{\otimes_An}[i])=0
\text{ for }i\not=-n.$$
Thus the composition factors of $H^i(C^{\otimes_An})$ are isomorphic to
$S_1\otimes S_1^*$ for $i\not=-n$ and we deduce that there is an
isomorphism in $\Ho(A^\en)$:
$$C^{\otimes_A n}\simeq
0\to P_2\otimes P_2^*\to\cdots\to P_2\otimes P_2^*\xrightarrow{b} A\to 0$$
where the non-zero terms are in degrees $-n,\ldots,0$ and the
complex on the right is the unique indecomposable complex with the given terms
and with homology isomorphic to
$S_1\otimes S_1^*$ in degrees $-n+1,\ldots,0$.

\smallskip
Let $F$ be a standard self-equivalence of $D^b(A)$ such that $F(S_1)\simeq S_1$.
Then, $F$ is a perverse equivalence for a perversity function
$\per$ with $\per(1)=0$. Consequently, $F$ (or $F^{-1}$) is given by the
Rickard complex $C^{\otimes_A n}$ for some $n\ge 0$.

\begin{rem}
The group $\TrPic(A)$ of isomorphism classes of
standard self-derived equivalences of $A$ has been determined in
\cite[\S 4]{RouZi}. The result above on the subgroup of those self-equivalences
that fix the trivial representation can be deduced easily.
\end{rem}

\subsubsection{Local twists for $3\times 3$}
\label{se:localtwists}
The construction of stable equivalences in \S \ref{se:constructionglue}
builds on the
``simplest'' possible local derived equivalences. Assume that $\ell=3$; then we have $U_Q=0$ for all subgroups $Q$ of $P$ of order $\ell$ (see Remark \ref{re:Moritacyclic}).
We have $b_0(H)\BF_\ell Gb_0(G)=M\oplus R$, where $M$ is an indecomposable 
$\bigl(\BF_\ell B_0(G)\otimes \BF_\ell B_0(H)\bigr)$-module inducing a stable
equivalence and
the indecomposable summands of $R$ have vertex of order at most $\ell$
(Remark \ref{re:Moritacyclic}, see also \cite[Lemma 3.8]{koshitanikunugi2002}).

We explain here how to modify the stable equivalence induced by $M$ using a
self-stable equivalence of $B_0(H)$.
Let $\CT$ be a set of representatives of $H$-conjugacy classes of
subgroups $Q$ of $P$ of order $\ell$ such that $|C_H(Q)/C_H(P)|=2$.
Fix a map $\eta:\CT\to\BZ_{\ge 0}$.

\smallskip
Let $Q\in\CT$. There is an $\ell'$-subgroup $E'_Q$ of
$\bar{N}_H(Q)$ such that
$\bar{N}_H(Q)=\bigl(P\cap\bar{C}_H(Q)\bigr)\rtimes E'_Q$. Let $E_Q=
E'_Q\cap \bar{C}_H(Q)$. We have
$\bar{C}_H(Q)=\bigl(P\cap\bar{C}_H(Q)\bigr)\rtimes E_Q$.
Let $V_Q$
be a non-trivial simple $\BF_\ell B_0\bigl(\bar{N}_H(Q)\bigr)$-module with non-trivial
restriction to $E_Q$. Note that $\Res_{E_Q}V_Q$ is uniquely defined: this
is the $1$-dimensional non-trivial 
$\BF_\ell\bigl(E_Q/O_{\ell'}\bigl(\bar{C}_H(Q)\bigr)\bigr)$-module.

\smallskip
Let $N'_\Delta=\bar{C}_H(Q)^\en \Delta \bar{N}_H(Q)$.
The construction of \S \ref{se:selfderived} provides an indecomposable 
complex of $\BF_\ell N'_\Delta$-modules
$$X_Q=0\to P_{V_Q}\otimes P_{V_Q}^*\to\cdots\to
P_{V_Q}\otimes P_{V_Q}^*\to \BF_\ell B_0\bigl(\bar{C}_H(Q)\bigr)\to 0,$$
where the non-zero terms are in degrees $-\eta(Q),\ldots,0$ and
whose restriction to
$\BF_\ell B_0\bigl(\bar{C}_H(Q)\bigr)^\en$ is a Rickard complex.

\smallskip
We proceed now as in \S \ref{se:constructionglue} to glue the complexes
$X_Q$.
We have $N_{H^\en}(\Delta Q)=Q^\en\times N'_\Delta$.
Let 
$$U'_Q=b_0(H^\en)\Ind_{N_{H^\en}(\Delta Q)}^{H^\en}
\bigl(\BF_\ell Q\otimes (P_{V_Q}\otimes P_{V_Q}^*)\bigr).$$
We have a bounded complex of
$\BF_\ell B_0(H^\en)$-modules
$$C'(\eta)=\bigl(\cdots\to \bigoplus_{Q\in\CT,\eta(Q)\ge 3} U'_Q
\xrightarrow{\sum h_Q}
\bigoplus_{Q\in\CT,\eta(Q)\ge 2} U'_Q
\xrightarrow{\sum h_Q}
\bigoplus_{Q\in\CT,\eta(Q)\ge 1} U'_Q\to \BF_\ell B_0(H)\to 0\bigr)$$
which induces a self-stable equivalence of $\BF_\ell B_0(H)$, for certain 
$h_Q\in\End_{\BF_\ell(H^\en)}(U'_Q)$. We
put $C(\eta)=C'(\eta)\otimes_{\BF_\ell H}M$: this complex of
$\bigl(\BF_\ell B_0(H)\otimes \BF_\ell B_0(G)^\opp\bigr)$-modules induces a
stable equivalence.

\medskip
Let $L$ be a simple $\BF_\ell B_0(G)$-module, and let
$L'$ be the unique indecomposable direct summand of 
$b_0(H)\Res_H^GL$ with vertex $P$.
Given $Q\in\CT$, let
$$L'_Q=\Res_{\Delta(Q\rtimes E'_Q)}
\Bigl(\Res_{Q\rtimes E'_Q}^{(Q\rtimes E'_Q)/Q}V_Q\otimes 
\Res_{Q\rtimes E'_Q}^{(Q\rtimes E'_Q)/E_Q}
\bigl(\Hom_{\BF_\ell E_Q}(\Res_{Q\rtimes E'_Q}^{(Q\rtimes E'_Q)/Q}V_Q,
L')\bigr)\Bigr).$$
Thus, $L'_Q=V_Q\otimes \Hom_{\BF_\ell E_Q}(V_Q,L')$, the action of
$x\in Q$ is given by $v\otimes f\mapsto v\otimes xf$ and the action
of $y\in E'_Q$ is $v\otimes f\mapsto yv\otimes yfy^{-1}$, for
$v\in V_Q$ and $f\in \Hom_{\BF_\ell E_Q}(V_Q,L')$. We have a decomposition
$\Res_{Q\rtimes E'_Q}(L')=L_Q^1\oplus L_Q^2$, where $L_Q^1$ is the
maximal direct summand such that
$\Res_{E_Q}(L_Q^1)$ is a multiple of $V_Q$. Then $L'_Q\simeq
V_Q\otimes L_Q^1$.

Consider a decomposition $L'_Q=L''_Q\oplus P$ as $\BF_\ell(Q\rtimes E'_Q)$-modules,
where $P$ is projective. Let $L_Q= b_0(H)\Ind_{Q\rtimes E'_Q}^HL''_Q$.
We have an isomorphism
\begin{align*}
C(\eta)\otimes_{\BF_\ell G}L\simeq& 
\bigl(\cdots\to \bigoplus_{Q\in\CT,\eta(Q)\ge 3} L_Q\xrightarrow{\sum \Ind(s_Q)}
\bigoplus_{Q\in\CT,\eta(Q)\ge 2} L_Q\xrightarrow{\sum \Ind(s_Q)}
\bigoplus_{Q\in\CT,\eta(Q)\ge 1} L_Q \to
L' \to 0\bigr)\\
&\oplus \text{ bounded complex of projective modules},
\end{align*}
where $s_Q\in\End_{\BF_\ell N_H(Q)}(\Ind^{N_H(Q)}_{Q\rtimes E'_Q}L''_Q)$ is
non-zero, but not invertible.

\begin{rem}
\label{re:onedimensional}
In the examples studied in Section \ref{se:defect33}, the maps in the complexes are
uniquely determined up to scalars, thanks to the fact that the following
conditions hold:
\begin{itemize}
\item when $\eta(Q)>0$, $\dim\Hom_{kH}(L_Q,L')=1$
\item when $\eta(Q)>1$, $\dim L''_Q=1$.
\end{itemize}
\end{rem}

Let $\tG$ be a finite group containing $G$ as a normal subgroup of
$\ell'$-index.
If the function $\eta$ is invariant under the 
action of $\tG$ on conjugacy classes of
subgroups of order $3$, then $C(\eta)$ extends to a
complex of $k\bigl((H\times G^\opp)\Delta N_{\tG}(P)\bigr)$-modules.

\begin{rem}
\label{re:fixedpointfree}
There are six conjugacy classes of $3'$-subgroups $E$ of $\GL_2(\BF_3)$ such
that $(\BF_3)^E=0$. They are determined by their isomorphism type:
$\Cyc2$, $\Cyc2^2$, $\Cyc4$, $D_8$, $Q_8$ and $SD_{16}$. Assume that all non-trivial elements of $E$ act fixed-point
freely on $(\BF_3)^2-\{0\}$: this corresponds to the types $\Cyc2$, $\Cyc4$ and
$Q_8$. Let $A=kP\rtimes E$, where $P=\Cyc3^2$. Let $M'$ be an
$A^\en$-module inducing a self-stable equivalence. By
\cite[Theorem 3.2]{CarRou} there is an integer $n$ such that
$\Omega^n_{A^\en}(M')$ induces a self-Morita equivalence. Let
$G$ be a finite group with Sylow $3$-subgroup $P$ and
with $N_G(P)/C_G(P)=E$. Let $\tG$ be a finite group containing $G$ as
a normal subgroup of $3'$-index. Let $C$ be a two-sided tilting complex
for $(A,kB_0(G))$. Let $D=\Hom_{kB_0(G)^\opp}^\bullet(C,M)$. This induces a
self-stable equivalence of $A$, so there is an integer $n$ and an
invertible $A^\en$-module $M''$ such that $\Omega^n(M'')\otimes_A D$
is isomorphic to $A$ in $A^\en\mstab$. Let
$C'=\Hom_A^\bullet(M''[-n],C)$: this is a two-sided tilting complex
for $\bigl(A,kB_0(G)\bigr)$ and it is isomorphic in the stable category to $M$. So,
up to shift and Morita equivalence, a two-sided tilting complex can be
assumed to lift a given stable equivalence.

Note that in such a finite simple group $G$, the automizer $E$ will be of type
$\Cyc4$, $D_8$, $Q_8$ or $SD_{16}$.
\end{rem}

\subsection{Lie type}
\label{se:Lie}
\subsubsection{Deligne--Lusztig varieties}
\label{se:DeligneLusztig}
For finite groups of Lie type in non-describing characteristic, Brou\'e
conjectured that a solution of the abelian defect conjecture will arise
from the complex of cohomology of a Deligne--Lusztig variety
\cite[\S 6]{Br2}. That
is known in very few cases, and in those cases, defect groups
are cyclic \cite{Rou4,BoRou,Du}. We recall now the setting and
constructions of \cite{BrMi}.

\medskip
Let $\BG$ be a reductive connected algebraic group 
endowed with an endomorphism $F$ such that
there is $\delta\in\BZ_{>0}$ with the property that $F^\delta$ is
a Frobenius endomorphism relative to an $\BF_{q^\delta}$ structure on
$\BG$. Here, $q\in\BR_{>0}$ and we assume there is a choice $q\in K$.
Let $G=\BG^F$. Let $W$ be the Weyl group of
$\BG$ and $B^+$ be the braid monoid of $W$. We denote by
$\phi$ the automorphisms of $W$ and $B^+$ induced by $F$. 
Let $w\mapsto {\bf w}:W\to B^+$ be the length-preserving lift
of the canonical map $B^+\to W$.
Let $\pi={\bf w}_0^2$, where $w_0$ is the longest element of $W$.

\smallskip
Let $\ell$ be a prime number that does not divide $q^{\delta}$, and let $P$ be a Sylow $\ell$-subgroup of $G$. We assume that $P$ is abelian
and $C_{\BG}(P)$ is a torus. Let $d$ be the multiplicative order of
$q$ in $k^\times$: this is a $\phi$-regular number for $W$. There
exists $b_d\in B^+$ such that $(b_d\phi)^d=\pi \phi^d$. Let
$B_d^+=C_{B^+}(b_d\phi)$, and
let $Y(b_d)$ be the corresponding Deligne--Lusztig variety. The
complex $C=R\Gamma\bigl(Y(b_d),\CO\bigr)b_0(G)$ has an action of
$C_G(P)\times G^\opp$. It is conjectured that
\begin{itemize}
\item the action extends (up to
homotopy) to an action of $\bigl(C_G(P)\rtimes B_d^+\bigr)\times G^\opp$, and
\item the canonical map $\CO\bigl(C_G(P)\rtimes B_d^+\bigr)\to
\End_{D(\CO G^\opp)}^\bullet(C)$
is a quasi-isomorphism of algebras, with image isomorphic to
$\CO B_0\bigr(N_G(P)\bigl)$.
\end{itemize}

It is conjectured further \cite{ChRou2} that these equivalences are
perverse. Let us explain how the maps $\per_Q$ of \S \ref{se:splendidperverse}
are encoded in the geometry.

Given $\chi$ a unipotent character of $G$, let $A_\chi$ denote the degree of its
generic degree.
Conjecturally, if $\BG$ has connected centre and $\ell$ is good,
the unipotent characters in $B_0(G)$ form a basic set and the
decomposition matrix of $B_0(G)$ is unitriangular with respect to that
basic set, for the order given by the function $A$ (cf
\cite[Conjecture 3.4]{GeHi} and \cite[Conjecture 1.3]{Ge}).
This
gives a bijection between $\CS_{B_0(G)}$ and the set of unipotent characters
in $B_0(G)$. The function $\per_1$ should be given by the unique
degree of cohomology of $Y(b_d,K)$
where the corresponding unipotent character occurs.

\smallskip
The complex $C$ has a canonical representative $\tR\Gamma\bigl(Y(b_d),\CO\bigr)$
in $\Ho\bigl(\CO\bigl(C_G(P)\times G^\opp\bigr)\bigr)$ that is splendid \cite{Ri4,Rou4} and
given $Q$ a subgroup of $P$, we have
$k\Br_{\Delta Q}\bigl(\tR\Gamma\bigr(Y(b_d)\bigl)\bigl)\simeq 
\tR\Gamma\bigl(Y(b_d)^{\Delta Q},k\bigr)$.
Hence, the local derived equivalences are controlled by 
Deligne--Lusztig varieties associated with Levi subgroups of $\BG$ and
this gives a corresponding description for the functions $\per_Q$.

\medskip
There is a conjecture for the unipotent part of the
cohomology of Deligne--Lusztig varieties associated with powers
of ${\bf w}_0$. There is no conjecture yet for other roots of powers of $\pi$.
For applications to Brou\'e's conjecture, the conjecture below covers the
cases $\ell\mid(q\pm 1)$.

\begin{conj}[{\protect \cite[\S 3.3.23]{DiMiRou}}]
Let $\chi$ be a unipotent character of $G$.
Given $n,i\ge 0$, if
$[H^i(Y({\bf w}_0^n),K):\chi]\not=0$, then $i=n A_\chi$.
\end{conj}

Assume that $\ell=3$.
We consider now all groups $(\BG,F)$ such that $G$ is semi-simple and
$P\simeq (\Cyc{3})^2$. For each such group, and for each conjugacy class
of subgroups $Q$ of order $3$, we provide the semi-simple type of
$(C_{\BG}(Q),F)$ and we give an element $b$ in the braid monoid of
$C_{\BG}(Q)$ such that $Y_G(b_d)^{\Delta Q}=Y_{C_G(Q)}(b)$. We also
provide in some cases another semi-simple group and an element in the braid
monoid such that the Deligne--Lusztig variety can be identified equivariantly
with $Y_{C_G(Q)}(b)$ \cite[\S 1.18]{Lu}.

\begin{itemize}
\item $B_2$, $e=1$. 
\begin{itemize}
\item $A_1$, ${\bf s}^2$
\item $A_1$, ${\bf s}^2$
\end{itemize}
\item ${^2A_3}$, $e=1$
\begin{itemize}
\item $A_1$, ${\bf s}^2$
\item $(A_1\times A_1,(x,y)\mapsto (y,F(x))$, $({\bf s}^2,{\bf s}^2)$. This
is equivalent to $A_1,{\bf s}^4$.
\end{itemize}
\item ${^2A_4}$, $e=1$
\begin{itemize}
\item ${^2A}_2$, $({\bf st})^3$
\item $(A_1\times A_1,(x,y)\mapsto (y,F(x))$, $({\bf s}^2,{\bf s}^2)$. This
is equivalent to $A_1,{\bf s}^4$.
\end{itemize}
\item $A_3$, $e=2$.
\begin{itemize}
\item $A_1$, ${\bf s}$
\item $(A_1\times A_1,(x,y)\mapsto (y,F(x))$, $({\bf s},{\bf s})$.
This is equivalent to $A_1$, ${\bf s}^2$.
\end{itemize}
\item $A_4$, $e=2$.
\begin{itemize}
\item $A_2$, ${\bf sts}$
\item $(A_1\times A_1,(x,y)\mapsto (y,F(x))$, $({\bf s},{\bf s})$.
This is equivalent to $A_1$, ${\bf s}^2$.
\end{itemize}
\item $B_2$, $e=2$
\begin{itemize}
\item $A_1$, ${\bf s}$
\item $A_1$, ${\bf s}$
\end{itemize}
\end{itemize}

\medskip
Note that there are finite simple groups of Lie type with elementary abelian Sylow
$3$-subgroups of order $9$ that do not arise as
rational points of a reductive connected algebraic group, but as a quotient
of such a group. There are two classes of such groups:

\begin{itemize}
\item $G=\PSL_3(q)$ with $q\equiv 4,7\pmod 9$;
\item $G=\PSU_3(q)$ with $q\equiv 2,5\pmod 9$.
\end{itemize}

\subsubsection{Morita equivalences}
Let $G$ be a finite group and $\ell$ a prime.
Let $T$ be an $\ell$-nilpotent subgroup of $G$ with Sylow
$\ell$-subgroup $P$. Let $W=N_G(T)/T$.

We assume that
\begin{itemize}
\item $C_T(P)=C_G(P)$,
\item there is an $\ell'$-subgroup $U$ of $G$
such that $T\subset N_G(U)$, $T\cap U=1$ and $G=UTU$, and
\item $W$ is an $\ell'$-group.
\end{itemize}

Let us recall a result of Puig \cite[Corollaire 3.6]{Pu}.

\begin{thm}
\label{th:Puig}
The bimodule $e_U\BZ_\ell Gb_0(G)$ induces a Morita equivalence between
$B_0(G)$ and $B_0\bigl(N_G(P)\bigr)$, where $e_U=\frac{1}{|U|}\sum_{x\in U}x$.
\end{thm}

Let $E$ be a group of automorphisms of $G$ that stabilizes $U$ and $P$.
Then, the $\bigl(B_0\bigl(N_G(P)\bigr)\otimes B_0(G)^\opp\bigr)$-module
$e_U\BZ_\ell Gb_0(G)$ extends to a
$\bigl(\bigl(B_0(N_G(P))\otimes B_0(G)^\opp\bigr)\rtimes E\bigr)$-module.

\smallskip
The main example is the following (cf.\ \cite[Theorem 23.12]{CabEn}).
We take $\BG$, $F$, and so on as in \S \ref{se:DeligneLusztig}, and we assume that $\delta=1$.
Let $\BT\subset\BB$ be an $F$-stable maximal torus contained in an
$F$-stable Borel subgroup of $\BG$
and let $\BU$ be the unipotent radical of $\BB$.
Let $U=\BU^F$ and $T=\BT^F$.
The assumptions above are satisfied when
$\ell\mid(q-1)$ and $\ell\nmid |W^F|$. We have
$N_G(P)=N_{\BG}(\BT)^F$.

\begin{rem}
\label{re:changeq}
Consider the same setting for
$\BG'$ another reductive group, defined over $\BF_{q'}$. If
the finite groups $N_{\BG}(\BT)^F/O_{\ell'}\bigl(N_{\BG}(\BT)^F\bigr)$ and
$N_{\BG'}(\BT')^{F'}/O_{\ell'}\bigl(N_{\BG'}(\BT')^{F'}\bigr)$ are isomorphic,
then Theorem \ref{th:Puig} provides a splendid Morita equivalence
between $B_0(G)$ and $B_0(G')$.
\end{rem}

\smallskip
Let us be more specific for our needs.
The condition above is satisfied when 
\begin{itemize}
\item $G=\PSU_n(q)$,
$\ell\mid(q-1)$ and $2\neq \ell>n/2$.
\item
$G=\PSp_4(q)$ and $\ell\mid(q-1)$, $\ell\neq2$.
\end{itemize}

Brou\'e's conjecture predicts the existence of another derived equivalence
(not a Morita equivalence),
provided by the Deligne--Lusztig variety associated with the element
$\pi$ of the braid group. Note that such an equivalence would arise from
an action of $G\times (P\rtimes B^+)^\opp$ on a geometric object,
while in the Harish-Chandra equivalence above, the action of $N_G(P)$
on $\BZ_\ell(G/U)$ doesn't arise from a monoid action on $G/U$.

\section{Automorphisms}
\label{se:automorphismes}
\subsection{Stability of equivalences}
\label{se:stability}
Extensions of equivalences and reductions to finite simple groups have been
considered in various particular situations: isotypies \cite{FoHa},
(splendid) Rickard and derived equivalences \cite{Ma1}. We introduce here
a framework that handles various types of equivalences.

\subsubsection{Extensions of equivalences}
\label{se:stabilityequiv}
Let $R$ be a commutative $\BZ_\ell$-algebra. We consider data $\CC$ consisting,
for every finite group
$G$, of a full subcategory $\CC(G)$ of the category of bounded complexes
of $R$-projective finitely generated $RG$-modules. We assume that $\CC(G)$ is
closed under taking direct sums and direct summands and that the
following holds:

\begin{itemize}
\item[(S1)] given $H\le G$ of $\ell'$-index and given
$X\in\CC(H)$, then $\Ind_H^G(X)\in\CC(G)$.
\end{itemize}

A consequence of the assumptions is that given $X\in\Comp^b(RG)$ and given
$H\le G$ of $\ell'$-index, if
$\Res_H(X)\in\CC(H)$, then $X\in\CC(G)$, since $X$ is a direct summand
of $\Ind^G\Res_H X$.

\begin{defi}
Let $G$ and $H$ be two finite groups. We say that
$X\in\Comp^b\bigl(RB_0(G)\otimes RB_0(H)^\opp\bigr)$
\emph{induces a $\CC$-equivalence}
between the principal blocks of $G$ and $H$ if
\begin{itemize}
\item the canonical map $RB_0(G)\to\End^\bullet_{RH^\opp}(X)$ is a
split injection with cokernel in $\CC(G^\en)$, and
\item the canonical map
$RB_0(H)\to\End^\bullet_{RG}(X)$ is a split injection with cokernel in 
$\CC(H^\en)$.
\end{itemize}
\end{defi}

\smallskip
Given $G\lhd\tilde{G}$ and $H\lhd\tilde{H}$ with $\tilde{H}/H=\tilde{G}/G=E$
an $\ell'$-group,
we put $\tilde{\Delta}(G,H)=\{(x,y)\in\tilde{G}\times\tilde{H}^\opp \mid
(xG,yH)\in\Delta E\}$.

\begin{defi}
We say that $X\in\Comp^b\bigl(R\tilde{\Delta}(G,H)\bigr)$ \emph{induces an $E$-equivariant $\CC$-equivalence} between the principal blocks of $G$ and $H$ if
$\Res_{G\times H^\opp}(X)$ induces a $\CC$-equivalence
between the principal blocks of $G$ and $H$ and
$b_0(\tG)\Ind^{\tG\times\tH^\opp}X=b_0(\tH)\Ind^{\tG\times\tH^\opp}X$.
\end{defi}

\begin{lemma}
\label{le:extensionequiv}
Let $G\lhd\tilde{G}$ be finite groups with $\ell\nmid[\tilde{G}:G]$. Let
$H\lhd\tilde{H}$ with $\tilde{H}/H=\tilde{G}/G$.
Let $X\in\Comp^b\bigl(R\tilde{\Delta}(G,H)\bigr)$ be a complex
inducing an
equivariant $\CC$-equivalence between the principal blocks of $G$ and $H$.

Then $b_0(\tG\times \tH)\Ind^{\tG\times\tH^\opp}X$
induces a $\CC$-equivalence between the
principal blocks of $\tilde{G}$ and $\tilde{H}$.
\end{lemma}

\begin{proof}
Let $X_1=\Res_{G\times H^\opp}(X)$ and
$X_2=\Ind^{\tilde{G}\times \tilde{H}^\opp}(X)$. We have canonical
isomorphisms (Mackey formula)
$$\Res_{G\times\tilde{H}^\opp}X_2\iso \Ind^{G\times\tilde{H}^\opp}X_1
\text{ and }
\Res_{\tilde{G}\times H^\opp}X_2\iso \Ind^{\tilde{G}\times H^\opp}X_1.$$

We have canonical isomorphisms in $\Comp^b(R(\tilde{G}\times G^\opp))$:
$$\Res_{\tilde{G}\times G^\opp}\End^\bullet_{R\tilde{H}^\opp}(X_2)\iso
\Hom^\bullet_{R\tilde{H}^\opp}(\Res_{G\times\tilde{H}^\opp}X_2,X_2)\iso
\Hom^\bullet_{R\tilde{H}^\opp}(\Ind^{G\times\tilde{H}^\opp}X_1,X_2)\iso$$
$$\Hom^\bullet_{RH^\opp}(X_1,\Res_{\tilde{G}\times H^\opp}X_2)\iso
\Hom^\bullet_{RH^\opp}(X_1,\Ind^{\tilde{G}\times H^\opp}X_1)\iso
R\tilde{G}\otimes_{RG}\End^\bullet_{RH^\opp}(X_1).$$
We deduce a commutative diagram in $\Comp^b(RG^\en)$:
$$\xymatrix{
RB_0(\tG)\otimes_{RB_0(G)}RB_0(G) \ar[rr]^-{1\otimes\can} 
\ar[d]_\sim^{\text{mult}} &&
RB_0(\tG)\otimes_{RB_0(G)}\End^\bullet_{RH^\opp}(X_1)\ar[d]^\sim \\
RB_0(\tG)\ar[rr]_-{\can}&& b_0(\tG)\End^\bullet_{R\tH^\opp}(X_2)
}$$
It follows that the canonical map
$RB_0(\tG)\to \End^\bullet_{R\tH^\opp}\bigl(b_0(\tG)X_2\bigr)$ is a split
injection with cokernel in $\CC(\tG^\en)$.

The other condition is checked by swapping the roles of $G$ and $H^\opp$.
\end{proof}

Using the notation of the proof of Lemma \ref{le:extensionequiv}, note
that we have a commutative diagram

$$\xymatrix{\Comp^b\bigl(RB_0(\tH)\bigr)\ar[rr]^{X_2\otimes_{R\tH}-}\ar[d]_{\Res} &&
\Comp^b\bigl(RB_0(\tG)\bigr)\ar[d]^{\Res} \\
\Comp^b\bigl(RB_0(H)\bigr)\ar[rr]_{X_1\otimes_{RH}-} && \Comp^b\bigl(RB_0(G)\bigr)
}$$

\begin{rem}
Consider $G\lhd\tilde{G}$ and $H\lhd\tilde{H}$ with $\tilde{H}/H=\tilde{G}/G=E$
an $\ell'$-group. Let $P$ be an $\ell$-Sylow subgroup of $G$ and $Q$ an
 $\ell$-Sylow subgroup of $H$.
If $C_{\tG}(P)\subset G$ and $C_{\tH}(Q)\subset H$, then given
$X\in\Comp^b\bigl(R\tilde{\Delta}(G,H)\bigr)$ whose restriction is in
$\Comp^b\bigl(RB_0(G\times H^\opp)\bigr)$, we have 
$X\in\Comp^b\bigl(RB_0\bigl(\tilde{\Delta}(G,H)\bigr)\bigr)$
\cite[Theorem 6.4.1]{Be}.
\end{rem}

We can even do a little better to extend equivalences.

\begin{lemma}
\label{le:extension}
Consider finite groups $G_1\lhd G_2\lhd \tG_2\le \tG_1$ and
$H_1\lhd H_2\lhd \tH_2\le \tH_1$ with $G_1\lhd\tG_1$, $H_1\lhd\tH_1$,
$\ell\nmid[\tG_1:G_1]$,
$\tG_1/G_1=\tH_1/H_1$, $\tG_2/G_1=\tH_2/H_1$ and $G_2/G_1=H_2/H_1$
(compatible with the inclusions $G_2/G_1\le \tG_2/G_1\le \tG_1/G_1$ and
$H_2/H_1\le \tH_2/H_1\le \tH_1/H_1$).
Let $E_i=\tG_i/G_i$.

Let $X\in\Comp^b\bigl(R\tilde{\Delta}(G_1,H_1)\bigr)$ be a complex
inducing an $E_1$-equivariant
$\CC$-equivalence between the principal blocks of $G_1$ and $H_1$.

Then $b_0(G_2\times H_2^\opp)
\Ind^{\tilde{\Delta}(G_2,H_2)}
\Res_{\tilde{\Delta}(G_1,H_1)\cap (\tilde{G}_2\times\tH_2^\opp)}(X)$
induces an $E_2$-equivariant $\CC$-equivalence between the
principal blocks of $G_2$ and $H_2$.
\end{lemma}

\begin{proof}
If $G_1=G_2$, then in this case the result is clear. If $\tG_1=\tG_2$, then
\[\Res_{G_2\times H_2^\opp}\Ind^{\tilde{\Delta}(G_2,H_2)}(X)\simeq \Ind^{G_2\times H_2^\opp}\Res_{\tDelta(G_1,H_1)\cap (G_2\times H_2^\opp)} X,\]
and the result follows from Lemma \ref{le:extensionequiv}. The general case follows from the two cases studied above.
\end{proof}

\begin{lemma}
\label{le:AlperinDade}
Let $G\lhd\tG\le\hat{G}$ with $G\lhd\hat{G}$ and $\ell\nmid[\hat{G}:G]$.
Let $P$ be a Sylow $\ell$-subgroup of $G$. Let $\tH=GC_{\tilde{G}}(P)$
and $\hat{H}=GC_{\hat{G}}(P)$.
Assume that $\hat{G}=\tilde{G}C_{\hat{G}}(P)$.

The $B_0(\tH\times\hat{H}^\opp)$-module
$B_0(\tH)\otimes_{\BZ_\ell G} B_0(\hat{H})$ extends
to a $B_0\bigl((\tH\times \hat{H}^\opp)\Delta(\tG)\bigr)$-module $M$,
where $h\in \tG$ sends $x\otimes y$ to $hxh^{-1}\otimes hyh^{-1}$. The module $M$ induces a splendid $(\tG/\tH)$-equivariant Morita equivalence between $B_0(\tH)$ and $B_0(\hat{H})$, and the module
$\Ind^{\tG\times\hat{G}^\opp}(M)$ provides
an isomorphism of algebras
$$B_0(\tilde{G})\iso B_0(\hat{G}),\ x\mapsto b_0(\hat{G})x.$$
\end{lemma}

\begin{proof}
The Alperin--Dade theorem (\cite[Theorem 2]{Alp}, \cite{Da})
shows that there are isomorphisms
$$B_0(G)\iso B_0(\tH),\ a\mapsto b_0(\tH)a \text{ and }B_0(G)\iso B_0(\hat{H}),\ a\mapsto b_0(\hat{H})a.$$
We obtain an isomorphism
$$B_0(G)\iso B_0(\tH)\otimes_{\BZ_\ell G} B_0(\hat{H}),\ a\mapsto b_0(\tH)a
\otimes b_0(\hat{H})$$
compatible with the $\Delta(\tG)$-action described in the lemma,
and this provides $M\simeq B_0(G)$ with a structure of a
$B_0\bigl((\tH\times \hat{H}^\opp)\Delta(\tG)\bigr)$-module.
Note that $\Res_{\tH\times\hat{H}^\opp}(M)$ induces a Morita
equivalence that sends $B_0(\tilde{H})$ to $B_0(\hat{H})$.
We have $\Ind_{\tH}^{\tG}\bigl(B_0(\tH)\bigr)=B_0(\tG)$ and
$\Ind_{\hat{H}}^{\hat{G}}\bigl(B_0(\hat{H})\bigr)=B_0(\hat{G})$ 
\cite[Theorem 6.4.1 (v)]{Be}, hence
the Morita equivalence induced by $\Ind^{\tG\times\hat{G}^\opp}(M)$
sends $B_0(\tilde{G})$ to $B_0(\hat{G})$ (cf.\ the proof of
Lemma \ref{le:extensionequiv}) and it gives rise to the isomorphism
of algebras described in the lemma.
\end{proof}

We assume now that the data $\CC$ satisfy the following additional
assumption:
\begin{itemize}
\item[(S2)] given $G$, $G'$ two finite groups, given
$X\in\CC(G)$ and $Y\in\Comp^b(RG')$
with $Y^i$ projective over $R$ for all $i$, 
then $X\otimes_R Y\in\CC(G\times G')$.
\end{itemize}

\begin{lemma}
\label{le:extensionproduct}
Let $G_i\lhd \tG_i$ and $H_i\lhd \tH_i$ for $i=1,2$. Assume that
$\tG_i/G_i=\tH_i/H_i$ and $\ell\nmid[\tG_i:G_i]$. Let $X_i$ be a complex
inducing a $\tG_i/G_i$-equivariant $\CC$-equivalence between the principal
blocks of $G_i$ and $H_i$ for $i=1,2$. Then, $X_1\otimes_R X_2$ induces
a $(\tG_1\times\tG_2)/(G_1\times G_2)$-equivariant $\CC$-equivalence
between the principal blocks of $G_1\times G_2$ and $H_1\times H_2$.
\end{lemma}

\begin{proof}
The equivariance part is clear, so we can assume $\tG_i=G_i$ and
$\tH_i=H_i$. We have a canonical isomorphism
$\End^\bullet_{RH_1^\opp}(X_1)\otimes_R\End^\bullet_{RH_2^\opp}(X_2)
\iso \End^\bullet_{R(H_1\times H_2)^\opp}(X_1\otimes_R X_2)$.
The canonical map
$R(G_1\times G_2)\to  \End^\bullet_{R(H_1\times H_2)^\opp}(X_1\otimes_R X_2)$
is a split injection, with cokernel $L$ isomorphic to
$R_1\otimes_R RG_2\oplus RG_1\otimes_R R_2\oplus R_1\otimes_R R_2$,
where $R_i$ is the cokernel of the canonical map
$RG_i\to \End^\bullet_{RH_i^\opp}(X_i)$. It follows from
(S2) that $L\in\CC(G_1\times G_2)$.
The other property is obtained by swapping the roles of $G_i$ and $H_i^\opp$.
\end{proof}

If $A$ is an $R$-algebra, $n\ge 0$ is an integer, and $X$ is a complex of
$A$-modules, then there is a canonical extension of $X^{\otimes n}$ from
a complex of $A^{\otimes n}$-modules to a complex
of $A\wr\GS_n$-modules: it is obtained as the total complex associated
with an $n$-fold complex \cite[\S 1.1]{De} (see also \cite[Lemma 4.1]{Ma1} for
an explicit description). The following lemma is a consequence
of Lemma \ref{le:extensionproduct}.

\begin{lemma}
\label{le:wreath}
Let $G\lhd\tG$, $H\lhd \tH$ with $\tG/G=\tH/H$ and $\ell\nmid[\tG:G]$. 
Let $X$ be a complex inducing a $\CC$-equivariant equivalence between the
principal blocks of $G$ and $H$.

Let $n\ge 1$ and let $L$ be an $\ell'$-subgroup of $\GS_n$. Then,
$X^{\otimes n}$ induces a $\CC$-equivariant equivalence between the 
principal blocks of $G\wr L$ and $H\wr L$.
\end{lemma}

\subsubsection{Stability of properties of finite groups with abelian Sylow
$p$-subgroups}

Let $\CE_1$ be the set of finite groups with abelian Sylow
$\ell$-subgroups and let
$\CE$ be the set of pairs $(G,\tilde{G})$ with $G\in\CE_1$,
$G\lhd \tilde{G}$ and $\ell\nmid[\tilde{G}:G]$.

\medskip
Recall that if $P$ is a Sylow $\ell$-subgroup of $G$, then $\tilde{G}=GN_{\tilde{G}}(P)$ (Frattini argument), and
hence $N_{\tilde{G}}(P)/N_G(P)=\tilde{G}/G$. There is an $\ell'$-subgroup
$E$ of $N_{\tG}(P)$ such that $N_{\tG}(P)=P\rtimes E$. We have
$\tG=GE$ and $\tG$ is a quotient of $G\rtimes E$ by an $\ell'$-subgroup. Let $N_\Delta(G,\tilde{G})=\tDelta(H,G)$, where $H=N_G(P)$ and
$\tH=N_{\tG}(P)$. We have
$N_\Delta(G,\tilde{G})=(H\times G^\opp)\Delta \tH$.

\begin{defi}
We say that a subset $\CP$ of $\CE$ satisfies (*) if
\begin{itemize}
\item[(i)] $\CP$ is closed under direct products;
\item[(ii)] given $(H,\tilde{H})\in\CP$ and $(G,\tilde{G})\in\CE$ such that
$H\lhd G\lhd\tilde{G}\le \tilde{H}$, we have $(G,\tilde{G})\in\CP$;
\item[(iii)] if $(G,\tilde{G})\in\CP$, $n\ge 0$ and $L$ is an $\ell'$-subgroup
of $\GS_n$, then $(G^n,\tilde{G}\wr L)\in\CP$;
\item[(iv)] if $G$ is an abelian $\ell$-group, then $(G,\tilde{G})\in\CP$;
\item[(v)] $(G/O_{\ell'}(G),\tilde{G}/O_{\ell'}(\tG))\in\CP$ if and only
if $(G,\tG)\in\CP$.
\end{itemize}
We say that a subset $\CP$ of $\CE$ satisfies (*') if, in addition, we have
\begin{itemize}
\item[(vi)] given $(G,\tilde{G})\in\CP$ and $(G,\hat{G})\in\CE$ with
$\tilde{G}\le\hat{G}$, and given a Sylow $\ell$-subgroup $P$ of $G$, if
$\hat{G}=\tilde{G}C_{\hat{G}}(P)$ then $(G,\hat{G})\in\CP$.
\end{itemize}
\end{defi}

\begin{prop}
\label{pr:starreduction}
Let $\CP$ be a subset of $\CE$ satisfying (*)(resp.\ (*')). Let $\CF$ be a set of non-cyclic
finite simple groups with non-trivial abelian Sylow $\ell$-subgroups. Given $G\in\CF$, let $\hat{G}\le\Aut(G)$ be such that
the image of $\hat{G}$ in $\Out(G)$ is a Hall $\ell'$-subgroup of $\Out(G)$.
Assume that there is a pair $(G,\tG)\in\CP$ such that $\tG/GC_{\tG}(G)=\hat{G}/G$
(resp.\ such that $\tG/GC_{\tG}(G)\le \hat{G}/G$ and given a Sylow $\ell$-subgroup
$P$ of $G$, we have $N_{\tilde{G}}(P)/C_{\tilde{G}}(P)=N_{\hat{G}}(P)/C_{\hat{G}}(P)$).

If $(G,\tilde{G})\in\CE$ is such that all non-cyclic composition factors of $G$ of order divisible by $\ell$
are in $\CF$, then $(G,\tilde{G})\in\CP$.
\end{prop}

\begin{proof}
Let us show first that the assumptions for (*') imply those for (*).
Let $G\in\CF$ and $\tilde{G}$ as in the ``resp.'' case of the
proposition. Because of (*,v), we may assume that $O_{\ell'}(\tilde{G})=1$,
hence we may assume that $\tilde{G}\le\hat{G}$.
We have $\hat{G}=\tilde{G}C_{\hat{G}}(P)$, and so $(G,\hat{G})\in\CP$ by
(*',vi).

\smallskip
Let us now prove the proposition in the case (*).
One may assume that $O_{\ell'}(G)=1$.
It follows from the classification of finite simple groups \cite[\S 5]{FoHa}
that there is a collection
\begin{itemize}
\item $(H_0,\tilde{H}_0)\in\CE$ where $H_0$ is an $\ell$-group,
\item $H_1,\ldots,H_n\in\CF$,
\item $d_1,\ldots,d_n\in\BZ_{>0}$, and
\item $L_1,\ldots,L_n$ a family of $\ell'$-subgroups of $\GS_{d_1},\ldots,
\GS_{d_n}$,
\end{itemize}
and there are embeddings
$$H_0\times H_1^{d_1}\times\cdots\times H_n^{d_n}\lhd
G\lhd \tilde{G}\le \tH_0\times \hat{H}_1\wr L_1\times\cdots\times \hat{H}_n\wr L_n.$$
where $H_i\lhd\hat{H}_i$ and $\hat{H}_i/H_i$ is a Hall $\ell'$-subgroup
of $\Out(H_i)$.

Property (*,iv) ensures that $(H_0,\tilde{H}_0)\in\CP$. Assume $i>0$.
Consider a pair $(H_i,\tH_i)\in\CP$ as in the proposition. We have
$\tH_i/O_{\ell'}(\tH_i)\simeq \hat{H}_i$ and we deduce from (*,v) that
$(H_i,\hat{H}_i)\in\CP$. Hence,
$(H_i^{d_i},\hat{H}_i\wr L_i)\in\CP$ by (*,iii). We deduce that
$(H_0\times H_1^{d_1}\times\cdots\times H_n^{d_n},
\tH_0\times \hat{H}_1\wr L_1\times\cdots\times \hat{H}_n\wr L_n)\in\CP$
by (*,i),
hence $(G,\tilde{G})\in\CP$ by (*,ii).
\end{proof}

\subsubsection{Equivalences and Brou\'e's conjecture}

\begin{prop}
\label{pr:stabBroue1}
Let $\CC$ be data satisfying (S1) and (S2).
Let $\CP$ be the set of pairs $(G,\tilde{G})\in\CE$ such that there is a
$\tilde{G}/G$-equivariant $\CC$-equivalence between the principal blocks
of $G$ and $N_G(P)$. The set $\CP$ satisfies property (*').
\end{prop}

\begin{proof}
Conditions (i), (ii) and (iii) follow from Lemmas \ref{le:extensionproduct},
\ref{le:extension} and \ref{le:wreath}, respectively. 
Since the principal blocks of $G$ and $G/O_{\ell'}(G)$ are isomorphic, we
have $(G,\tG)\in\CP$ if and only if $\bigl(G/O_{\ell'}(G),\tG/O_{\ell'}(G)\bigr)\in\CP$.
Assume that $O_{\ell'}(G)=1$. Then $O_{\ell'}(\tG)$ centralizes $G$, and hence
$(G,\tG)\in\CP$ if and only if $\bigl(G,\tG/O_{\ell'}(\tG)\bigr)\in\CP$. So,
condition (v) holds.
Condition (iv) holds as well: take $X=RG$.

\smallskip
Consider now $G\lhd\tilde{G}\le\hat{G}$ with $G\lhd\hat{G}$ and
$\ell\nmid[\hat{G}:G]$. Let $P$ be a Sylow $\ell$-subgroup of $G$
and assume that $\hat{G}=\tG C_{\hat{G}}(P)$.
Let $X\in\Comp^b(RN_{\Delta}(G,\tG))$ inducing
an equivariant $\CC$-equivalence between the  principal blocks
of $G$ and $N_G(P)$. We have
$$N_{\Delta}(G,\hat{G})=\bigl(N_G(P)\times G^\opp\bigr)\Delta N_{\tG}(P)
\Delta C_{\hat{G}}(P)=
N_{\Delta}(G,\tilde{G})C_{N_{\Delta}(G,\hat{G})}(P^\en).$$
Let $Y=b_0\bigl(N_{\Delta}(G,\tilde{G})\bigr)\Ind_{N_{\Delta}(G,\tilde{G})}^{N_{\Delta}
(G,\hat{G})}(X)$.
We have $\Res_{N_{\Delta}(G,\tilde{G})}(Y)\simeq X$ by
Lemma \ref{le:AlperinDade}, and this shows condition (vi).
\end{proof}

The next proposition is clear.
\begin{prop}
\label{pr:examplesS1S2}
The following data $\CC$ satisfy properties (S1) and (S2):
\begin{itemize}
\item acyclic complexes with $R$-projective components
 ($\CC$-equivalences are standard derived equivalences);
\item contractible complexes with $R$-projective
components ($\CC$-equivalences are Rickard equivalences).
\end{itemize}
\end{prop}

\subsection{Automorphisms}
\label{se:auto}
We provide extension results for derived equivalences in the
presence of automorphism groups. The main results 
are due to Marcus \cite{Ma1,Ma2,Ma3} (except for \S \ref{se:extensionperverse}).

\subsubsection{Extensions of modules}
\label{se:extmodules}
Let $A$ be a $k$-algebra and $G$ a finite group endowed with a
homomorphism $\phi:G\to\Aut(A)$.
Let $M$ be an $(A\rtimes G)$-module. The structure
of $(A^\en\rtimes G^\en)$-module on $\End_k(M)$ restricts to
a structure of $k\Delta G$-module on $\End_A(M)$. The corresponding morphism
$G\to\Aut\bigl(\End_A(M)\bigr)$ is the following:
given $g\in G$ and $f\in\End_A(M)$, we set
$g(f)(m)=g(f(g^{-1}(m))$ for $m\in M$.
We have a canonical isomorphism $kG\otimes M\iso \Ind^{A\rtimes G}\Res_A M$, and an isomorphism of
algebras 
$$\End_A(M)\rtimes G\iso \End_{A\rtimes G}(kG\otimes M),\
f\otimes g\mapsto \bigl(h\otimes m\mapsto hg\otimes hf(h^{-1}m)\bigr).$$

Recall also that $J(A)G\subset J(A\rtimes G)$, hence 
$\Res_A^{A\rtimes G}$ preserves semi-simplicity.

\subsubsection{Two-sided tilting complexes}
Let $k$ be a commutative ring.
Let $A$ and $B$ be two flat $k$-algebras and $G$ a finite group with
homomorphisms $G\to\Aut(A)$ and $G\to\Aut(B)$.

\smallskip
We start with a classical result \cite{Ma1}.

\begin{lemma}
\label{le:equivalencetilting}
Let $X$ be a complex of $\bigl((A\otimes B^\opp)\rtimes G\bigr)$-modules.
Then $\Res_{A\otimes B^\opp}X$ is a two-sided tilting complex for $(A,B)$
if and only if $\Ind^{(A\rtimes G)\otimes (B\rtimes G)^\opp}X$
is a two-sided tilting complex for $(A\rtimes G,B\rtimes G)$.
\end{lemma}

\begin{proof}
Let $X_1=\Res_{A\otimes B^\opp}X$ and
$X_2=\Ind^{(A\rtimes G)\otimes (B\rtimes G)^\opp}X$.
Following the proof of 
Lemma \ref{le:extensionequiv}, we obtain a commutative diagram
$$\xymatrix{
B\rtimes G\ar[rr]^-{\can}\ar@{=}[d] &&
 R\End^{\bullet}_A(X_1)\rtimes G\ar[d]^\sim \\
B\rtimes G\ar[rr]^-{\can} &&
 R\End^{\bullet}_{A\rtimes G}(X_2)
}$$
There is a similar commutative diagram with the roles of $A$ and $B$
reversed. The lemma follows.
\end{proof}

\subsubsection{Rickard complexes}
The following lemma is an equivariant version of a result of Rickard
\cite[p.336]{Ri5}.

\begin{lemma}
Assume that $k$ is a regular noetherian ring and $A$ and $B$ are symmetric 
$k$-algebras.
Let $C$ be a complex of $\bigl((A\otimes B^\opp)\rtimes G\bigr)$-modules which
restricts to a two-sided tilting complex for $(A,B)$. There
is a complex $D$ of $\bigl((A\otimes B^\opp)\rtimes G\bigr)$-modules that is
quasi-isomorphic to $C$ and which restricts to a Rickard complex
for $(A,B)$.
\end{lemma}

\begin{proof}
Let $C'$ be a right-bounded complex quasi-isomorphic to $C$, all of whose terms
are finitely generated and projective. Let $n\in\BZ$ such that $H^i(C)=0$
for $i<n$. Then $D=\tau_{\ge n-m}(C')$ satisfies the required property
for $m$ the Krull dimension of $k$ (see \cite[pp.135--136]{Ri5}).
\end{proof}

A proof similar to that of Lemma \ref{le:equivalencetilting} shows the following
classical result. Note that the assumption on $|G|$ is necessary to
ensure that a complex of $(A^\en\rtimes G)$-modules is contractible if
its restriction to $A^\en$ is contractible.

\begin{lemma}
Assume that $|G|\in k^\times$.
Let $X$ be a complex of $\bigl((A\otimes B^\opp)\rtimes G\bigr)$-modules.
Then $\Res_{A\otimes B^\opp}X$ is a Rickard complex for $(A,B)$ if and only
if $\Ind^{(A\rtimes G)\otimes (B\rtimes G)^\opp}X$
is a Rickard complex for $(A\rtimes G,B\rtimes G)$.
\end{lemma}

\subsubsection{Tilting complexes}
\label{se:automorphismstilting}
Let $A$ be a flat $k$-algebra and $G$ a finite group endowed with
a homomorphism $G\to\Aut(A)$.

\begin{lemma}
\label{le:criteriontilting}
Let $T\in\Comp(A\rtimes G)$. Then $\Res_A T$ is a tilting complex
for $A$ if and only if $kG\otimes T$ is a tilting complex for
$A\rtimes G$.
There is a canonical isomorphism
$\End_{D(A)}(T)\rtimes G\iso \End_{D(A\rtimes G)}(kG\otimes T)$.
\end{lemma}

\begin{proof}
By \S \ref{se:extmodules}, we have
$$R\End^\bullet_A(T)\rtimes G\iso R\End^\bullet_{D(A\rtimes G)}(kG\otimes T).$$
Thus, $\Hom_{D(A)}(T,T[i])=0$ for $i\not=0$ if and only if
$\Hom_{D(A\rtimes G)}(kG\otimes T,kG\otimes T[i])=0$ for $i\not=0$.

Assume that $T$ is a tilting complex for $A$. Then, $\Ind^{A\rtimes G}(T)$ is
perfect. Also, $A$ is in the thick subcategory of $D(A)$ generated by
$T$, hence $A\rtimes G=\Ind^{A\rtimes G}A$ 
is in the thick subcategory of $D(A)$ generated by $\Ind^{A\rtimes G}T$.
So, $\Ind^{A\rtimes G}T$ is a tilting complex.

Conversely, assume that $\Ind^{A\rtimes G}T$ is a tilting complex.
We have $\Res_A (kG\otimes T)\simeq T^{|G|}$, and hence $T$ is a perfect
complex for $A$. Since $A$ is in the thick subcategory of $D(A\rtimes G)$
generated by $\Ind^{A\rtimes G}T$, it follows that $A$
is in the thick subcategory of $D(A\rtimes G)$
generated by $\Res_A\Ind^{A\rtimes G}T\simeq T^{|G|}$. It follows that
$T$ is a tilting complex for $A$.
\end{proof}

\smallskip
The following lemma is an equivariant version of a result of Keller
\cite[\S 8.3.1]{Ke}.

\begin{lemma}
\label{le:Toda}
Let $C\in\Comp(A\rtimes G)$ and let $B=\End_{D(A)}(C)$.
Assume that $\Hom_{D(A)}(C,C[n])=0$ for $n<0$. There
is a complex $X$ of $\bigl((A\otimes B^\opp)\rtimes G\bigr)$-modules and an isomorphism
$\phi:C\iso X$ in $D(A\rtimes G)$ such that the composition of canonical maps
$\End_{D(A)}(C)=B\to\End_{\Comp(A)}(X)\to \End_{D(A)}(X)$ is given by $\phi$.
\end{lemma}

\begin{proof}
Up to isomorphism in $D(A\rtimes G)$, we
may assume that $C$ is homotopically projective.
Let $R=\bigl(A\otimes \End_A^\bullet(C)^\opp\bigr)\rtimes G$, a dg algebra.
The complex $C$ extends to a dg $R$-module. Let
$R_-=\tau_{\le 0}(R)$, a dg subalgebra of $R$. We have $H^0(R_-)\iso
H^0(R)=(A\otimes B)\rtimes G$.
Let $X=H^0(R)\otimes_{R_-}^\BL C$. The canonical quasi-isomorphism
$R_-\iso H^0(R_-)$ induces an isomorphism $\phi:C\iso X$ in $D(A\rtimes G)$.
It satisfies the stated property.
\end{proof}

Lemma \ref{le:Toda} gives the following useful criterion to extend
equivalences.

\begin{lemma}
\label{le:ext2side}
Let $C\in\Comp(A\rtimes G)$ be such that
$\Res_A C$ is a tilting complex. Let $B=\End_{D(A)}(C)$.
There is a complex $X$ of $\bigl((A\otimes B^\opp)\rtimes G\bigr)$-modules
and an isomorphism
$\phi:C\iso X$ in $D(A\rtimes G)$ such that $\Res_{A\otimes B^\opp}(X)$
is a two-sided tilting complex and the composition of canonical maps
$\End_{D(A)}(C)=B\to\End_{\Comp(A)}(X)\to \End_{D(A)}(X)$ is given by $\phi$.
\end{lemma}

\begin{lemma}
\label{le:lefttocentre}
Let $T$ be a two-sided tilting complex for $(A,B)$. Assume that there
is a complex $C$ of $(A\rtimes G)$-modules such that
$T\simeq C$ in $D(A)$. There is an action of $G$ on $B$ and a complex $X$
of $\bigl((A\otimes B^\opp)\rtimes G\bigr)$-modules such that $X\simeq T$ in
$D(A\otimes B^\opp)$.
\end{lemma}

\begin{proof}
Consider $X$ as in Lemma \ref{le:ext2side}.
Then $\Res_{A\otimes B^\opp}(X)$ and
$T$ are two-sided tilting complexes and there is an isomorphism
$X\iso T$ in $D(A)$ compatible with the canonical maps
$B\to \End_{D(A)}(X)$ and $B\to \End_{D(A)}(T)$. This forces
$X$ and $T$ to be isomorphic in $D(A\otimes B^\opp)$
(Lemma \ref{le:uptohomotopy}).
\end{proof}

\begin{lemma}
\label{le:uptohomotopy}
Let $C,D\in\Comp(A\otimes B^\opp)$ such that
$C$ is a two-sided tilting complex. Assume that there is an isomorphism
$\phi:C\iso D$ in $D(A)$ such that 
the composition of canonical maps
$\End_{D(A)}(C)=B\to\End_{\Comp(A)}(X)\to \End_{D(A)}(X)$ is given by $\phi$. There is an isomorphism $C\iso D$ in $D(A\otimes B^\opp)$ restricting to $\phi$.
\end{lemma}

\begin{proof}
The map $\phi$ induces isomorphisms of $B^\en$-modules
$\Hom_{D(A)}(C,C[i])\iso \Hom_{D(A)}(C,D[i])$ for $i\in\BZ$. 
The canonical map $B\to R\End_A^\bullet(C)$ is an isomorphism
in $D(B^\en)$ and we obtain an isomorphism
$B\iso R\Hom_A^\bullet(C,D)$ in $D(B^\en)$. Applying
$C\otimes_B^\BL -$ gives the required isomorphism.
\end{proof}

\subsubsection{Equivariant lifts of stable equivalences}
We assume that $k$ is a field for the rest of \S \ref{se:auto}.

\begin{prop}
\label{pr:equivariantstable}
Let $A$ and $B$ be two symmetric $k$-algebras with no simple direct factors
and endowed with an action of a
finite group $G$. Let $M$ be a bounded
complex of $\bigl((A\otimes B^\opp)\rtimes G\bigr)$-modules
such that $\Res_{A\otimes B^\opp}M$ induces a stable equivalence.
Assume that there is a two-sided tilting complex $T$ for $A\otimes B^\opp$ that
is isomorphic to $\Res_{A\otimes B^\opp}(M)$ in $(A\otimes B^\opp)\mstab$,
and such that $\Res_A(T)$ extends to a complex of $(A\rtimes G)$-modules.

There is a bounded complex $C$ of $\bigl((A\otimes B^\opp)\rtimes G\bigr)$-modules
such that $\Res_{A\otimes B^\opp}C\simeq T$ in $D(A\otimes B^\opp)$.
\end{prop}

\begin{proof}
Let $\rho:G\to\Aut(A)$ and $\psi:G\to\Aut(B)$ be the canonical homomorphisms.
Lemma \ref{le:lefttocentre} shows that there is a homomorphism
$\psi':G\to\Aut(B)$ and a complex $X$ of $\bigl((A\otimes B^\opp)
\rtimes_{\rho,\psi'}G\bigr)$-modules such that $\Res_{A\otimes B^\opp}(X)$ is a
two-sided tilting complex and $X\simeq T$ in $D(A\otimes B^\opp)$.

Let $N$ be an $\bigl((A\otimes B^\opp)
\rtimes_{\rho,\psi'}G\bigr)$-module that is isomorphic to $X$ in
$\bigl((A\otimes B^\opp)\rtimes_{\rho,\psi'}G\bigr)\mstab$. Then
$\Res_{A\otimes B^\opp}N\simeq \Res_{A\otimes B^\opp}M$ in
$(A\otimes B^\opp)\mstab$. We have the structure of a
$\bigl((A\otimes B^\opp)\rtimes_{\rho,\psi} G,
(A\otimes B^\opp)\rtimes_{\rho,\psi'} G\bigr)$-bimodule on 
$\Hom_k(N,M)$, and this restricts to a
a structure of $(B^\en\rtimes_{\psi,\psi'} G)$-module on $\Hom_A(N,M)$.
We have $\Hom_A(N,M)\simeq B\oplus R$ as $B^\en$-modules, where $R$
is a projective $B^\en$-module.
We deduce that the $B^\en$-module $B$ extends to a 
$(B^\en\rtimes_{\psi,\psi'} G)$-module $L$. Let
$C=\Hom^\bullet_{B^{\opp}}(L,X)$: this is a complex of
$\bigl((A\otimes B^\opp)\rtimes_{\rho,\psi}G\bigr)$-modules satisfying the 
required property.
\end{proof}

We have a descent counterpart.

\begin{prop}
\label{pr:descentstable}
Let $A$ and $B$ be two symmetric $k$-algebras with no simple direct factors
and endowed with an action of a
finite group $G$ such that $|G|\in k^\times$. Let $M$ be a bounded
complex of $\bigl((A\otimes B^\opp)\rtimes G\bigr)$-modules
such that $\Res_{A\otimes B^\opp}M$ induces a stable equivalence.
Assume that there is a two-sided tilting complex $T$ for 
$\bigl((A\rtimes G),(B\rtimes G)\bigr)$ that
is isomorphic to $\Ind^{(A\rtimes G)\otimes (B\rtimes G)^\opp}(M)$ in
$\bigl((A\rtimes G)\otimes (B\rtimes G)^\opp\bigr)\mstab$,
and such that there is an isomorphism $\Res_{A\rtimes G}T\simeq C\otimes kG$ in
$D(A\rtimes G)$, where $C\in\Comp(A\rtimes G)$.

There is a bounded complex $Y$ of $\bigl((A\otimes B^\opp)\rtimes G\bigr)$-modules
such that $\Ind^{(A\rtimes G)\otimes (B\rtimes G)^\opp}(Y)\simeq T$ in
$D\bigl((A\rtimes G)\otimes (B\rtimes G)^\opp\bigr)$.
\end{prop}

\begin{proof}
Lemma \ref{le:criteriontilting} shows that $C$ is a tilting complex for $A$.
Let $B'=\End_{D(A)}(C)$.
Let $X\in\Comp\bigl((A\otimes B^{\prime\opp})\rtimes G\bigr)$ as in
Lemma \ref{le:ext2side}.
We have an isomorphism
$C\iso X$ in $D(A\rtimes G)$, whose restriction to $D(A)$ is compatible
with the action of $B'$ up to homotopy.
Let $N$ be an $\bigl((A\otimes B')\rtimes G\bigr)$-module isomorphic to $X$ in 
$(A\otimes B')\rtimes G)\mstab$. We have
$\Hom_A(M,N)=L\oplus R$ as $(B,B')$-bimodules, where $L$ has
no projective direct summand. Then $L$ extends to a
$((B\otimes B^{\prime\opp})\rtimes G)$-module. The
$(B\rtimes G,B'\rtimes G)$-bimodule $L'=\Ind^{(B\rtimes G)\otimes
(B'\rtimes G)^\opp}_{(B\otimes B^{\prime\opp})\rtimes G}(L)$ induces
a stable equivalence that sends simple modules to simple modules. Hence,
$L\otimes_B-$ sends simple modules to semi-simple modules. Since it
induces a stable equivalence, it sends simple modules to simple
modules, and we deduce that $L\otimes_{B'}-:B'\mMod\to B\mMod$ is
an equivalence \cite[Theorem 2.1]{Li2} and
the equivalence $L'\otimes_{B'\rtimes G}:(B'\times G)\mMod\iso
(B\rtimes G)\mMod$ comes from an isomorphism $B\rtimes G\iso B'\rtimes G$.
We deduce that there is a $G$-invariant isomorphism $\sigma:B\iso B'$ 
such that $B_{\sigma}\iso L$ as $(B\otimes B')\rtimes G$-modules. The 
complex $Y=\sigma^*X$ satisfies the required property.
\end{proof}

\subsubsection{Okuyama's sequential lifts}
\label{se:Okuyama}
Let $A$ and $B$ be two symmetric $k$-algebras with no simple direct
factors and acted on by a finite group
$G$. Assume that the simple $A$ and $B$-modules are absolutely simple.
Let $M$ an $\bigl((A\otimes B^\opp)\rtimes G\bigr)$-module, projective
as an $A$-module and as a $B^\opp$-module. We assume that $M$ induces a
stable equivalence between $A$ and $B$ and it has no non-zero projective
direct summand.

Let $\phi:P\to M$ be a projective cover of $M$. We have
$\Res_{A\otimes B^\opp} P\simeq \bigoplus_{S\in\CS_B}
P_{M\otimes_B S}\otimes P_S^*$ (cf.\ \cite[Lemma 2.12]{Rou2} and
Lemma \ref{le:extensionprojectivecovers} below).

Let $I$ be a $G$-invariant subset of $\CS_B$. Then, there is
a direct summand $Q$ of $P$ such that $\Res_{A\otimes B^\opp} Q\simeq 
\bigoplus_{S\in I} P_{M\otimes_B S}\otimes P_S^*$.
Let $C=C(M,I)=0\to Q\xrightarrow{\phi_{|Q}} M\to 0$, a complex
of $((A\otimes B^\opp)\rtimes G)$-modules with $M$ in degree $0$.

Assume that $\Res_A(C)$ is a tilting complex and let $B'=\End_{D(A)}(C)$.
Lemma \ref{le:ext2side} provides a complex $X$ of 
$\bigl((A\otimes B^{\prime\opp})\rtimes G\bigr)$-modules such that
$\Res_{A\otimes B^{\prime\opp}}(X)$ is a two-sided tilting complex.

\medskip
We consider a situation studied (in the non-equivariant setting)
by Okuyama \cite{Ok2}.
Let $I_0,\ldots,I_l$ be $G$-invariant subsets of $\CS_B$. We assume that
the sequence satisfies the following conditions.
Let $A_0=A$ and $M_0=M$. Assume that a $G$-algebra $A_i$ and an
$\bigl((A_i\otimes B^\opp)\rtimes G\bigr)$-module $M_i$ have been defined.
We assume that $\Res_{A_i}C(M_i,I_i)$ is a tilting complex. We set
$A_{i+1}=\End_{D(A_i)}\bigl(C(M_i,I_i)\bigr)$. We view $A_{i+1}$ as
an $\bigl((A_{i+1}\otimes B)\rtimes G\bigr)$-module by restricting the
canonical $(A_{i+1}^\en\rtimes G)$-module structure via the canonical map
$B\to \End_{D(A_i)}\bigl(C(M_i,I_i)\bigr)$, and we denote by
$M_{i+1}$ a maximal direct summand of $A_{i+1}$ with no non-zero
projective direct summand. Note that 
$\Res_{A_{i+1}\otimes B^\opp}(M_{i+1})$ induces a stable equivalence.

Let us assume finally that $B=A_{l+1}$ (as algebras, without $G$-action).
We have a sequence of derived equivalences
$$D^b(A)\iso D^b(A_1)\iso\cdots\iso D^b(A_l)\iso D^b(B),$$
whose composition
lifts the stable equivalence induced by $\Hom_A(M,-)$. It follows
from Proposition \ref{pr:equivariantstable} that there is a complex
of $\bigl((A\otimes B^\opp)\rtimes G\bigr)$-modules whose restriction to
$A\otimes B^\opp$ is a two-sided tilting complex.

\subsubsection{Perverse equivalences}
\label{se:extensionperverse}
Let $A$ be a symmetric $k$-algebra and
$G$ a finite group of automorphisms of $A$ with $|G|\in k^\times$.

Define an equivalence relation on $\CS_{A\rtimes G}$ by
$M\sim N$ if $\Hom_A(\Res_A M,\Res_A N)\not=0$.
Induction and restriction define a bijection 
$\CS_A/G \iso \CS_{A\rtimes G}$. 

The perversity datum
$(q,\CS_\bullet)$ for $A\rtimes G$
is said to be {\em $G$-compatible} if it is compatible with $\sim$.

\begin{prop}
\label{pr:perverseextension}
Let $(q_1,\CS^1_\bullet),\ldots,(q_d,\CS^d_\bullet)$ be a family of
$G$-invariant
perversity data for $A$ and let
$(q'_1,\CS^{\prime 1}_\bullet),\ldots,(q'_d,\CS^{\prime d}_\bullet)$ be
the corresponding $G$-compatible perversity data for $A\rtimes G$. There are algebras $B_0=A,\ldots,B_d$ endowed with a $G$-action
and complexes $X_i$ of $\bigl((B_{i-1}\otimes B_i^\opp)\rtimes G\bigr)$-modules
such that 
\begin{itemize}
\item $\Res_{B_{i-1}\otimes B_i^\opp}(X_i)$ induces a
perverse equivalence $D^b(B_i)\iso D^b(B_{i-1})$ relative to
$(q_i,\CS^i_\bullet)$, and
\item $\Ind^{(B_{i-1}\rtimes G)\otimes (B_i\rtimes G)^\opp}(X_i)$ induces a
perverse equivalence relative to $(q'_i,\CS^{\prime i}_\bullet)$.
\end{itemize}
\end{prop}

\begin{proof}
Every perverse equivalence is a composition of perverse equivalences
associated to $2$-step filtrations with $q(1)-q(0)=\pm 1$. This
shows that it is enough to prove the proposition when the $(q_i,\CS^i_\bullet)$
satisfy that requirement. Next, it is enough to deal with an individual
$(q,\CS^i_\bullet)$. Shifting if necessary, it is enough to deal with
the case $q(0)=0$.

Consider an elementary equivalence as in \S \ref{se:elementary} .
The submodule $U$ of $A$ is $G$-stable and so is
the submodule $V$ of $A/J(A)$. 
It follows from Lemma \ref{le:extensionprojectivecovers} below that
the complex $X$ extends to a complex of $(A\rtimes G)$-modules.
Lemma \ref{le:ext2side} shows that there is a complex $C$ of
$\bigl((A\otimes B^\opp)\rtimes G\bigr)$-modules inducing a perverse
equivalence.

The case of a perverse equivalence associated to a two-step
filtration and $q$ given by $q(0)=0$, $q(1)=-1$ is handled similarly.

The last part of the lemma follows from \ref{le:indperverse} below.
\end{proof}

\begin{lemma}
\label{le:extensionprojectivecovers}
Let $M$ be an $(A\rtimes G)$-module with projective cover $P$. Then
$\Res_A P$ is a projective cover of $\Res_A M$.
\end{lemma}

\begin{proof}
Set $A'=A\rtimes G$. The $A$-module $A'/(A'J(A)A')$ is semi-simple, hence
it is semi-simple as an $A'$-module, so
$J(A')\subset A'J(A)A'=A'J(A)$ and $\soc(A)\subset \soc(A')$.
Let $N$ be the kernel of
a projective cover $P\to M$. Since $N$ has no non-zero projective
direct summand, we have $\soc(A')N=0$, hence $\soc(A)N=0$: this shows that
$\Res_A N$ has no non-zero projective direct summand.
\end{proof}

\begin{rem}
Note that if $G$ acts trivially
on $\CS_A$, then sequences of perversity data are automatically $G$-invariant.
\end{rem}

\begin{lemma}
\label{le:indperverse}
Let $B$ be a symmetric $k$-algebra endowed with an action of $G$.
Let $X\in\Comp^b((A\otimes B^\opp)\rtimes G)$. 
Let $(q,\CS_{A,\bullet},\CS_{B,\bullet})$ be a $G$-invariant perversity datum
and $(q',\CS_{A\rtimes G,\bullet},\CS_{B\rtimes G,\bullet})$ the
corresponding $G$-compatible datum.

Then
$\Res_{A\otimes B^\opp}X$ induces a perverse equivalence relative
to $(q,\CS_{A,\bullet},\CS_{B,\bullet})$ if and only if
$\Ind^{(A\rtimes G)\otimes (B\rtimes G)^\opp}(X)$ induces
a perverse equivalence relative to
$(q',\CS_{A\rtimes G,\bullet},\CS_{B\rtimes G,\bullet})$.
\end{lemma}

\begin{proof}
Let $Y=\Ind^{(A\rtimes G)\otimes (B\rtimes G)^\opp}(X)$.
The equivalence part is Lemma \ref{le:equivalencetilting}.
Let $S'\in\CS_{B\rtimes G}$ and $L=Y\otimes^{\BL}_{B\rtimes G}S'$. We  have
$\Res_A L\simeq X\otimes^{\BL}_B\Res_B(S')$. 

Assume that $\Res_{A\otimes B^\opp}X$ induces a perverse equivalence, and let $S'\in\CS_{B\rtimes G,i}$. Since $\Res_B(S')$ is
a direct sum of simple modules in $\CS_{B,i}$, we deduce that
$\Res_A H^j(L)$ has composition factors in $\CS_{A,i-1}$ for
$j\not=-q(i)$ and composition factors in $\CS_{A,i}$ for
$j=-q(i)$. We deduce that the composition factors of $H^j(L)$ have
the required property and $Y$ induces a perverse equivalence.
The converse statement has a similar proof.
\end{proof}

The following result shows that the lifting strategy is well behaved
with respect to outer automorphisms.

\begin{cor}
\label{co:liftingperv}
Let $A$ and $B$ be two symmetric $k$-algebras with no simple direct
factors and endowed with the action of a finite group $G$ with
$|G|\in k^\times$.
Let $M$ be a bounded
complex of $\bigl((A\otimes B^\opp)\rtimes G\bigr)$-modules
such that $\Res_{A\otimes B^\opp}M$ induces a stable equivalence.
Let
$(q_1,\CS^1_\bullet),\ldots,(q_d,\CS^d_\bullet)$ be a family of $G$-invariant
perversity data for $A$ and let
$(q'_1,\CS^{\prime 1}_\bullet),\ldots,(q'_d,\CS^{\prime d}_\bullet)$ be
the corresponding $G$-compatible perversity data for $A\rtimes G$.
Let $B_0=A,\ldots,B_d$ be algebras endowed with a $G$-action
and $X_i$ be complexes of $\bigl((B_{i-1}\otimes B_i^\opp)\rtimes G\bigr)$-modules
for $i=1,\ldots,d$, such that 
$\Res_{B_{i-1}\otimes B_i^\opp}(X_i)$ induces a
perverse equivalence $F_i:D^b(B_i)\iso D^b(B_{i-1})$ relative to
$(q_i,\CS^i_\bullet)$. Assume that the sets $\{M\otimes_B S\}_{S\in\CS_B}$ and
$\{F_1\cdots F_d(T)\}_{T\in\CS_{B_d}}$ coincide up to isomorphism
in $A\mstab$.

There is a $\bigl((B_d\otimes B^\opp)\rtimes G\bigr)$-module $N$ such that
$\Res_{B_d\otimes B^\opp}N$ induces a Morita equivalence and such that
the composition of perverse equivalences
$$D^b(A)\xrightarrow[\sim]{F_1}D^b(B_1)\to\cdots\to D^b(B_d)
\xrightarrow[\sim]{\Hom_{B_d}(N,-)}D^b(B)$$
lifts the stable equivalence induced by $M$.
\end{cor}

\begin{proof}
The existence of the algebras $B_i$ and of the complexes $X_i$ is
provided by Proposition \ref{pr:perverseextension}.
There is a $\bigl((B\otimes B_d^\opp)\rtimes G\bigr)$-module $N$ with no
projective direct summands and such that $\Res_{B_d\otimes B^\opp}N$
induces a stable equivalence isomorphic to the one induced by
the composition of functors
$$D^b(B)\xrightarrow{M\otimes_B-}D^b(A)\xrightarrow[\sim]{F_1}D^b(B_1)\to
\cdots\to D^b(B_d).$$
That stable equivalence sends simple modules to simple modules, and hence
$N$ induces a Morita equivalence by \cite[Theorem 2.1]{Li2}. The 
corollary follows.
\end{proof}

\subsection{Equivalences with extra structure}
\label{se:extra}
\subsubsection{Particular equivalences}
We consider data $\CB$ consisting,
for every $(G,\tG)\in\CE$, of
a family $\CB(G,\tG)$ of objects of $\Comp\bigl(RB_0\bigl(N_{\Delta}(G,\tG)\bigr)\bigr)$.

\smallskip
We say that $(G,\tG)\in\CE$
satisfies Brou\'e's $(\CB,\CC)$-conjecture (for principal
blocks) if there exists
$X\in\CB(G,\tG)$ such that $\Res_{N_G(P)\times G^\opp}(X)$
induces a $\CC$-equivalence between the principal blocks
of $G$ and $N_G(P)$.

\smallskip
We say that $\CB$ satisfies (S3) if:
\begin{itemize}
\item[(i)]
whenever $(G_i,\tG_i)\in\CE$ and $X_i\in\CB(G_i,\tG_i)$ for $i=1,2$, we have
$X_1\otimes X_2\in \CB(G_1\times G_2,\tG_1\times\tG_2)$;
\item[(ii)] whenever $G_1\lhd G_2\lhd\tG_2\le\tG_1$ with $G_1\lhd \tG_1$,
if $X\in \CB(G_1,\tG_1)$, we have
\[b_0(G_2)\Ind^{N_\Delta(G_2,\tG_2)}\Res_{N_\Delta(G_1,\tG_1)\cap (N_{\tG_2}(P_2)\times\tG_2^\opp)}(X)\in\CB(G_2,\tG_2);\]
\item[(iii)] given $(G,\tG)\in\CE$, $X\in\CB(G,\tG)$,
 $n\ge 0$ and $L$ an $\ell'$-subgroup
of $\GS_n$, we have $X^{\otimes n}\in \CB(G^n,\tG\wr L)$;
\item[(iv)] whenever $G$ is an abelian $\ell$-group and $(G,\tG)\in\CE$, we have
$RG\in\CB(G,\tG)$;
\item[(v)] given $(G,\tilde{G})\in\CE$ and $(G,\hat{G})\in\CE$ with
$\tilde{G}\le\hat{G}$ and $\hat{G}=\tilde{G}C_{\hat{G}}(P)$ where
$P$ is a Sylow $\ell$-subgroup of $G$, if $X\in\CB(G,\tG)$, then
$b_0\bigl(N_\Delta(G,\hat{G})\bigr)\Ind^{N_\Delta(G,\hat{G})}(X)\in\CB(G,\hat{G})$.
\end{itemize}

\begin{prop}
\label{pr:reductionS3}
Let $\CP$ be the set of pairs $(G,\tilde{G})\in\CE$ satisfying Brou\'e's
$(\CB,\CC)$-conjecture.

If $\CB$ satisfies (S3) and
$\CC$ satisfies (S1) and (S2), then $\CP$ satisfies property (*').
\end{prop}

\begin{proof}
The proposition follows from Proposition \ref{pr:stabBroue1} and its proof.
\end{proof}

\subsubsection{Examples of stable data}

Let us define various data $\CB$. Given $(G,\tG)\in\CE$, we describe
the condition for $X$ to be in $\CB(G,\tG)$. We set
$A=RB_0(G)$, $B=RB_0\bigl(N_G(P)\bigr)$ and $Y=\Res_{A\otimes B^\opp}(X)$. We
denote by $E$ an $\ell'$-subgroup of $N_{\tG}(P)$ such that
$N_{\tG}(P)=PE$.

\begin{itemize}
\item {\em Splendid complexes}:
$Y^i$ is a direct summand of a direct sum of modules of the form
$\Ind_{\Delta Q}^{N_G(P)\times G^\opp}(R)$, where $Q\le P$, for $i\in\BZ$.

\item {\em (Increasing) perverse complexes} ($R=k$):
$Y\otimes_A-$ is perverse relative to some datum $(q,\CS_{A,\bullet},
\CS_{B,\bullet})$ (resp.\ and $q$ is increasing).

\item {\em Iterated perverse complexes} ($R=k$ or $\CO$):
there is a sequence of algebras $A_1=B,A_2,\ldots,A_l=A$ with actions
of $E$ and complexes $X_i\in \Comp^b\bigl((A_i\otimes A_{i+1}^\opp)\rtimes E\bigr)$
for $i=1,\ldots,l-1$ such that 
\begin{itemize}
\item the actions of $E$ on $A_1$ and $A_l$ are the canonical actions,
\item $kX_i\otimes_{kA_{i+1}}^\BL-$ is perverse, and
\item $X\simeq X_1\otimes_{A_2}^\BL\cdots\otimes_{A_{l-1}}^\BL X_{l-1}$
in $D\bigl((B\otimes A^\opp)\rtimes E\bigr)$.
\end{itemize}

\item {\em Positively gradable complexes} ($R=k$ or $\CO$).
There is a non-negative grading on $kB_0(G)$ and
a structure of graded $(kB\otimes kA^\opp)$-module on $kY$.
Here, we take a tight grading on $kB$, \ie, one for which
$J^i(kB)=(kB)_i\oplus J^{i+1}(kB)$ \cite{Rou6}.

\item {\em Character maps} ($R=K$).
$H^i(Y)=0$ for $i\not=0,1$ and $\Hom_{A\otimes B^\opp}\bigl(H^0(Y),H^1(Y)\bigr)=0$.

\item {\em Perfect character maps}: $Y$ defines a character map (cf.\ above) and
denoting by $\mu$ the character of $X$, the following holds for
$g\in \tG$ and $h\in N_{\tG(P)}$ such that $(h,g)\in N_{\Delta}(G,\tG)$:
\begin{itemize}
\item $\mu(h,g)\in \gcd\bigl(|C_{\tG}(g)|,|C_{N_{\tG}(P)}(h)|\bigr)\CO$;
\item if one of $g,h$ is an $\ell'$-element and the other is not, then
$\mu(h,g)=0$.
\end{itemize}
\end{itemize}

\begin{rem}
Note that the data $\CB$ defined above prescribe conditions only on
$\Res_{N_G(P)\times G^\opp}(X)$,
except in the cases of iterated perverse complexes and perfect character
maps.
\end{rem}

Let us explain how the definitions above relate to classical notions.

\smallskip
Let $R=K$ and $\CB$ be the data
of character maps. Assume that $\tG=G$. Then
$[X\otimes_{KG}-]:K_0\bigl(KB_0(G)\bigr)\to K_0\bigl(KB_0\bigl(N_G(P)\bigr)\bigr)$ is
a morphism of abelian groups. This gives a bijection
$$\CB(G,G)/\mathrm{iso.}\iso
\Hom_{\BZ\mMod}\Bigl(K_0\bigl(KB_0(G)\bigr),K_0\bigl(KB_0\bigl(N_G(P)\bigl)\bigl)\Bigl).$$

Let $\CC$ be the class of acyclic complexes. Then, $X$ induces a
$\CC$-equivalence if and only if $[X\otimes_{KG}-]$ is an
isometry (we require isometries to be bijective). The isometry is
perfect \cite{Br2} if $\CB$ is taken to be the data of perfect character
maps.

\medskip
Let $R$ be either $k$ or $\CO$, and let $\CC$ be given by complexes homotopy equivalent to
$0$ and $\CB$ by splendid complexes. An object $X\in\CB(G,G)$
induces a $\CC$-equivalence if and only if it induces a splendid
Rickard equivalence \cite{Ri5}.

\begin{prop}
\label{pr:specialB}
The data $\CB$ defined above satisfy property (S3).
\end{prop}

\begin{proof}
The case of character maps is immediate, while the case of perfect character
maps follows from \cite[Theorem 1E, Theorem 2B, \S 6]{FoHa}.

The properties (S3, i, iii, iv, v) are easy in all other cases.

Property (S3, ii) for (iterated, increasing) perverse complexes follows from
Lemma \ref{le:indperverse}.

Let us consider the case of graded complexes. Take $R=k$.
Let $(G,\tG)\in\CE$ and $X$ a complex of
$\bigl((B\otimes A^\opp)\rtimes E\bigr)$-modules whose restriction to $B\otimes A^\opp$
is a two-sided tilting complex.
There is an $E$-invariant tight grading on
$B$ \cite[\S 6.2.1]{Rou6}. Proposition
\ref{pr:liftgradings} below shows that there is an $E$-invariant grading
on $A$ and a compatible grading on $X$. Furthermore, if there is
a non-negative grading on $A$ compatible with that on $B$ via
$\Res_{B\otimes A^\opp}(X)$, then we can choose an $E$-invariant
compatible grading on $A$ that is non-negative. We deduce that
property (S3, ii) holds.

The property is clear for the other types of data.
\end{proof}

The following proposition examines the behaviour of automorphisms and
gradings under actions of finite groups.

\begin{prop}
\label{pr:liftgradings}
Assume that $k$ is algebraically closed.
Let $A$ and $B$ be two finite dimensional $k$-algebras endowed with
the action of a finite group $G$. Let $\Out^G(A)$ be the image of
$C_{\Aut(A)}(G)$ in $\Out(A)$.

Let $C$ be a complex of $\bigl((A\otimes B^\opp)\rtimes G\bigr)$-modules such that
$\Res_{A\otimes B^\opp}(C)$ is a two-sided tilting complex. Then
the canonical isomorphism $\Out(A)^0\iso\Out(B)^0$ induced by $C$
restricts to an isomorphism $\Out^G(A)^0\iso\Out^G(B)^0$.

Assume that $|G|\in k^\times$, and fix a grading on $B$ invariant under $G$.
Then there is a grading on $A$
invariant under $G$ and a structure of graded 
$\bigl((A\otimes B^\opp)\rtimes G\bigr)$-module on $C$, where $G$ is in degree $0$.

Assume that the morphism $\BG_m\to\Out^G(B)$ induced by the grading on $B$ can
be lifted to a morphism $\BG_m\to \Aut(B)$ such that the corresponding grading
is non-negative. Then, it can also be
be lifted to a morphism $\BG_m\to C_{\Aut(B)}(G)$ such that the corresponding
$G$-invariant grading is non-negative.
\end{prop}

\begin{proof}
We use the results of \cite[\S 4.2]{Rou6}. Let $D\Pic(A)$ be the locally
algebraic group whose $S$ points is the group of quasi-isomorphism classes of
invertible complexes of $(A^\en\otimes\CO_S)$-modules \cite{Ye}.
Let $D\Pic^G(A)$ be its subgroup of complexes that extend to a structure of 
complexes of $\bigl((A^\en\rtimes G)\otimes\CO_S\bigr)$-modules.
Let $\Pic^f(A)$ the subgroup of $D\Pic(A)$ whose $S$-points consist of
isomorphism classes of $(A^\en\otimes\CO_S)$-modules that are locally
free of rank $1$ as $(A\otimes\CO_S)$-modules and as
$(A^\opp\otimes\CO_S)$-modules. Let $\Pic^{f,G}(A)=\Pic^f(A)\cap D\Pic^G(A)$.
The canonical map $\Pic^f(A)\to\Out(A)$ restricts to a map
$\Pic^{f,G}(A)\to \Out^G(A)$.
The canonical isomorphism $R\Hom^\bullet_A(C,-\otimes_A^\BL C):
D\Pic(A)\iso D\Pic(B)$ restricts to isomorphisms
$\Pic^f(A)\iso \Pic^f(B)$ and $D\Pic^G(A)\iso D\Pic^G(B)$, and hence to
$\Pic^{f,G}(A)\iso \Pic^{f,G}(B)$. Passing to quotients, we obtain an
isomorphism between connected components $\Out^G(A)^0\iso\Out^G(B)^0$.

\smallskip
Let us assume now that there is a $G$-invariant grading on $B$, \ie, a
morphism $\phi:\BG_m\to C_{\Aut(B)}(G)$. This induces a morphism
$\BG_m\to\Out^G(B)$, and hence a morphism $\psi:\BG_m\to\Out^G(A)$. In order to
show that this morphism lifts to a morphism $\BG_m\to C_{\Aut(A)}(G)$, it
is enough to prove that there is a lift as a morphism of varieties,
in a neighbourhood of the identity of $\BG_m$.

We have the canonical structure of a
$\bigl((B^\en\rtimes G)\otimes\CO_{C_{\Aut(B)}(G)}\bigr)$-module on
$B\otimes\CO_{C_{\Aut(B)}(G)}$: the action of 
$(B\rtimes G)\otimes\CO_{C_{\Aut(B)}(G)}$ is the canonical one, while
the right action of $b\in B$ is given by right multiplication by $\rho(b)$,
where $\rho:B\to B\otimes\CO_{{C_{\Aut(B)}(G)}}$ is the universal
algebra map: at a closed point of $C_{\Aut(B)}(G)$, it is the
corresponding automorphism of $B$.

Let $L=R\Hom^\bullet_{B^\opp}\bigl(C,C\otimes^\BL_B (B\otimes\CO_{C_{\Aut(B)}(G)})\bigr)$,
a complex of $\bigl((A^\en\rtimes G)\otimes\CO_{C_{\Aut(B)}(G)}\bigr)$-modules.
We have $L(1)\simeq A$ in $D(A^\en\rtimes G)$, hence
there is a neighbourhood $U$ of the identity in $C_{\Aut(A)}(G)$ such that
$H^i(L)=0$ for $i\not=0$ and
$H^0(L)$ is an 
$\bigl((A^\en\rtimes G)\otimes\CO_U\bigr)$-module
that is free of rank $1$ as a left $(A\otimes\CO_U)$-module and as a right
$(A^\opp\otimes\CO_U)$-module. Shrinking $U$, we can
assume that $H^0(L)$ is isomorphic to $A\otimes\CO_U$ as
an $\bigl((A\rtimes G)\otimes\CO_U\bigr)$-module, since $A$ is projective
as an $(A\rtimes G)$-module, hence it is rigid.
Fixing an isomorphism of
$\bigl((A\rtimes G)\otimes\CO_U\bigr)$-modules
$H^0(L)\iso A\otimes\CO_U$ provides
a morphism $U\to C_{\Aut(A)}(G)$. We obtain a morphism 
$\phi^{-1}(U)\to C_{\Aut(A)}(G)$ lifting $\psi$ locally and we are done.

\smallskip
Assume that there is a lift $\BG_m\to\Aut(B)$ such that the
grading is non-negative. Consider the original grading on $B$.
Given two simple $B$-modules $S$ and $T$ in degree $0$, let $f(S,T)$ be
the smallest integer $d$ such that there are simple $B$-modules
$S_1,\ldots,S_n$ in degree $0$ and
integers $d_0,\ldots,d_n$ with $\sum d_i=d$ and
\begin{multline*}
\Ext^1_B(S,S_1\langle -d_0\rangle)\not=0,\Ext^1_B(S_1,S_2\langle
-d_1\rangle)\not=0,\ldots,\\
\Ext^1_B(S_{n-1},S_n\langle -d_{n-1}\rangle)\not=0,\Ext^1_B(S_n,T\langle
-d_n\rangle)\not=0.
\end{multline*}
\cite[Proposition 5.14 and Lemma 5.15]{Rou6} shows the existence
of a function $d:\CS_B\to\BZ$ such that given two simple
$B$-modules $S$ and $T$ in degree $0$, then $f(S,T)+d(S)-d(T)\ge 0$. Furthermore,
the proof of \cite[Lemma 5.15]{Rou6} provides a function that is
$G$-invariant, since $f$ is $G$-invariant.
Given $l\in\BZ$, there is a
decomposition as $(B\rtimes G)$-modules
$B=\bigoplus_l P_l$, where $P_l\simeq\bigoplus_{d(S)=l}P_S^{\dim S}$.
The graded algebra of endomorphisms of
$\bigoplus_l P_l\langle l\rangle$ has a non-negative grading, it is isomorphic
to $B$ as a $G$-algebra, and the induced morphism
$\BG_m\to\Out(B)$ is the same as the one coming from the original
grading on $B$  \cite[proof of Proposition 5.14]{Rou6}.
\end{proof}

We deduce from Propositions \ref{pr:starreduction}, \ref{pr:examplesS1S2},
\ref{pr:reductionS3} and \ref{pr:specialB} our main reduction theorem.

\begin{thm}
\label{th:mainreduction}
Let $\CC$ be given by contractible or acyclic complexes, with $R$-projective
components. Let $\CB$ be
one of the data defined above.

Let $\CF$ be a set of non-cyclic finite simple groups with non-trivial
abelian Sylow $\ell$-subgroups. Assume that, given $G\in\CF$, there
is a pair $(G,\tG)$  satisfying Brou\'e's $(\CB,\CC)$-conjecture with
the following property.
 Let $\hat{G}$ be a finite group containing
$\tG/O_{l'}(\tG)$ with $G\lhd\hat{G}$ and $\hat{G}/G$ a Hall
$\ell'$-subgroup of $\Out(G)$. Then, we require that
$N_{\hat{G}}(P)/C_{\hat{G}}(P)=N_{\tG}(P)/C_{\tG}(P)$, where $P$ is
a Sylow $\ell$-subgroup of $G$.

\smallskip
Let $G$ be a finite group with an abelian Sylow $\ell$-subgroup. 
Assume that all non-cyclic
composition factors of $G$ with order divisible by $\ell$ are in $\CF$.
Let $\tG$ be a finite group containing $G$ as a normal subgroup of $\ell'$-index.
Then $(G,\tG)$ satisfies Brou\'e's $(\CB,\CC)$-conjecture.
\end{thm}

\subsection{Brou\'e's conjecture for principal blocks}
\label{se:principal}

\begin{thm}
Let $G$ be a finite group with an abelian Sylow $\ell$-subgroup $P$, where
$\ell=2$ or $\ell=3$.

\begin{itemize}
\item If $\ell=2$ or ($\ell=3$ and $|P|\le 9$), then there is
a splendid Rickard equivalence between 
$B_0(G)$ and $B_0\bigl(N_G(P)\bigr)$.
\item If $|P|\le 9$, such an equivalence can be chosen to be a composition
of perverse equivalences.
\item If 
$|P|\le \ell^2$ and $P$ has no simple factor $\GA_6$ or $M_{22}$ (when
$\ell=3$), then the equivalence can be chosen to be a single
increasing perverse equivalence.
\end{itemize}
\end{thm}

The existence of a splendid increasing perverse Rickard equivalence
when $P$ is cyclic is already known \cite{ChRou2}. In the next two
sections, we show that the theorem holds by reduction
to simple groups, using Theorem \ref{th:mainreduction}.

\section{Defect $3\times 3$}
\label{se:defect33}

This section gives a combinatorial description of all principal blocks
of finite groups with Sylow $3$-subgroup $\Cyc3^2$, up to splendid Morita
equivalence, \ie, the source algebra is determined by the combinatorics.
The description is done in terms of perversity functions (both global
and local). Some blocks Morita equivalent to the local block
are given non-zero perversity function: in doing so, we try to
follow the precise form of the abelian defect group conjecture for finite groups
of Lie type.

\subsection{Local structure}
In this section we will collate the local information that we need to prove that the maps we obtain from the algorithm are really derived equivalences. This includes information on the centralizers of elements of order $3$, and on automizers of Sylow $3$-subgroups.

\subsubsection{Decompositions}
The structure of finite groups with elementary abelian Sylow $3$-subgroups of order $9$ is described in the following proposition
(cf.\ the proof of Proposition \ref{pr:starreduction}).

\begin{prop}
Let $G$ be a finite group with a Sylow $3$-subgroup $P$ isomorphic to
$\Cyc3^2$. Assume that $O_{3'}(G)=1$. Then
$O^{3'}(G)$ is simple or
$O^{3'}(G)=G_1\times G_2$ where $G_1$ and $G_2$ are simple groups with
Sylow $3$-subgroups of order $3$.
\end{prop}

When $O^{3'}(G)=G_1\times G_2$ in the proposition above, then
$B_0(G_i)$ is splendidly Morita equivalent
to the principal block of the normalizer of a Sylow $3$-subgroup
(cf.\ Remark \ref{re:Moritacyclic}) and we deduce that the same holds
for $G$. We describe below perversity functions for other
equivalences. 

\smallskip
Let $G$ be a finite simple group with Sylow $3$-subgroup $P$ of order $3$.
Then, $G$ is of one of the following types
\begin{itemize}
\item $J_1$;
\item $\PSL_2(q)$, $\PSL_2(r)$;
\item $\PSL_3(q)$;
\item $\PSU_3(r)$.
\end{itemize}
Here, $q\equiv 2,5 \pmod 9$ ($q>2$ for $\PSL_2(q)$) and $r\equiv 4,7\pmod 9$.

\smallskip
In all of those cases, $|N_G(P)/C_G(P)|=2$.
Denote by $k$ (resp.\ $\eps$) the trivial (resp.\ non-trivial) simple $kB_0(N_G(P))$-module.
There is a perverse equivalence
between the principal blocks of $G$ and $N_G(P)$ corresponding to the following
perversity functions:

$$\begin{array}{|c|cc|}
\cline{1-3}
  & k & \eps \\
\cline{1-3}
J_1       & 0 & 0\\
\PSL_2(q) & 0 & 1\\
\PSL_2(r) & 0 & 2\\
\PSL_3(q) & 0 & 3\\
\PSU_3(r) & 0 & 6\\
\cline{1-3}
\end{array}$$

\begin{rem}
Note that $\PSL_2(4)\simeq \PSL_2(5)$ and the Deligne--Lusztig theory provides
two different perversity functions.
\end{rem}

\subsubsection{Automizers}
\label{se:automizers}
 We begin by describing the automizers for almost simple groups with Sylow $3$-subgroups isomorphic to $\Cyc3^2$.

Let $\tilde{G}$ be an almost simple group, \ie, a finite group
whose derived subgroup $G$ is simple and is the unique minimal
non-trivial normal subgroup. Let $P$ be a Sylow $3$-subgroup of $G$.
Assume that $P\simeq \Cyc3^2$ and $3\nmid[\tilde{G}:G]$.
We list the almost simple groups $\tilde{G}$ modulo the equivalence
relation generated by $\tilde{G}\sim\tilde{H}$ if $[G,G]=[H,H]$,
$\tilde{H}\le\tilde{G}$ and
$\tilde{G}=\tilde{H}C_{\tilde{G}}(P)$.

\smallskip
The classification is
the following. We indicate first the simple groups and then the
almost simple ones, modulo equivalence.

\begin{itemize}
\item $\Cyc4$: $\GA_6$, $\GA_7$
\item $\Cyc8$: $\GA_6.2_2=\PGL_2(9)$
\item $Q_8$: $M_{22}$, $\PSU_3(q)$, $\PSL_3(r)$; $\GA_6.2_3=M_{10}$
\item $D_8$: $\GA_8$, $\PSp_4(q)$, $\PSp_4(r)$, $\PSL_4(q)$, $\PSL_5(q)$,
$\PSU_4(r)$, $\PSU_5(r)$;  $\GS_6$, $\GS_7$
\item $SD_{16}$: $M_{11}$, $M_{23}$, $HS$; $\GA_6.2^2=\Aut(\GA_6)$, $M_{22}.2$,
\begin{itemize}
\item $\PSp_4(q).\Cyc{2n}$ if $q=2^n$ and
$\PSp_4(r).\Cyc{2n}$ if $r=2^n$ (extension by the extraordinary graph
automorphism)
\item $\PSL_3(r).\Cyc{n}$ if $r=p^n$ and $p\equiv 2\pmod 3$ (extension by the
Frobenius automorphism over $\BF_p$)
\item $\PSU_3(q).\Cyc{n}$ if $q^2=p^n$ and $p\equiv 2\pmod 3$ (extension by the
Frobenius automorphism over $\BF_p$).
\end{itemize}
\end{itemize}

Here, $q\equiv 2,5\pmod 9$ $(q>2$ for $\PSU_3(q)$ and $\PSp_4(q)$)
and $r\equiv 4,7\pmod 9$, so that
$\ell\mid \Phi_1(r)$ and $\ell\mid \Phi_2(q)$. Note that $n$ is even in
the extended groups $\PSL_3(r).\Cyc{n}$ and $\PSU_3(q).\Cyc{n}$. Note 
finally that the square of the extraordinary graph automorphism is
the Frobenius over $\BF_2$.


\subsubsection{Centralizers of $3$-elements}
The structure of the centralizers of $3$-elements influences the local perversity functions,
for finite groups of Lie type. We provide here a description of
centralizers for those finite groups considered in \S \ref{se:automizers}.

\begin{prop} Let $\tilde{G}$ be one of the groups listed
in \S \ref{se:automizers} and assume that $G$ is an alternating group or is sporadic.
Let $x$ be an element of order $3$ in $\tilde G$.
We have 
$\Cent_{\tilde G}(x)=\langle x\rangle \times A$, where $A$ is given in the table
below.
The group $\Cent_{\tilde G}(x)$ is of the form $\Cyc3\times A$,
where $A$ is given in the table below. When there is more than
one conjugacy class of elements of order $3$, we list all possibilities.
$$\begin{array}{|c||c|c||c||c|c||c|c|c|}
\hline
\tilde{G} & \Alt6 & \GA_7 & \PGL_2(9) & M_{10} & M_{22} & 
\GS_6 & \GS_7 & \GA_8 \\
\hline
A & \Cyc3, \Cyc3 & \GA_4,\Cyc3 & \Cyc3 & \Cyc3 & \GA_4 & 
\GS_3,\GS_3 & \GS_3,\GS_4 & \GS_3,\GA_5\\
\hline
\end{array}$$

$$\begin{array}{|c||c|c|c|c|}
\hline
\tilde{G} & \Aut(\GA_6),M_{11} & M_{22}.2 & M_{23} & HS \\
\hline
A & \GS_3 & \GS_4 & \GA_5 & \GS_5 \\
\hline
\end{array}$$
\end{prop}

\begin{proof}The proof of this is trivial for the alternating groups and follows from the information in the Atlas for the sporadic groups \cite{Atl}.
\end{proof}

We now move on to the groups of Lie type, where we choose convenient
representatives for $\tilde{G}$ up to equivalence.

\begin{prop} Let $\tilde G$ be one of the groups of Lie type in
\S \ref{se:automizers}.
Let $x$ be an element of order $3$ in $\tilde G$. We have 
$\Cent_{\tilde G}(x)=\langle x\rangle \times A$, where $A$ is given in the table
below.  When there is more than
one conjugacy class of elements of order $3$, we list all possibilities.
\begin{center}\begin{tabular}{|c|c|}
\hline $\tilde G$ & $A$
\\ \hline $\PSL_3(r)$ & $(\Cyc{(r-1)/3})^2\rtimes \Cyc3$
\\ $\PSL_3(r).\Cyc{n}$ & $((\Cyc{(r-1)/3})^2\rtimes\Cyc3)\rtimes \Cyc{n}$
\\ $\PSU_3(q)$ & $(\Cyc{(q+1)/3})^2\rtimes \Cyc3$
\\ $\PSU_3(q).\Cyc{n}$ & $((\Cyc{(q+1)/3})^2\rtimes \Cyc3)\rtimes\Cyc{n}$
\\ $\Sp_4(q)$ & $\SL_2(q)\times\Cyc{(q+1)/3}$, $\GU_2(q)/\Cyc3$
\\ $\Sp_4(q).\Cyc{2n}$ & $(\SL_2(q)\times\Cyc{(q+1)/3})\rtimes\Cyc{n}$,
 $(\GU_2(q)/\Cyc3)\rtimes\Cyc{n}$
\\ $\Sp_4(r)$ & $\SL_2(r)\times\Cyc{(r-1)/3}$, $\GL_2(r)/\Cyc3$
\\ $\Sp_4(r).\Cyc{2n}$ & $(\SL_2(r)\times\Cyc{(r-1)/3})\rtimes\Cyc{n}$,
 $(\GL_2(r)/\Cyc3)\rtimes\Cyc{n}$
\\ $\GL_4(q)$ & $\GL_2(q)\times \Cyc{(q^2-1)/3}$, $\GL_2(q^2)/\Cyc3$
\\ $\GU_4(r)$ & $\GU_2(r)\times\Cyc{(r^2-1)/3}$, $\GL_2(r^2)/\Cyc3$
\\ $\GL_5(q)$ & $\GL_3(q)\times \Cyc{(q^2-1)/3}$,
 $\Cyc{q-1}\times\GL_2(q^2)/\Cyc3$
\\ $\GU_5(r)$ & $\GU_3(r)\times\Cyc{(r^2-1)/3}$,
 $\Cyc{r+1}\times\GL_2(r^2)/\Cyc3$.
\\ \hline
\end{tabular}\end{center}
\end{prop}
\begin{proof}
These descriptions are well known. Here are some references
\begin{enumerate}
\item $\PSL_3(r)$: see \cite[Lemma 3.1]{kunugi2000}.
\item $\PSU_3(q)$: see \cite[Lemma 2.5]{koshitanikunugi2001}.
\item $\Sp_4(q)$ ($\Sp_4(r)$ is similar): see \cite[Example 3.6]{Ok1}
\item $\GL_4(q)$: see \cite[Lemma 2.2]{koshitanimiyachi2000}.
\item $\GU_4(r)$: see the proof of \cite[Lemma 2.2]{koshitanimiyachi2001}.
\item $\GL_5(q)$: see \cite[Lemma 2.6]{koshitanimiyachi2000}.
\item $\GU_5(r)$: see the proof of \cite[Lemma 2.2]{koshitanimiyachi2001}.
\end{enumerate}
\end{proof}

%

\subsubsection{Reductions for Lie type}
We now reduce the number of groups that need to be checked to finitely many
by using splendid Morita equivalences between groups of a given Lie type.
\begin{itemize}
\item For $G=\PSL_3(r)$, in \cite[Theorem 1.2]{kunugi2000} it is shown that 
the principal block of $G$ is splendidly Morita equivalent to
that of $\PSL_3(4)$.
\item For $G=\PSU_3(q)$, in \cite[Theorem 0.2]{koshitanikunugi2001} it is shown that the principal block of $G$ is splendidly Morita
equivalent to that of $\PSU_3(2)=\Cyc3^2\rtimes Q_8$.
\item For $G=\PSL_4(q)$, in \cite[Theorem 0.3]{koshitanimiyachi2000} it is shown
that the principal block of $G$ is splendidly Morita equivalent to
that of $\PSL_4(2)$.
\item For $G=\PSL_5(q)$, in \cite[Theorem 0.2]{koshitanimiyachi2000}  it is shown
 that the principal block of $G$ is splendidly Morita equivalent 
to that of $\PSL_5(2)$ (and also to that of $N$).
\item For $G=\PSp_4(q)$, in \cite[Example 3.6]{Ok1}
it is shown that the principal block of $G$ is splendidly Morita equivalent 
to that of $\PSp_4(2)$.
\item  The principal $3$-blocks of $\PSp_4(r)$, $\PSU_4(r)$, and $\PSU_5(r)$
are splendidly Morita to those corresponding to $r=4$ (cf.\ Remark
\ref{re:changeq}).
\end{itemize}

Hence the simple
groups that we have to analyze are $\PSL_3(4)$, $\PSU_3(2)$, $\PSL_4(2)$, $\PSL_5(2)$, $\PSU_4(4)$, $\PSU_5(4)$, $\PSp_4(2)$ and $\PSp_4(4)$ of Lie type, and
$\GA_6$, $\GA_7$, $M_{11}$, $M_{22}$, $M_{23}$ and $HS$.

\subsection{Results of the algorithm}

In the next four subsections we will describe the result of the algorithm 
(see \S \ref{se:strategy}) on
various simple (and in two cases almost simple) groups $G$ with
a Sylow $3$-subgroup $P$ isomorphic to $\Cyc3^2$.
These are divided according to the
automizer $E=N_G(P)/C_G(P)$,
in the order $\Cyc4$, $Q_8$, $D_8$ and finally $SD_{16}$.

We assume that $K$ is big enough for the finite groups considered.
Each
section will follow the same template for automizer $E$: we start by giving
information on the group $P\rtimes E$, in particular its
simple modules and the radical series for trivial source modules (including
projectives). We write$N=N_G(P)$. Note that the canonical isomorphism
$N/O_{3'}(N)\iso P\rtimes E$ gives an isomorphism of algebras
$kB_0(N)\iso k(P\rtimes E)$ and we will freely identify modules
in the principal blocks of $N$ with $k(P\rtimes E)$-modules. We list in a
table the perversity functions and the local twists (local perversity
functions) as a summary of the results to be described in the subsections.
Note that a row of the table determines the block up to a splendid
Morita equivalence. This applies as well for composite perverse
equivalences.

\smallskip
In the subsections we examine each (almost) simple group in
turn, describing first the simple modules and Green correspondents, then
giving the perversity function $\per$ together with the decomposition matrix.
The Green correspondents are known in Lie type when ``$\ell\mid(q-1)$''
(Theorem \ref{th:Puig}).
In the other cases, they can be determined by a computer by
constructing the simple modules and decomposing their restriction.
 The identification
of $N$ with $\Cyc3^2\rtimes \Cyc4,\ldots,\Cyc3^2\rtimes SD_{16}$ is
not canonical. The choice we make affects the description of the
Green correspondents $C_i$.
When $G=G(q_0)$ is a finite group of Lie type, we 
provide the generic degree of the irreducible characters, a polynomial in
$q$ that specializes to the actual degree for $q=q_0$.

We give the decomposition matrix of $B_0(G)$ in an upper triangular form (in
some cases, we only provide the upper square part). This
gives rise to a basic set of ``unipotent characters'' $\{\chi_i\}_i$
in bijection with
simple modules (we always choose $\chi_i=1$). They agree with the
unipotent characters for Lie type, except for $\PSU_3(r)$ and
$\PSL_3(q)$, where we need a different (and larger) set.
Our numbering of simple modules gives an implicit bijection
between simple $B_0(G)$-modules $S_i$ and simple $B_0(N)$-modules $T_i$.
We construct the images $X_i$ in $B_0(N)$ of the simple modules under the
perverse equivalence determined by $\per$.
We give explicit descriptions of the complexes $X_i$
in all cases where feasible, and when they are not simple, \ie, when
$\per(i)\not=0$. We describe
the cohomology of the complexes $X_i$ in table form. Write
$[X_i]=\sum_j a_{ij}[T_j]$ in $K_0(kN)$. We have
$[S_i]=\sum_j (-1)^{\per(j)}a_{ij}d(\chi_j)$, where
$d:K_0(KG)\to K_0(kG)$ is the decomposition map. We indicate 
$\sum_j (-1)^{\per(j)}a_{ij}\chi_j$ in the table. This explains how
the classes $[X_i]$ determine the decomposition matrix of $B_0(G)$.

The last ingredient
in the construction is a twist of the stable equivalence between
$B_0(G)$ and $B_0(N)$. This is determined by functions
$\eta_R$, where $R$ runs over $N$-conjugacy classes of subgroups
of order $3$ of $P$. This twist is only needed in some cases
where $E=D_8$ or $E=SD_{16}$. We determine the images $Y_i$ of the modules
$S_i$ under the twisted stable equivalence, following \S \ref{se:localtwists}.
The maps in the complexes are uniquely determined, as we are in the
setting described in Remark \ref{re:onedimensional}. Note that in some
cases we have been unable to find perverse equivalences without introducing
local perverse twists.

We show that
$X_i$ is isomorphic to $Y_i$ in the stable category, as
needed to obtain a perverse equivalence between $kB_0(G)$ and
$kB_0(N)$ lifting the twisted stable equivalence (cf.\ \S \ref{se:strategy}).

In all cases it is immediate to check that the perversity function
$\per$ is invariant under the action of field automorphisms and under
the action of the outer automorphism group of $G$. Whenever this automorphism
group is non-trivial, this follows from the fact that the function $\per$
takes the same value on non-trivial ``unipotent'' characters of the same
degree. Also, the twisted stable equivalences are invariant under the
outer automorphism group of $G$.
This enables us to use Corollary \ref{co:liftingperv} and the descent
method of \S \ref{se:recognition}.

In the last paragraph \S \ref{se:nonprincipal}, we give an example
of a non-principal block.

\smallskip
Let us introduce some more notation.
The trivial modules are labelled $S_1$ and $T_1$. Often when describing the structure of a $kN$-module we will abbreviate $T_i$ to $i$, and to save space we denote the radical layers by $/$, so a module with $T_1$ in the head and $T_2$ in the second radical layer would be described as $1/2$. The projective cover of $T_i$ will be
denoted by $\Proj i$.

\begin{rem}
Okuyama has constructed derived equivalences for all blocks of simple groups
with Sylow $3$-subgroup $\Cyc3^2$ \cite{Ok1}.
Note that the equivalences in \cite{Ok1} are all compositions of
perverse equivalences. If the subsets $I_0,\ldots,I_r$ used
by Okuyama are nested (ie, $I_l\subset I_m$ or $I_m\subset I_l$ for
all $l,m$), then the composition itself is perverse.
\end{rem}

\subsection{Automizer $\Cyc4$}

Over an algebraically closed field of characteristic $3$, the group $\Cyc3^2\rtimes \Cyc4$ has four simple, $1$-dimensional modules, but over $\F_3$ (or any other field without a fourth root of unity) it possesses only three. We denote the two $1$-dimensional modules over $\BF_3$ by $T_1$ and $T_2$, and the $2$-dimensional simple (but not absolutely simple) module by $T_3$, which over $\F_9$ splits into $T_{3,1}$ and $T_{3,2}$. 

\[ \Proj1=\begin{array}{c}1\\3\\122\\3\\1\end{array},\quad \Proj2=\begin{array}{c}2\\3\\112\\3\\2\end{array},\quad \Proj3=\begin{array}{c}3\\1122\\333\\1122\\3\end{array}.\]

In this section we will prove that the perverse form of Brou\'e's
conjecture holds for $\Alt7$, but not for $\Alt6$. Instead we set up a
perverse equivalence from $\Alt6$ to $\Alt7$. (Note that if $\tilde G$ is an extension of $\Sym6$, rather that just $\Alt6$, then the perverse form of
Brou\'e's conjecture \emph{does} hold for $\tilde G$, as we shall see in the
next section.)

The perversity function is given in the following table.

$$\begin{array}{|c|ccc|}
\cline{1-4}
  & T_1 & T_2 & T_3 \\
\cline{1-4}
\GA_7 & 0 & 1 & 0\\
\cline{1-4}
\end{array}$$

In addition, there is a $2$-step perverse equivalence for $\GA_6$ 
$$kB_0(\GA_6)\xrightarrow{(0,0,1)} kB_0(\GA_7) \xrightarrow{(0,1,0)}
kB_0(N).$$

\subsubsection{The alternating group $\Alt7$}
\label{sec:A7}

Let $G=\Alt7$.

\smallskip
\paragraph{Simple modules}

There are four simple modules in $kB_0(G)$, of dimensions $1$, $10$, $10$ and $13$. Over $\F_3$, the two $10$-dimensional modules $S_{3,1}$ and $S_{3,2}$
amalgamate into a $20$-dimensional module $S_3$. Write $S_2$ for the $13$-dimensional simple module. The Green correspondents are
\[ C_1=1,\quad C_2=\begin{array}{c}2\\3\\2\end{array},\quad C_3=3.\]

\paragraph{The perverse equivalence}

There is a perverse equivalence between $kB_0(G)$ and $kB_0(N)$ with the $\per$-values on the left, which makes the decomposition matrix look as follows. Note
that this equivalence has been constructed already by Okuyama
\cite[Example 4.1]{Ok1}.

\begin{center}\begin{tabular}{|c|c|cccc|}
\hline $\per$ & Ord.\ Char.{} & $S_1$ & $S_{3,1}$ & $S_{3,2}$ & $S_2$
\\ \hline
   $0$&$1$ &1& & &
\\ $0$&$10$& &1& &
\\ $0$&$10$& & &1&
\\ $1$&$35$&2&1&1&1
\\\hline
      &$14$&1& & &1
\\    &$14$&1& & &1
\\ \hline \end{tabular}\end{center}

We have
\[ X_2:\;\;0\to \Proj2\to C_2\to 0,\]
and $H^{-1}(X_2)=\Omega^{-1}(C_2)=11/3/2$.

\subsubsection{The alternating group $\Alt6$}
\label{sec:A6}

Let $G=\Alt6$.

\smallskip
\paragraph{Simple modules}

There are four simple modules in $kB_0(G)$, of dimensions $1$, $3$, $3$ and $4$. Over $\F_3$, the two $3$-dimensional modules amalgamate into a $6$-dimensional module $S_3$. We label the $4$-dimensional simple module $S_2$.

We construct a perverse equivalence between the principal blocks
of $G$ and $H=\Alt7$
lifting the stable equivalence given by induction and restriction.
We label the three simple $\F_3 H$-modules in the principal block by $U_1$,
$U_2$ and $U_3$, with the ordering taken from Section \ref{sec:A7}. The images
$C_i$ of $S_i$ in $H$ under the stable equivalence are
\[ C_1=U_1,\quad C_2=U_2,\quad C_3=\begin{array}{c}U_3\\U_1\oplus U_1\\U_3\end{array}.\]

\paragraph{The perverse equivalence}

We set $\per(U_1)=\per(U_2)=0$ and $\per(U_3)=1$.
Let $X_i$ be the complex for $H$
image of the $i$-th simple module for the corresponding perverse equivalence.
We have $X_1=C_1$ and $X_2=C_2$. The structure of $\Proj{U_3}$ is
\[ \Proj{U_3}=\begin{array}{c}U_3\\U_1\oplus U_1\\U_2\oplus U_3\oplus U_2\\U_1\oplus U_1\\U_3\end{array}.\]
We deduce that $X_3=(0\to \Proj{U_3}\to C_3\to 0)$, with cohomology
$$H^{-1}(X_3)=(U_2\oplus U_2)/(U_1\oplus U_1)/U_3.$$
It satisfies the conditions of the algorithm, so produces a perverse
equivalence. This equivalence has been constructed by Okuyama
\cite[Example 4.2]{Ok1}.

\smallskip
\paragraph{Outer automorphisms}

The group $\Out(\Alt6)$ has order $4$, with three order-$2$ extensions,
yielding the groups $\Sym6$, $\PGL_2(9)$ and $M_{10}$
(the one-point stabilizer of the Mathieu group $M_{11}$).
In Section \ref{sec:S6} we will provide a perverse equivalence between the
principal block of $\Sym6$ and the principal block of
its normalizer, and this will be compatible
with the outer automorphism of $\Sym6$.
For the other two extensions $\PGL_2(9)$ and $M_{10}$, the decomposition
matrices are not triangular, and so there can be no perverse equivalence for
their principal blocks. However, since both of the equivalences
$D^b(kB_0(G))\iso D^b(kB_0(H))$ (from above) and
$D^b(kB_0(H))\iso D^b(kB_0(N))$ (from Section \ref{sec:A7})
are compatible with exchanging the two simple modules defined over $\F_9$,
the derived equivalence obtained by composing these two perverse equivalences
will extend to both $\PGL_2(9)$ and $M_{10}$.

\begin{rem}
While the decomposition matrix of $B_0(G)$ is triangular, the fact that
the principal block of $\PGL_2(9)$ has a non-triangular decomposition matrix
means that there can be no perverse equivalence between $kB_0(G)$ and
$kB_0(N)$.
Indeed, all standard stable equivalences are compatible with the action
of $\PGL_2(9)/G$ by Remark \ref{re:fixedpointfree}, hence all perverse equivalences 
extend to $\PGL_2(9)$
since the two modules of dimension $3$ in $B_0(G)$ are fixed by $\PGL_2(9)$
(cf.\ Corollary \ref{co:liftingperv}).
\end{rem}

\subsection{Automizer $Q_8$}

For the group $\Cyc3^2\rtimes Q_8$, there are five simple $\F_3$-modules, all absolutely simple. The first four are $1$-dimensional, and the last is $2$-dimensional. The three non-trivial $1$-dimensional module are permuted transitively by the $\Sym3$-group of outer automorphisms of $Q_8$. The projective indecomposable modules are as follows.
\[ \Proj1=\begin{array}{c}1\\5\\234\\5\\1\end{array},\quad \Proj2=\begin{array}{c}2\\5\\134\\5\\2\end{array},\quad \Proj3=\begin{array}{c}3\\5\\124\\5\\3\end{array},\quad \Proj4=\begin{array}{c}4\\5\\123\\5\\4\end{array},\quad \Proj5=\begin{array}{c}5\\1234\\555\\1234\\5\end{array}\]

The perversity functions are given in the following table.

$$\begin{array}{|c|ccccc|}
\cline{1-6}
  & T_1 & T_2 & T_3 & T_4 & T_5\\
\cline{1-6}
\PSL_3(r) & 0 & 3 & 3 & 3 & 2\\
\PSU_3(q) & 0 & 6 & 6 & 6 & 4\\
\cline{1-6}
\end{array}$$

\smallskip
In addition, there is a $2$-step perverse equivalence for $M_{22}$ and
for $M_{10}$
given by the composite $\per$-values $(0,1,1,0,0),(0,0,0,0,1)$.
$$kB_0(M_{22})\sim kB_0(M_{10})\xrightarrow{(0,0,0,0,1)} \bullet
\xrightarrow{(0,1,1,0,0)} kB_0(N).$$

\subsubsection{The group $\PSL_3(4)$}

Let $G=\PSL_3(4)$. Okuyama has shown that the principal blocks of $G$ and $N$
are derived equivalent \cite[Example 4.6]{Ok1}. In this section, we
produce a perverse equivalence.

\smallskip
\paragraph{Simple modules}

There are five simple modules in the principal $3$-block, of dimensions $1$,
$15$, $15$, $15$, and $19$. We label the $19$-dimensional module $S_5$, and
the three $15$-dimensional simple modules $S_2$ to $S_4$. There is an
$\Sym3$-group of outer automorphisms that permutes transitively $S_2$, $S_3$
and $S_4$. We choose the $S_i$ so that the Green correspondents $C_i$ are
\[ C_1=1,\quad C_2=\begin{array}{c}5\\12\\5\end{array},\quad C_3=\begin{array}{c}5\\13\\5\end{array},\quad C_4=\begin{array}{c}5\\14\\5\end{array},\quad C_5=\begin{array}{c}234\\55\\234\end{array}.\]

\paragraph{The perverse equivalence}

There is a perverse equivalence between $kB_0(G)$ and $kB_0(N)$ with the $\per$-values on the left, which makes the decomposition matrix look as follows. We 
provide the generic degree for the corresponding irreducible characters
of $\PSL_3(q)$.

\begin{center}\begin{tabular}{|c|c|ccccc|}
\hline $\per$ & Ord.\ Char & $S_1$ & $S_5$ & $S_2$ & $S_3$ & $S_4$
\\ \hline $0$ & $1$ & 1&&&&
\\ $2$ & $q(q+1)$ & 1&1&&&
\\ $3$ & $(q+1)(q^2+q+1)/3$ & 1&1&1&& 
\\ $3$ & $(q+1)(q^2+q+1)/3$ & 1&1&&1&
\\ $3$ & $(q+1)(q^2+q+1)/3$ & 1&1&&&1
\\ \hline & $q^3$ &&1&1&1&1
\\ \hline\end{tabular}\end{center}

\smallskip
The explicit complexes are as follows.
\[\begin{array}{lr}
   X_5:&\;\; 0\to \Proj5\to\Proj{234}\to C_5\to 0.
\\ X_2:&\;\; 0\to \Proj2\to\Proj{34}\to\Proj5 \to C_2\to 0.
\\ X_3:&\;\; 0\to \Proj3\to\Proj{24}\to\Proj5\to C_3\to 0.
\\ X_4:&\;\; 0\to \Proj4\to\Proj{23}\to\Proj5\to C_4\to 0.
\end{array}\]

The cohomology of the complexes above is displayed in the following table.

\smallskip
\begin{center}\begin{tabular}{|c|ccc|c|}
\hline $X_i$ & $H^{-3}$ & $H^{-2}$ & $H^{-1}$ & Total
\\ \hline $5$&&$1/5$&$11$&$5-1$
\\ $2$&$1/5/2$ &$1$ & &$2-5$
\\ $3$&$1/5/3$ &$1$ & &$3-5$
\\ $4$&$1/5/4$ &$1$ & &$4-5$
\\ \hline \end{tabular}\end{center}

\subsubsection{The group $\PSU_3(2)$}

Let $G=\PSU_3(2)=\Cyc3^2\rtimes Q_8$.

We construct a self-perverse equivalence of $kG$ with the $\per$-values on the
left, which makes the decomposition matrix look as follows.
 
\smallskip
\begin{center}\begin{tabular}{|c|c|ccccc|}
\hline $\per$ & Ord.\ Char & $S_1$ & $S_5$ & $S_2$ & $S_3$ & $S_4$
\\ \hline $0$ & $1$ & 1&&&&
\\ $4$ & $q(q-1)$   & &1&&&
\\ $6$ & $(q-1)(q^2-q+1)/3$ & &&1&& 
\\ $6$ & $(q-1)(q^2-q+1)/3$ & &&&1&
\\ $6$ & $(q-1)(q^2-q+1)/3$ & &&&&1
\\ \hline & $q^3$ &1&2&1&1&1
\\ \hline\end{tabular}\end{center}

\smallskip
The explicit complexes are as follows.
\[\begin{array}{lr}
   X_5:&\;\; 0\to \Proj5\to \Proj{234}\to \Proj{234}\to \Proj5\to C_5\to 0.
\\ X_2:&\;\; 0\to \Proj2\to \Proj{34}\to \Proj{234}\to\Proj{25}\to\Proj2\to C_2\to 0.
\\ X_3:&\;\; 0\to \Proj3\to \Proj{24}\to \Proj{234}\to\Proj{35}\to\Proj3\to C_3\to 0.
\\ X_4:&\;\; 0\to \Proj4\to \Proj{23}\to \Proj{234}\to\Proj{45}\to\Proj4\to C_4\to 0.\end{array}\]

The cohomology of the complexes above is displayed in the following table.

\smallskip
\begin{center}\begin{tabular}{|c|cccccc|c|}
\hline $X_i$ & $H^{-6}$ & $H^{-5}$ & $H^{-4}$ & $H^{-3}$ & $H^{-2}$ & $H^{-1}$ & Total
\\ \hline $5$&&&$1/5$&$11$&$11$&$1$&$5$
\\ $2$&$1/5/2$ &$1\oplus 1/5$ &$11$ &$1$&&&$2$
\\ $3$&$1/5/3$ &$1\oplus 1/5$ &$11$ &$1$&&&$3$
\\ $4$&$1/5/4$ &$1\oplus 1/5$ &$11$ &$1$&&&$4$
\\ \hline \end{tabular}\end{center}

\begin{rem}
Note that the perversity function for $\PSU_3(r)$ is twice that
of $\PSL_3(q)$, after identification of the ``unipotent characters''. This
is our reason for providing this perverse equivalence, instead of the
identity.
\end{rem}

\subsubsection{The Mathieu group $M_{22}$}
\label{sec:M22}

Let $G=M_{22}$. By  \cite[Example 4.5]{Ok1}, there is a splendid Morita
equivalence between $kB_0(G)$ and $kB_0(H)$, where $H=M_{10}$. There is an
embedding of $H$ inside $G$ so that $N_H(P)=N$.
The composition of the splendid Morita equivalence from $kB_0(G)$ to $kB_0(H)$
and the derived equivalence $kB_0(H)\to kB_0(N)$ from Section \ref{sec:A6},
yields a derived equivalence $kB_0(G)\to kB_0(N)$.
There can be no perverse equivalence between $kB_0(G)$ and $kB_0(N)$ since
there is none between $kB_0(H)$ and $kB_0(N)$.

\smallskip
\paragraph{Outer automorphisms}

There is an outer automorphism of $M_{22}$, and this will be discussed in Section \ref{sec:M22.2}. Note that the splendid Morita equivalence from $kB_0(G)$ to
$kB_0(H)$ extends to one between $kB_0(M_{22}.2)$ and $kB_0(\GA_6.2^2)$.

\subsection{Automizer $D_8$}

For the group $N=\Cyc3^2\rtimes D_8$, there are five simple $\F_3$-modules, all absolutely simple. The first four are $1$-dimensional, and the last is $2$-dimensional. We denote by $T_4$ the exterior square of the $2$-dimensional module.
The modules $T_2$ and $T_3$ are permuted by the outer automorphism of $D_8$.
There are two Klein four subgroups lying in $D_8$. One acts trivially on $T_2$
and not on $T_3$, while the other acts trivially on $T_3$ and not on $T_2$.
The projective indecomposable modules are as follows.
\[ \Proj1=\begin{array}{c}1\\5\\123\\5\\1\end{array},\quad \Proj2=\begin{array}{c}2\\5\\124\\5\\2\end{array},\quad \Proj3=\begin{array}{c}3\\5\\134\\5\\3\end{array},\quad \Proj4=\begin{array}{c}4\\5\\234\\5\\4\end{array},\quad \Proj5=\begin{array}{c}5\\1234\\555\\1234\\5\end{array}.\]
We also need relative projective modules. There are two conjugacy classes of subgroups of order $3$ in $N$, with representatives $Q_1$ and $Q_2$.
We denote by $M_{1,j},\ldots,M_{4,j}$ the indecomposable summands
of the permutation module $\Ind_{Q_i}^N k$.

\[ M_{1,1}=\begin{array}{c}5\\12\\5\end{array},\quad M_{2,1}=\begin{array}{c}5\\34\\5\end{array},\quad M_{3,1}=\begin{array}{c}12\\5\\12\end{array},\quad M_{4,1}=\begin{array}{c}34\\5\\34\end{array},\]
\[ M_{1,2}=\begin{array}{c}5\\13\\5\end{array},\quad M_{2,2}=\begin{array}{c}5\\24\\5\end{array},\quad M_{3,2}=\begin{array}{c}13\\5\\13\end{array},\quad M_{4,2}=\begin{array}{c}24\\5\\24\end{array}.\]

The perversity and local twist functions are given in the following table.
$$\begin{array}{|c|ccccc|cc|}
\cline{1-8}
  & T_1 & T_2 & T_3 & T_4 & T_5 & \eta(Q_1) & \eta(Q_2)\\
\cline{1-8}
\GS_7           & 0 & 1 & 1 & 0 & 0 & 0 & 0\\
\PSp_4(q),\GS_6 & 0 & 3 & 3 & 4 & 3 & 1 & 1\\
\PSp_4(r)       & 0 & 6 & 6 & 8 & 6 & 2 & 2\\
\PSL_4(q),\GA_8 & 0 & 3 & 5 & 6 & 4 & 2 & 1\\
\PSU_4(r)       & 0 & 6 & 10& 12& 8 & 4 & 2\\
\PSL_5(q)       & 0 & 6 & 8 & 10& 7 & 3 & 2\\
\PSU_5(r)       & 0 & 12& 16& 20&14 & 6 & 4\\
\cline{1-8}
\end{array}$$

In addition, there is a $2$-step perverse equivalence for $\PGL_2(9)=\GA_6.2_2$
given by the composite $\per$-values $(0,1,1,0,0),(0,0,0,0,1)$ (and 
$\eta=0$).

\subsubsection{The group $\PSp_4(2)=\GS_6$}
\label{sec:S6}

Let $G=\PSp_4(2)=\GS_6$.

\smallskip
\paragraph{Simple modules}

There are five simple modules in the principal $3$-block, of dimensions $1$, $1$, $4$, $4$, and $6$. We label the non-trivial $1$-dimensional module $S_5$, the $6$-dimensional simple module $S_4$, and the two $4$-dimensional simple modules $S_2$ and $S_3$. There is an outer automorphism that swaps $S_2$ and $S_3$ and
we choose the labelling so that the Green correspondents $C_i$ are
\[ C_1=1,\quad C_2=\begin{array}{c}3\\5\\3\end{array},\quad C_3=\begin{array}{c}2\\5\\2\end{array},\quad C_4=\begin{array}{c}5\\14\\5\end{array},\quad C_5=4.\]

\paragraph{The perverse equivalence}

There is a perverse equivalence between $kB_0(G)$ and $kB_0(N)$ with
local twist $\eta_{Q_1}=\eta_{Q_2}=1$ and with
the $\per$-values on the left, which makes the decomposition matrix look as follows.
\begin{center}\begin{tabular}{|c|c|ccccc|}
\hline $\per$ & Ord.\ Char & $S_1$ & $S_5$ & $S_2$ & $S_3$ & $S_4$
\\ \hline $0$ & $1$ & 1&&&&
\\ $3$ & $q(q-1)^2/2$ & &1&&& 
\\ $3$ & $q(q^2+1)/2$ & 1&&1&&
\\ $3$ & $q(q^2+1)/2$ & 1&&&1&
\\ $4$ & $q^4$ & 1&1&1&1&1
\\ \hline & $(q-1)(q^2+1)$&&1&&1&
\\  &$(q-1)(q^2+1)$&&1&1&&
\\  &$q(q-1)(q^2+1)$&&&1&&1
\\  &$q(q-1)(q^2+1)$&&&&1&1
\\ \hline\end{tabular}\end{center}

The explicit complexes are as follows.
\[\begin{array}{lr}
   X_5:&\;\; 0\to \Proj5\to \Proj{234}\to M_{4,1}\oplus M_{4,2}\to C_5\to 0.
\\ X_2:&\;\; 0\to \Proj2\to\Proj5\to \Proj3\oplus M_{1,2}\to C_2\to 0.
\\ X_3:&\;\; 0\to \Proj3\to \Proj5\to\Proj2\oplus M_{1,1}\to C_3\to 0.
\\ X_4:&\;\; 0\to \Proj4\to \Proj4\to\Proj{23}\to \Proj5\to C_4\to 0.\end{array}\]
The cohomology of the complexes above is displayed in the following table.
\begin{center}\begin{tabular}{|c|cccc|c|}
\hline $X_i$ & $H^{-4}$ & $H^{-3}$ & $H^{-2}$ & $H^{-1}$ & Total
\\ \hline $5$&&$1/5$&$1$&&$5$
\\ $2$&&$2$&&$1$&$2-1$
\\ $3$&&$3$&&$1$&$3-1$
\\ $4$&$23/5/4$&&$1$&&$4+1-2-3-5$
\\ \hline \end{tabular}\end{center}

\subsubsection{The group $\PSp_4(4)$}

Let $G=\PSp_4(4)$.

\smallskip
\paragraph{Simple modules}

There are five simple modules in the principal $3$-block, of dimensions $1$, $34$, $34$, $50$, and $256$. We label the $50$-dimensional module $S_4$ and the $256$-dimensional module $S_5$. The 
two $34$-dimensional modules $S_2$ and $S_3$ are permuted by
an outer automorphism and we choose the $S_i$ so that the Green correspondents
are $C_i=T_i$.

\smallskip
\paragraph{The perverse equivalence}

There is a perverse equivalence between $kB_0(G)$ and $kB_0(N)$ with local
twist $\eta_{Q_1}=\eta_{Q_2}=2$ and
 the $\per$-values on the left, which makes the decomposition matrix look as follows.

\begin{center}\begin{tabular}{|c|c|ccccc|}
\hline $\per$ & Ord.\ Char & $S_1$ & $S_5$ & $S_2$ & $S_3$ & $S_4$
\\ \hline $0$ & $1$ & 1&&&&
\\ $6$ & $q(q+1)^2/2$ & &1&&& 
\\ $6$ & $q(q^2+1)/2$ & &&1&&
\\ $6$ & $q(q^2+1)/2$ & &&&1&
\\ $8$ & $q^4$ & &&&&1
\\ \hline &$(q+1)(q^2+1)$ & 1&1&1&&
\\ &$(q+1)(q^2+1)$ & 1&1&&1&
\\ &$q(q+1)(q^2+1)$ & &1&&1&1
\\ &$q(q+1)(q^2+1)$ & &1&1&&1
\\ \hline\end{tabular}\end{center}

\smallskip
The explicit complexes are as follows.
\begin{footnotesize}\[\begin{array}{lr}
   X_2:&\;\; \Proj2\to \Proj{5}\to \Proj{35}\to \Proj{234}\to\Proj4\oplus M_{4,2}\to M_{4,2}\to C_2.
\\ X_3:&\;\; \Proj3\to \Proj{5}\to \Proj{25}\to \Proj{234}\to\Proj4\oplus M_{4,1}\to M_{4,1}\to C_3.
\\ X_5:&\;\; \Proj{234}\to\Proj{2344}\to \Proj{455}\to \Proj5\oplus M_{1,1}\oplus M_{1,2}\to M_{1,1}\oplus M_{1,2}\to C_5.
\\ X_4:&\;\; \Proj4\to \Proj4\to\Proj4\to\Proj5\to\Proj{55}\to\Proj{23445}\to\Proj{234}\oplus M_{4,1}\oplus M_{4,2}\to M_{4,1}\oplus M_{4,2}\to C_4.
\end{array}\]\end{footnotesize}
The cohomology of the complexes above is displayed in the following table.

\smallskip
\begin{center}\begin{tabular}{|c|cccccccc|c|}
\hline $X_i$ & $H^{-8}$& $H^{-7}$ & $H^{-6}$ & $H^{-5}$ & $H^{-4}$ & $H^{-3}$ & $H^{-2}$ & $H^{-1}$ & Total
\\ \hline $2$&&&$2$&&$1$&$1$&&&$2$
\\  $3$&&&$3$&&$1$&$1$&&&$3$
\\ $5$&&&$1/5$&$1$&&&$1$&$1$&$5$
\\ $4$&$23/5/4$&$23/5$&&$1$&$1$&$1$&$1$&&$4$
\\ \hline \end{tabular}\end{center}

\begin{rem}
This is the first instance where the relatively projective modules needed for the stable equivalence are not concentrated in a single degree. We will see this phenomenon occur for $\PSL_4(2)$, $\PSU_4(4)$, $\PSL_5(2)$ and $\PSU_5(4)$ below. If one merely runs the algorithm to find a perverse equivalence, the extra factor needed in degree $-1$ is a module \emph{filtered} by relative projective modules, but not itself relative projective. In the cases above, the (indecomposable) module obtained by shifting each relatively projective module from degree $-2$ into degree $-1$ is of dimension $18$, so cannot be relatively projective without being projective. In the next three cases we will have to add relatively
projective modules in degrees down to $-6$ in the case of $\PSU_5(4)$.
\end{rem}

\subsubsection{The group $\PSL_4(2)=\Alt8$}
\label{sec:A8}

Let $G=\PSL_4(2)=\GA_8$. There is an easy perverse equivalence constructed
by Okuyama \cite[Example 4.3]{Ok1}: the perversity function vanishes on
all simple modules except one, where the $\per$-value is $1$.
However, this is not compatible with the Deligne--Lusztig theory for $\PSL_4(q)$, and we provide a different perverse equivalence.

\smallskip
\paragraph{Simple modules}

There are five simple modules in the principal $3$-block, of dimensions $1$, $7$, $13$, $28$ and $35$. We label the modules $S_1$ to $S_5$ so that they have dimensions $1$, $13$, $35$, $28$ and $7$ respectively. The Green correspondents
$C_i$ are
\[ C_1=1,\quad C_2=\begin{array}{c}3\\5\\3\end{array},\quad C_3=5,\quad C_4=2,\quad C_5=4.\]

\paragraph{The perverse equivalence}

There is a perverse equivalence between $kB_0(G)$ and $kB_0(N)$ with 
local twist $\eta_{Q_1}=2$ and $\eta_{Q_2}=1$ and with the $\per$-values on the left, which makes the decomposition matrix look as follows. 

\begin{center}\begin{tabular}{|c|c|ccccc|}
\hline $\per$ & Ord.\ Char.{} & $S_1$&$S_2$&$S_5$ & $S_3$ & $S_4$
\\ \hline $0$ & $1$ &  1&&&&
\\ $3$ & $q(q^2+q+1)$ &1&1&&&
\\ $4$ & $q^2(q^2+1)$ &&1&1&&
\\ $5$ &$q^3(q^2+q+1)$&1&1&1&1&
\\ $6$ & $q^6$        &1&&&1&1
\\ \hline & $(q-1)^2(q^2+q+1)$&&&1&&
\\ &$q^2(q-1)^2(q^2+q+1)$&&&&&1
\\ &$(q-1)(q^2+1)(q^2+q+1)$&&&&1&
\\ &$q(q-1)(q^2+1)(q^2+q+1)$&&&1&1&1
\\ \hline\end{tabular}\end{center}

The explicit complexes are as follows.
\[\begin{array}{lr}
   X_2:&\;\; 0\to \Proj2\to \Proj{5}\to \Proj3\oplus M_{1,2}\to C_2\to 0.
\\ X_5:&\;\; 0\to \Proj5\to\Proj{345}\to\Proj{234}\oplus M_{4,1}\to M_{4,1}\oplus M_{4,2}\to C_5\to 0.
\\ X_3:&\;\; 0\to \Proj3\to \Proj{34}\to\Proj{45}\to\Proj5\oplus M_{1,1}\to M_{1,1}\oplus M_{1,2}\to C_3\to 0.
\\ X_4:&\;\; 0\to \Proj4\to \Proj4\to\Proj3\to \Proj3\to \Proj4\to M_{4,2}\to C_4\to 0.\end{array}\]
The cohomology of the complexes above is displayed in the following table.
\begin{center}\begin{tabular}{|c|cccccc|c|}
\hline $X_i$ & $H^{-6}$ & $H^{-5}$ & $H^{-4}$ & $H^{-3}$ & $H^{-2}$ & $H^{-1}$ & Total
\\ \hline $2$&&&&$2$&&$1$&$2-1$
\\ $5$&&&$12/5$&$1$&$1$&&$5-2+1$
\\ $3$&&$1/5/3$&&&$1$&$1$&$3-5-1$
\\ $4$&$23/5/4$&$2$&$1$&$1$&&&$4+5-3$
\\ \hline \end{tabular}\end{center}

\subsubsection{The group $\PSU_4(4)$}

Let $G=\PSU_4(4)$.


\smallskip
\paragraph{Simple modules}

There are five simple modules in the principal $3$-block, of dimensions $1$, $52$, $272$, $832$ and $4096$. We label the modules $S_1$ to $S_5$ so that they have dimensions $1$, $52$, $832$, $4096$ and $272$ respectively. The Green
correspondents are $C_i=T_i$.

\smallskip
\paragraph{The perverse equivalence}

There is a perverse equivalence between $kB_0(G)$ and $kB_0(N)$ with 
local twist $\eta_{Q_1}=4$ and $\eta_{Q_2}=2$ and with
the $\per$-values on the left, which makes the decomposition matrix look as follows.

\begin{center}\begin{tabular}{|c|c|ccccc|}
\hline $\per$ & Ord.\ Char.{} & $S_1$&$S_2$&$S_5$ & $S_3$ & $S_4$
\\ \hline $0$ & $1$ &  1&&&&
\\ $6$ & $q(q^2-q+1)$ &&1&&&
\\ $8$ & $q^2(q^2+1)$ &&&1&&
\\ $10$ &$q^3(q^2-q+1)$&&&&1&
\\ $12$ & $q^6$        &&&&&1
\\ \hline &$(q+1)^2(q^2-q+1)$ & 1&1&1&&
\\ &$(q+1)(q^2+1)(q^2-q+1)$ & 1&&1&1&
\\ &$q^2(q+1)^2(q^2-q+1)$ & &&1&1&1
\\ &$q(q+1)(q^2+1)(q^2-q+1)$ & &1&1&&1
\\
\hline\end{tabular}\end{center}

The explicit complexes are too long to write down here, but we make a record of the relatively projective modules involved.
\begin{center}\begin{tabular}{|c|cccc|}
\hline $X_i$& $-4$ & $-3$ & $-2$ & $-1$
\\\hline$X_2$& &&$M_{4,2}$&$M_{4,2}$
\\$X_5$& $M_{1,1}$&$M_{1,1}$&$M_{1,1}\oplus M_{1,2}$&$M_{1,1}\oplus M_{1,2}$
\\ $X_3$&$M_{4,1}$&$M_{4,1}$&$M_{4,1}$&$M_{4,1}$
\\ $X_4$&$M_{4,1}$&$M_{4,1}$&$M_{4,1}\oplus M_{4,2}$&$M_{4,1}\oplus M_{4,2}$
\\ \hline\end{tabular}\end{center}
The cohomology of the complexes above is displayed in the following table.
\begin{center}\begin{tabular}{|c|cccccccccccc|c|}
\hline $X_i$ & $H^{-12}$ & $H^{-11}$ & $H^{-10}$ & $H^{-9}$ & $H^{-8}$ & $H^{-7}$ & $H^{-6}$ & $H^{-5}$ & $H^{-4}$ & $H^{-3}$ & $H^{-2}$ & $H^{-1}$ & Total
\\ \hline $X_2$ &&&&&&&2&&1&1&&&$2$
\\ $X_5$ &&&&&$12/5$&$12$&$1$&$1$&&&$1$&$1$&$5$
\\ $X_3$ &&&$1/5/3$&$12/5$&$2$&&&&$1$&$1$&&&$3$
\\ $X_4$ &$23/5/4$&$5/23$&$2/5$&$12/5$&$1$&&&$1$&$1$&$1$&$1$&&$4$
\\ \hline \end{tabular}\end{center}

\subsubsection{The group $\PSL_5(2)$}

Let $G=\PSL_5(2)$.
In \cite[Theorem 0.2]{koshitanimiyachi2000}  it is shown
 that the principal block of $G$ is splendidly Morita equivalent 
to that of $N$. We provide here a different equivalence.

\smallskip
\paragraph{Simple modules}

There are five simple modules in the principal $3$-block, of dimensions $1$, $124$, $155$, $217$ and $868$. We label the modules $S_1$ to $S_5$ so that they have dimensions $1$, $124$, $217$, $868$ and $155$ respectively. The Green
correspondents $C_i$ are
\[ C_1=1,\quad C_2=2,\quad C_3=4,\quad C_4=3,\quad C_5=5.\]

\paragraph{The perverse equivalence}

There is a perverse equivalence between $kB_0(G)$ and $kB_0(N)$ with
local twists $\eta_{Q_1}=3$ and $\eta_{Q_2}=2$ and with
 the $\per$-values on the left, which makes the decomposition matrix look as follows.

\begin{center}\begin{tabular}{|c|c|ccccc|}
\hline $\per$ & Ord.\ Char & $S_1$ & $S_2$ & $S_5$ & $S_3$ & $S_4$
\\ \hline $0$ & $1$ & 1&&&&
\\ $6$ & $q^2(q^4+q^3+q^2+q+1)$ &&1&&&
\\ $7$ & $q^3(q^2+1)(q^2+q+1)$  &1&1&1&& 
\\ $8$ & $q^4(q^4+q^3+q^2+q+1)$ &&1&1&1&
\\ $10$ & $q^{10}$ &1&&1&&1
\\ \hline &$(q-1)(q^2+1)(q^4+q^3+q^2+q+1)$ &&&1&&
\\ &$(q-1)^2(q^2+q+1)(q^4+q^3+q^2+q+1)$ &&&&1&
\\ &$q^2(q-1)^2(q^2+q+1)(q^4+q^3+q^2+q+1)$ &&&&&1
\\ &$q^3(q-1)(q^2+1)(q^4+q^3+q^2+q+1)$ &&&1&1&1
\\ \hline\end{tabular}\end{center}
The explicit complexes are as follows.

\begin{footnotesize}\[\!\!\!\begin{array}{lr}X_2:&\;\; \Proj2\to\Proj5\to\Proj{35}\to\Proj{234}\to\Proj4\oplus M_{4,2}\to M_{4,2}\to C_2.
\\ X_5:&\;\; \Proj5\to\Proj{345}\to\Proj{23344}\to\Proj{23445}\to\Proj{455}\oplus M_{1,1}\to\Proj5\oplus M_{1,1}\oplus M_{1,2}\to M_{1,1}\oplus M_{1,2}\to C_5.
\\ X_3:&\!\!\!\! \Proj3\to\Proj{34}\to\Proj{45}\to\Proj{55}\to\Proj{3455}\to\Proj{23445}\oplus M_{4,1}\to\Proj{234}\oplus M_{4,1}\oplus M_{4,2}\to M_{4,1}\oplus M_{4,2}\to C_4.
\\ X_4:&\;\; \Proj4\to\Proj4\to\Proj4\to\Proj{34}\to\Proj{35}\to\Proj{55}\to\Proj{2345}\to\Proj{234} \oplus M_{4,1}\to\Proj4\oplus M_{4,1}\to M_{4,1}\to C_3.
\end{array}\]\end{footnotesize}

The cohomology of the complexes above is displayed in the following table.

$$\begin{array}{|c|cccccccccc|c|}
\hline X_i & H^{-10} & H^{-9} & H^{-8} & H^{-7} & H^{-6} & H^{-5} & H^{-4} &
 H^{-3} & H^{-2} & H^{-1} & \text{Total}
\\ \hline X_2 & &&&&2&&1&1&&&2
\\ X_5 & &&&12/5 & 1&1&&&1&1&5-1-2
\\ X_3 & &&1/5/3 &&& 1&1&1&1&&3-5+1
\\ X_4 & 23/5/4 & 5/23 & 2/5 &&&&1&1&&&4+2-5
\\ \hline \end{array}$$

\subsubsection{The group $\PSU_5(4)$}

Let $G=\PSU_5(4)$.

\smallskip
\paragraph{Simple modules}

There are five simple modules in the principal $3$-block, of dimensions $1$, $3280$, $14144$, $52840$ and $1048576$. We label the modules $S_1$ to $S_5$ so that they have dimensions $1$, $52840$, $3280$, $1048576$ and $14144$. The Green
correspondents are $C_i=T_i$.

\smallskip
\paragraph{The perverse equivalence}

There is a perverse equivalence between $kB_0(G)$ and $kB_0(N)$ with
local twist $\eta_{Q_1}=6$ and $\eta_{Q_2}=4$ and with
 the $\per$-values on the left, which makes the decomposition matrix look as follows.


\smallskip
\begin{center}\begin{tabular}{|c|c|ccccc|}
\hline $\per$ & Ord.\ Char & $S_1$ & $S_2$ & $S_5$ & $S_3$ & $S_4$
\\ \hline $0$ & $1$ & 1&&&&
\\ $12$ & $q^2(q^4-q^3+q^2-q+1)$ &&1&&&
\\ $14$ & $q^3(q^2+1)(q^2-q+1)$  &&&1&& 
\\ $16$ & $q^4(q^4-q^3+q^2-q+1)$ &&&&1&
\\ $20$ & $q^{10}$ &&&&&1
\\ \hline &$(q+1)(q^2+1)(q^4-q^3+q^2-q+1)$ & 1&1&1&&
\\ &$(q+1)^2(q^2-q+1)(q^4-q^3+q^2-q+1)$ & 1&&1&1&
\\ &$q^3(q+1)(q^2+1)(q^4-q^3+q^2-q+1)$ & &&1&1&1
\\ &$q^2(q+1)^2(q^2-q+1)(q^4-q^3+q^2-q+1)$ & &1&1&&1
\\ \hline\end{tabular}\end{center}

\smallskip
The explicit complexes are too long to write down here, but we make a record of the relatively projective modules involved.

\smallskip
\begin{center}\begin{tabular}{|c|cccccc|}
\hline $X_i$& $-6$&$-5$&$-4$ & $-3$ & $-2$ & $-1$
\\\hline$X_2$&&&$M_{4,2}$&$M_{4,2}$&$M_{4,2}$&$M_{4,2}$
\\$X_5$&$M_{1,1}$&$M_{1,1}$&$M_{1,1}\oplus M_{1,2}$&$M_{1,1}\oplus M_{1,2}$&$M_{1,1}\oplus M_{1,2}$&$M_{1,1}\oplus M_{1,2}$
\\ $X_3$&$M_{4,1}$&$M_{4,1}$&$M_{4,1}$&$M_{4,1}$&$M_{4,1}$&$M_{4,1}$
\\ $X_4$&$M_{4,1}$&$M_{4,1}$&$M_{4,1}\oplus M_{4,2}$&$M_{4,1}\oplus M_{4,2}$&$M_{4,1}\oplus M_{4,2}$&$M_{4,1}\oplus M_{4,2}$
\\ \hline\end{tabular}\end{center}

\smallskip
The cohomology of the complexes above is displayed in the following table.

\smallskip
\begin{center}\begin{tabular}{|c|cccccccccc|c|}
\hline $X_i$ & $H^{-20}$ & $H^{-19}$ & $H^{-18}$ & $H^{-17}$ & $H^{-16}$ & $H^{-15}$ & $H^{-14}$ & $H^{-13}$ & $H^{-12}$ & $H^{-11}$ &
\\ \hline $X_2$ &&&&&&&&&$2$&&
\\ $X_5$ &&&&&&&$12/5$&$12$&$1$&$1$&
\\ $X_3$ &&&&&$1/5/3$&$12/5$&$2$&&&&
\\ $X_4$ &$23/5/4$&$5/23$&$23/5$&$5/23$&$2/5$&$12/5$&$1$&&&$1$&
\\\hline $X_i$ & $H^{-10}$ & $H^{-9}$ & $H^{-8}$ & $H^{-7}$ & $H^{-6}$ & $H^{-5}$ & $H^{-4}$ & $H^{-3}$ & $H^{-2}$ & $H^{-1}$ & Total
\\ \hline $X_2$ &$1$&$1$&&&&&$1$&$1$&&&$2$
\\ $X_5$ &&&$1$&$1$&$1$&$1$&&&$1$&$1$&$5$
\\ $X_3$ &$1$&$1$&&&&&$1$&$1$&&&$3$
\\ $X_4$ &$1$&$1$&$1$&&&$1$&$1$&$1$&$1$&&$4$
\\ \hline \end{tabular}\end{center}

\subsection{Automizer $SD_{16}$}
\label{se:autoD16}

The group $\Cyc3^2\rtimes SD_{16}$ has seven simple $\F_3$-modules, all
absolutely simple. There are four of dimension $1$ -- denoted $T_1$ to $T_4$ -- and three of dimension $2$ -- denoted
$T_5$ to $T_7$. They are chosen so that $\ker T_3=D_8$ and $\ker T_4=Q_8$.
There is a unique labelling so that the projective indecomposable modules are given below:
\[ \Proj1=\begin{array}{c}1\\7\\35\\6\\1\end{array},\quad \Proj2=\begin{array}{c}2\\7\\45\\6\\2\end{array},\quad \Proj3=\begin{array}{c}3\\6\\15\\7\\3\end{array},\quad \Proj4=\begin{array}{c}4\\6\\25\\7\\4\end{array}\]
\[\Proj5=\begin{array}{c}5\\67\\12345\\67\\5\end{array},\quad \Proj6=\begin{array}{c}6\\125\\677\\345\\6\end{array},\quad \Proj7=\begin{array}{c}7\\345\\667\\125\\7\end{array}.\]

Let $Q$ be a representative of the unique conjugacy class of subgroups of $P$ of order $3$.
There are four modules with vertex $Q$ and trivial source, labelled as below.
\[ M_{1,1}=\begin{array}{c} 135\\67\\135\end{array},\quad M_{2,1}=\begin{array}{c} 245\\67\\245\end{array},\quad M_{3,1}=\begin{array}{c} 67\\135\\67\end{array}, \quad M_{4,1}=\begin{array}{c} 67\\245\\67\end{array}.\]
There are also four modules with vertex $C_3$ and $2$-dimensional source, labelled as below.
\[ M_{1,2}=\begin{array}{c} 135\\6677\\123455\\67\end{array},\quad M_{2,2}=\begin{array}{c} 245\\6677\\123455\\67\end{array},\quad M_{3,2}=\begin{array}{c} 67\\123455\\6677\\135\end{array}, \quad M_{4,2}=\begin{array}{c} 67\\123455\\6677\\245\end{array}.\]

The perversity and local twist functions are given in the following table.

$$\begin{array}{|c|ccccccc|c|}
\cline{1-9}
  & T_1 & T_2 & T_3 & T_4 & T_5 & T_6 & T_7 & \eta(Q)\\
\cline{1-9}
M_{11} & 0 & 4 & 2 & 5 & 7 & 6 & 3 & 1\\
M_{23} & 0 & 1 & 0 & 0 & 2 & 0 & 0 & 0\\
M_{22}.2,\GA_6.2^2, \PSp_4(q).\Cyc{2n} & 0 & 4 & 0 & 4 & 3 & 3 & 3 & 1\\
\PSp_4(r).\Cyc{2n} & 0 & 8 & 0 & 8 & 6 & 6 & 6 & 2\\
\PSL_3(r).\Cyc{n} & 0 & 2 & 0 & 2 & 3 & 3 & 3 & 0\\
\PSU_3(q).\Cyc{n} & 0 & 4 & 0 & 4 & 6 & 6 & 6 & 0\\
HS     & 0 & 7 & 0 & 4 & 3 & 10& 3 & 1\\
\cline{1-9}
\end{array}$$

\subsubsection{The Mathieu group $M_{11}$}
\label{sec:M11}

Let $G=M_{11}$. By  \cite[Example 4.9]{Ok1}, there is a derived equivalence
between $kB_0(G)$ and $kB_0(N)$. In this section we will produce a perverse equivalence between the two blocks.

\smallskip
\paragraph{Simple modules}

There are seven simple modules in the principal $3$-block, of dimensions $1$, $5$, $5$, $10$, $10$, $10$ and $24$. The ordering on the $S_i$ is the chosen
such that $S_2$ and $S_4$ are $5$-dimensional (dual) modules,
$S_3$ and $S_5$ are $10$-dimensional (dual) modules,
$S_6$ is the $24$-dimensional module,
and $S_7$ is the self-dual $10$-dimensional module. The choice of $S_2$ through to $S_7$ is such that the Green correspondents are given below.
\[ C_1=1,\quad C_2=\begin{array}{c}3\\6\\5\end{array},\quad C_3=\begin{array}{c}7\\45\\6\\27\end{array},\quad C_4=\begin{array}{c}5\\7\\3\end{array},\quad C_5=\begin{array}{c}2\\67\\45\\6\end{array},\quad C_6=\begin{array}{c}6\\12\\7\end{array},\quad C_7=4.\]

\paragraph{The perverse equivalence}

There is a perverse equivalence between $kB_0(G)$ and $kB_0(N)$ with 
local twist $\eta_Q=1$ and with
the $\per$-values on the left, which makes the decomposition matrix look as follows.
\begin{center}\begin{tabular}{|c|c|ccccccc|}
\hline $\per$ & Ord.\ Char.{} & $S_1$ & $S_3$ & $S_7$ & $S_2$ & $S_4$ & $S_6$ & $S_5$
\\ \hline
   $0$&$1$ &1& & & & & &
\\ $2$&$10$& &1& & & & &
\\ $3$&$10$& & &1& & & &
\\ $4$&$16$&1&1& &1& & &
\\ $5$&$11$&1& & &1&1& &
\\ $6$&$44$& & &1&1&1&1&
\\ $7$&$55$&1&1& &1&1&1&1
\\\hline
      &$10$& & & & & & &1
\\    &$16$&1& & & &1& &1
\\ \hline \end{tabular}\end{center}
The explicit complexes are as follows.
\[\begin{array}{lr} X_3:&\;\; 0\to \Proj3\to \Proj7\to C_3\to 0.
\\ X_7:&\;\; 0\to \Proj7\to \Proj{25}\to M_{2,1}\to C_7\to 0.
\\ X_2:&\;\; 0\to \Proj2\to \Proj6\to \Proj{67}\to \Proj3\oplus M_{4,2}\to C_2\to 0.
\\ X_4:&\;\; 0\to \Proj4\to \Proj5\to \Proj{25}\to \Proj{246}\to M_{2,2}\to C_4\to 0.
\\ X_6:&\;\; 0\to \Proj6\to \Proj5\to \Proj4\to \Proj2\to \Proj5\to \Proj6\to C_6\to 0.
\\ X_5:&\;\; 0\to \Proj5\to \Proj5\to \Proj6\to \Proj{46}\to \Proj{45}\to \Proj{25}\to \Proj{26}\to C_5\to 0.\end{array}\]

The cohomology of the complexes above is displayed in the following table.

\smallskip
\begin{center}\begin{tabular}{|c|ccccccc|c|}
\hline $X_i$ & $H^{-7}$ & $H^{-6}$ & $H^{-5}$ & $H^{-4}$ & $H^{-3}$ & $H^{-2}$ & $H^{-1}$ & Total
\\ \hline $3$& & & & & &$3$&&$3$
\\ $7$& & & & &$1/7$&$1$&&$7$
\\ $2$& & & &$2$&$3$& &$1$& $2-3-1$
\\ $4$& & &$2/7/4$&$1/7,3$&$1$ & & & $4-2+3$
\\ $6$& &$7/34/6$ &$3$ & & &$1$ & & $6-4-7+1$
\\ $5$&$1234/67/5$ &$7/34,1,2$ &$1/7/3$ &$3$ &$3$ &$1$ &$1$ & $5-6+7-3-1$
\\ \hline \end{tabular}\end{center}

\smallskip
Here, ``$1/7,3$'' means a direct sum of $1/7$ and $3$ and ``$7/34,1,2$''
 a direct
sum of $7/34$, $1$ and $2$.

\subsubsection{The Mathieu group $M_{23}$}
\label{sec:M23}

Let $G=M_{23}$.

\smallskip
\paragraph{Simple modules}

There are seven simple modules in the principal $3$-block, of dimensions $1$, $22$, $104$, $104$, $253$, $770$ and $770$. The ordering on the $S_i$ is such that $S_2$ and $S_5$ are $104$-dimensional (dual) modules, $S_3$ is $253$-dimensional, $S_4$ is $22$-dimensional, and $S_6$ and $S_7$ are (dual) $770$-dimensional modules. The choice of $S_2$ and $S_5$ through to $S_7$ is such that the Green correspondents are given below.
\[ C_1=1,\quad C_2=\begin{array}{c}2\\7\\5\end{array},\quad C_3=3,\quad C_4=4,\quad C_5=\begin{array}{c}5\\6\\2\end{array},\quad C_6=6,\quad C_7=7.\]

\paragraph{The perverse equivalence}

There is a perverse equivalence between $kB_0(G)$ and $kB_0(N)$ with
local twist $\eta_Q=0$ and with
the $\per$-values on the left, which makes the decomposition matrix look as follows. This
equivalence has been constructed previously by Okuyama  \cite[Example 4.7]{Ok1}.

\begin{center}\begin{tabular}{|c|c|ccccccc|}
\hline $\per$ & Ord.\ Char.{} & $S_1$ & $S_3$ & $S_4$ & $S_6$ & $S_7$ & $S_2$ & $S_5$
\\ \hline
   $0$&$1$ &1& & & & & &
\\ $0$&$253$& &1& & & & &
\\ $0$&$22$& & &1& & & &
\\ $0$&$770$& & & &1& & &
\\ $0$&$770$& & & & &1& &
\\ $1$&$896$& & &1&1& &1&
\\ $2$&$230$& & &1& & &1&1
\\\hline
      &$896$& & &1& &1& &1
\\   &$2024$&1&1&1&1&1&1&1
\\ \hline \end{tabular}\end{center}

\smallskip
The explicit complexes are as follows.
\[\begin{array}{lr} X_2:&\;\; 0\to \Proj2\to C_2\to 0.
\\ X_5:&\;\; 0\to \Proj5\to \Proj5\to C_5\to 0.\end{array}\]
The cohomology of the complexes above is as follows: the first complex $X_2$ has $H^{-1}(X_2)=4/6/2$, and the second complex $X_5$ has $H^{-2}(X_5)=1234/67/5$ and $H^{-1}(X_5)=1\oplus 7/34$.

\subsubsection{The extended Mathieu group $M_{22}.2$}
\label{sec:M22.2}

Let $G=M_{22}.2$.
In this section we will produce a perverse equivalence between $kB_0(G)$
and $kB_0(N)$.
As we mentioned in \S \ref{sec:M22}, there is a splendid Morita equivalence
between the principal blocks of $M_{22}$ and $M_{10}=\GA_6.2_3$, and this
extends to a splendid Morita equivalence between the principal blocks
of $M_{22}.2$ and $\GA_6.2^2$. As we saw in \S \ref{sec:S6}, there is
a perverse equivalence between the principal blocks of $\GA_6.2^2$ and its
normalizer, and so therefore the same holds for $M_{22}.2$.
As we have already proved that the perverse equivalence exists, and have
described it earlier, we simply give the simple modules, and then the decomposition matrix and $\per$-values.

\smallskip
\paragraph{Simple modules}

There are seven simple modules in the principal $3$-block, of dimensions $1$, $1$, $55$, $55$, $98$, $231$ and $231$. The ordering on the $S_i$ is such that $S_2$ and $S_4$ are $231$-dimensional modules, $S_3$ and $S_6$ are $55$-dimensional modules, $S_5$ is $98$-dimensional, and $S_7$ is the non-trivial $1$-dimensional module. The choice of $S_2$ through to $S_4$ and $S_6$ is such that the Green correspondents are given below.
\[ C_1=1,\quad C_2=\begin{array}{c}7\\34\\6\end{array},\quad C_3=3,\quad C_4=\begin{array}{c}6\\12\\7\end{array},\quad C_5=\begin{array}{c}5\\67\\5\end{array},\quad C_6=2,\quad C_7=4.\]

\paragraph{The perverse equivalence}

There is a perverse equivalence between $kB_0(G)$ and $kB_0(N)$ with the $\per$-values on the left, which makes the decomposition matrix look as follows.

\smallskip
\begin{center}\begin{tabular}{|c|c|ccccccc|}
\hline $\per$ & Ord.\ Char.{} & $S_1$ & $S_3$ & $S_5$ & $S_7$ & $S_6$ & $S_2$ & $S_4$
\\ \hline
   $0$&$1_1$&1& & & & & &
\\$0$&$55_1$& &1& & & & &
\\$3$&$154_1$&1&1&1& & & &
\\ $3$&$1_2$& & & &1& & &
\\$3$&$55_2$& & & & &1& &
\\ $4$&$385_1$&1& &1& &1&1&
\\ $4$&$385_2$& &1&1&1& & &1
\\\hline
    &$154_2$& & &1&1&1& &
\\   &$560$& & &1& & &1&1
\\ \hline \end{tabular}\end{center}

\subsubsection{The Higman--Sims group $HS$}

Let $G=HS$. By \cite[Example 4.8]{Ok1}, there is a derived equivalence between $kB_0(G)$ and $kB_0(N)$. In this section we will produce a perverse equivalence between the two blocks.

\smallskip
\paragraph{Simple modules}

There are seven simple modules in the principal $3$-block, of dimensions $1$, $22$, $154$, $321$, $748$, $1176$ and $1253$. The ordering on the $S_i$ is such that the dimensions of the $S_i$ are (in order) $1$, $1176$, $154$, $321$, $1253$, $748$ and $22$. The Green correspondents are given below.
\[ C_1=1,\quad C_2=\begin{array}{c}7\\34\\6\end{array},\quad C_3=3,\quad C_4=\begin{array}{c}6\\12\\7\end{array},\quad C_5=\begin{array}{c}5\\67\\5\end{array},\quad C_6=\begin{array}{c}25\\67\\25\end{array},\quad C_7=4.\]

\paragraph{The perverse equivalence}

There is a perverse equivalence between $kB_0(G)$ and $kB_0(N)$ with local
twist $\eta_Q=1$ and with the $\per$-values on the left, which makes the decomposition matrix look as follows.
\begin{center}\begin{tabular}{|c|c|ccccccc|}
\hline $\per$ & Ord.\ Char.{} & $S_1$ & $S_3$ & $S_7$ & $S_5$ & $S_4$ & $S_2$ & $S_6$
\\ \hline
   $0$&$1$   &1& & & & & &
\\ $0$&$154$ & &1& & & & &
\\ $3$&$22$  & & &1& & & &
\\ $3$&$1408$&1&1& &1& & &
\\ $4$&$1750$& &1&1&1&1& &
\\ $7$&$2750$& & & &1&1&1&
\\ $10$&$3200$&1& &1&1& &1&1
\\\hline
      &$770$& & &1& & & &1
\\    &$1925$&1& & & & &1&1
\\ \hline \end{tabular}\end{center}

\smallskip
The explicit complexes are as follows.
\[\begin{array}{lr}
   X_7:&\;\; \Proj7\to \Proj{25}\to M_{2,1}\to C_7.
\\ X_5:&\;\; \Proj5\to \Proj{67}\to \Proj5\oplus M_{3,1}\to C_5.
\\ X_4:&\;\; \Proj4\to \Proj2\to \Proj5\to \Proj6\to C_4.
\\ X_2:&\;\; \Proj2\to \Proj6\to \Proj6\to \Proj2\to \Proj4\to \Proj5\to \Proj7\to C_2.
\\ X_6:&\;\; \Proj6\to \Proj6\to \Proj6\to \Proj6\to \Proj2\to \Proj2\to \Proj6\to \Proj6\to \Proj{47}\to \Proj{25}\to C_6.\end{array}\]
The cohomology of the complexes above is displayed in the following table.

\smallskip
\begin{footnotesize}
\begin{center}\begin{tabular}{|c|cccccccccc|c|}
\hline $X_i$ & $H^{-10}$ & $H^{-9}$ & $H^{-8}$ & $H^{-7}$ & $H^{-6}$ & $H^{-5}$ & $H^{-4}$ & $H^{-3}$ & $H^{-2}$ & $H^{-1}$ & Total
\\ \hline $7$& & & & & & & & $1/7$& $1$ & & $7$
\\ $5$& & & & & & & & $5$ & & $13$& $5-1-3$
\\ $4$& & & & & & & $5/7/4$& & $1$ & & $4+1-5-7$
\\ $2$& & & &$2$ &$1/7/3$ &$15/77/34$ & $5$& & $3$& & $2+3+7-4$
\\ $6$& $A_1$& $A_2$& $A_3$& $15/77/34$& $7/45$& &$1/7/3$ &$3$ & & $1$& $6-2+4-5-7-7$
\\ \hline \end{tabular}\end{center}
\end{footnotesize}

\smallskip
Here, $A_1=12/77/345/6$, $A_2=125/77/34$ and $A_3=12/77/345$.

\smallskip
\paragraph{A non-principal block of $HS$}
\label{se:nonprincipal}

We have $N_G(P)\simeq\Cyc2\times (P\rtimes E)$. We denote by $A$
(resp. $B$) the unique
non-principal block of $\BF_3 G$ (resp. $\BF_3N$) with defect group $P$. We
have a canonical isomorphism
$B\iso \BF_3P\rtimes E$ and we label simple modules for $B$ as described in
\S \ref{se:autoD16}. By \cite[Theorem 0.2]{KoKuWa}, there is a derived equivalence
between $A$ and $B$, and we show it is perverse. Let us recall the construction.

The block $A$ has seven simple modules. We denote by $S_1$ and $S_3$
the (dual) simple modules of dimension $49$, by $S_2$ and $S_4$ the simple
modules of dimension $154$, by $S_5$ the simple module of dimension $77$
and by $S_6$ and $S_7$ the (dual) simple modules of dimension $770$. The
choice is such that the Green correspondents are given below.

$$C_1=\begin{matrix}1\\7\\3\end{matrix},\ C_2=2,\ 
C_3=\begin{matrix}3\\6\\1\end{matrix},\ C_4=4,\ C_5=5,\ C_6=6,\ C_7=7.$$

There is a perverse equivalence between $A$ and $B$ with zero local
twists and with the 
$\per$-values on the left, which makes the decomposition matrix look as follows.
\begin{center}\begin{tabular}{|c|c|ccccccc|}
\hline $\per$ & Ord.\ Char.{} & $S_5$ & $S_2$ & $S_4$ & $S_6$ & $S_7$ & $S_1$ & $S_3$
\\ \hline
   $0$&$77$    &1& & & & & &
\\ $0$&$154_1$ & &1& & & & &
\\ $0$&$154_2$ & & &1& & & &
\\ $0$&$770_1$ & & & &1& & &
\\ $0$&$770_2$ & & & & &1& &
\\ $1$&$896_1$ &1& & &1& &1&
\\ $1$&$896_2$ &1& & & &1& &1
\\\hline
      &$175$   &1& & & & &1&1
\\    &$1925$  &1&1&1&1&1& & 
\\ \hline \end{tabular}\end{center}

\smallskip
The explicit complexes are as follows.
\[\begin{array}{lr}
   X_1:&\;\; 0\to \Proj1\to C_1\to 0.
\\ X_3:&\;\; 0\to \Proj3\to C_3\to 0.
\end{array}\]
We have $H^{-1}(X_1)=5/6/1$ and $H^{-1}(X_3)=5/7/3$.

\smallskip
The outer automorphism of order $2$ of $HS$ swaps the simple modules 
$S_1$ and $S_3$. It follows that the perversity function is equivariant.

\section{Prime $2$}
\label{se:ell2}
In this section, we assume that $\ell=2$.

\subsection{Defect $2\times 2$}
Let $G=\PSL_2(q)$ where $q\equiv 3,5\pmod 8$. We have a splendid Morita equivalence between $B_0(G)$ and
$B_0(\GA_4)$ when $q\equiv 3\pmod 8$ (resp. $B_0(\GA_5)$ when
$q\equiv 5\pmod 8$) \cite{Er}. It can be checked to be compatible with
automorphisms.

There is a perverse equivalence between $kB_0(\GA_5)$ and
$kB_0(\GA_4)$ \cite{Ri1} \cite[\S 3]{Ri5}.
We denote by $T_2$ the non-trivial simple $\BF_2\GA_4$-module.
There are three simple $B_0(\GA_5)$-modules. Over $\BF_2$, the
$2$-dimensional modules $S_{2,1}$ and $S_{2,1}$ amalgamate into
a $4$-dimensional simple module $S_2$.
The Green correspondents are
\[C_1=1,\quad C_2=\begin{array}{c}2\\2\end{array}.\]
There is a perverse equivalence between $B_0(G)$ and $B_0(N)$
with the $\per$-values on the left, which makes the decomposition matrix look as follows.

\smallskip

\begin{center}\begin{tabular}{|c|c|ccc|}
\hline $\per$ & Ord.\ Char.{} & $S_1$ & $S_{2,1}$ & $S_{2,2}$ \\ \hline
   $0$&$1$   &1& & 
\\ $1$&$3_1$ &1&1& 
\\ $1$&$3_2$ &1& &1
\\\hline
      &$5_1$ &1&1&1
\\ \hline \end{tabular}\end{center}
We have
$$X_2= 0\to \Proj2\to C_2\to 0.$$
The cohomology is
$H^{-1}(X_2)=\Omega^{-1}(C_2)=11/2$.

\begin{rem}
Note that \cite[\S 6.3]{Rou3} provides a direct proof of Brou\'e's equivariant
conjecture for blocks with Klein four defect groups, and it is easily
seen that the equivalence constructed is perverse. The existence of
derived equivalences for blocks with Klein four defect groups
is due to Linckelmann \cite[Corollary 1.5]{Li1}.
\end{rem}

\subsection{Defect $2\times 2\times 2$}

\subsubsection{The group $\PSL_2(8)$}
\label{se:l28}

Let $G=\PSL_2(8)$.
A splendid Rickard equivalence has been constructed
in \cite[\S 2.3, Example 2]{Rou1}. It is
not perverse, but we show below that there is a two-step perverse
equivalence in that case. The method used here is similar to that
developed by Okuyama \cite{Ok1}.

\smallskip
Let $P$ be a Sylow $2$-subgroup of $G$ and $N=N_G(P)$. We have
$P\simeq\Cyc2^3$ and $N/P\simeq \Cyc7$.
There are two non-trivial
simple $\BF_2 N$-modules, $T_2$ and $T_3$. The labelling is chosen so that
\[\Proj1=\begin{array}{c}1\\2\\3\\1\end{array},\quad
\Proj2=\begin{array}{c}2\\233\\11123\\2\end{array},\quad 
\Proj3=\begin{array}{c}3\\11123\\223\\3\end{array}\]
Let $S_2$ be the $6$-dimensional irreducible $\BF_2 G$-module and
$S_3$ the $12$-dimensional one. We have
\[\Proj{S_1}=\begin{array}{c}S_1\\S_2\\S_1^3S_3\\S_2^2\\S_1^3S_3\\S_2\\S_1
\end{array},\quad \Proj{S_2}=\begin{array}{c}S_2\\S_1^3\\S_2^3\\S_1^6S_3\\S_2^3\\
S_1^3S_3\\S_2\end{array},\quad
\Proj{S_3}=\begin{array}{c}S_3\\S_2\\S_1\\S_2\\S_1\\S_2\\S_3\end{array}.\]
The Green correspondents are
\[C_2=\begin{array}{c}3\\2\end{array},\quad
C_3=\begin{array}{c}2\\23\\3\end{array}.\]

\smallskip
Consider $A''$ an $\BF_2$-algebra with a perverse equivalence with
$B_0(G)$ with perversity function $(0,1,0)$. The images of the simple
$A$-modules under the corresponding stable equivalence are
$$C''_1=S_1,\ C''_2=\Omega^{-1}\left(
 \begin{matrix}S_1^3\oplus S_3\\S_2\end{matrix}\right),\ C''_3=S_3$$
while the (non-simple) image under the derived equivalence is
$$X''_2=0\to\Proj2 \to C''_2\to 0.$$

\smallskip
Consider $A'$ an $\BF_2$-algebra with a perverse equivalence with
$kN$ with perversity function $(0,2,3)$. The images of the simple
$A$-modules under the corresponding stable equivalence are
\[C'_1=1,\quad C'_2=\Res_N^G C''_2,\quad C'_3=C_3\]
while the images under the derived equivalence are
$$X'_2=0\to\Proj3\to\Proj3\to\Proj2\to C_3\to 0$$
$$X'_3=0\to \Proj2\to \Proj{23}\to C'_2.$$

\smallskip
The composite equivalence
$$A''\mstab \xrightarrow[\sim]{(0,1,0)}B_0(\BF_2 G)\mstab\xrightarrow[\sim]{\Res}
\BF_2 N\mstab \xrightarrow[\sim]{(0,2,3)^{-1}} A'\mstab$$
sends simple modules to simple modules, hence comes from
a Morita equivalence between $A''$ and $A'$ \cite[Theorem 2.1]{Li2}.
Hence, we may assume that $A''=A'$, and the composition
$$D^b(B_0(\BF_2 G))\xrightarrow[\sim]{(0,1,0)^{-1}}
D^b(A')\xrightarrow[\sim]{(0,2,3)}D^b(\BF_2 N)$$
lifts the stable equivalence induced by restriction. Thus there is a $2$-step perverse equivalence
\[B_0(\BF_2G)\xleftarrow{(0,1,0)} A' \xrightarrow{(0,2,3)} \BF_2 N,\]
where the algebra $A'$ is acted on by $\Cyc3$, and the equivalences are
compatible with the action of $\Cyc3=\Out(\PSL_2(8))$.

\subsubsection{The Ree groups}
Let $G={^2G}_2(q)$ where $q=3^{2n+1}$. Note that
${^2G}_2(3)\simeq\PSL_2(8)\rtimes \Cyc{3}$.
There is a splendid Morita equivalence
between $B_0(G)$ and $B_0({^2G}_2(3))$ \cite[Example 3.3]{Ok1}. Therefore
we obtain from \S \ref{se:l28} a derived equivalence as a composition
of two perverse equivalences.


\subsubsection{The Janko group $J_1$}
Let $G=J_1$. In this case, a splendid Rickard equivalence has been
constructed by Gollan and Okuyama \cite{GoOk} and we recall their construction,
with a more direct proof.

\smallskip
\paragraph{Simple modules}
There are four simple modules over $\BF_2$ in the principal block of $G$,
of dimensions $1$, $20$, $76$ and $112$. We label the simple modules
$S_1,\ldots,S_4$ by increasing dimension.

Let $P$ be a Sylow $2$-subgroup of $G$, and write $N=N_G(P)$. 
We have $N\simeq \Cyc2^3\rtimes (\Cyc7\rtimes\Cyc3)$. There are three
non-trivial $\BF_2N$-modules: $T_2$ of dimension $2$, and $T_3$ and
$T_4$ of dimension $3$. The labelling is chosen so that
\[\CP(1)=\begin{array}{c} 1\\3\\4\\1\end{array},\quad 
\CP(2)=\begin{array}{c} 2\\33\\44\\2\end{array},\quad
\CP(3)=\begin{array}{c} 3\\344\\1234\\3\end{array},\quad
\CP(4)=\begin{array}{c} 4\\1234\\334\\4\end{array}.\]

Let $H=\PSL_2(8).3$. 
There are four simple modules over $\BF_2$ in the principal block of $H$,
of dimensions $1$, $2$, $6$ and $12$. We denote by $S'_2$ the 
$12$-dimensional one, by $S'_3$ the $6$-dimensional one and by
$S'_4$ the $2$-dimensional one.

\smallskip
\paragraph{The equivalence}
 The construction
of \S \ref{se:constructionglue} extends to the case of $P\simeq \Cyc2^3$
(see \cite[\S 6.4]{Rou3}). 
Let $Q$ be subgroup of $P$ of order $2$ (there is a unique $N$-conjugacy
class of such subgroups). We have $C_G(Q)/Q\simeq\GA_5$, while
$C_N(Q)/Q\simeq\GA_4$. We define $\CE=\{V_{4,1},V_{4,2}\}$ to be the set of
non-trivial simple modules of the principal block of $kC_G(Q)/Q$, and
we construct as in \S \ref{se:constructionglue}
(see \cite[\S 6.4]{Rou3}) a complex
that induces a splendid standard stable equivalence
$F:\BF_2 B_0(G)\mstab\iso \BF_2 N\mstab$. Let $C_i=F(S_i)$. Let $P_2$ be
the projective cover of the non-trivial simple $\BF_2(C_H(Q)/Q)$-module 
and let $M=\Ind_{C_H(Q)}^H\Res_{C_H(Q)}^{C_H(Q)/Q}P_2$.

We have
$$C_2\simeq\Res^H_N S'_2 \text{ and }
C_4\simeq 0\to M\to V\to 0$$
where $V=\begin{matrix}2 44\\233\end{matrix}$.

\smallskip
Let $A$ be an $\BF_2$-algebra equipped with a standard perverse
equivalence $D^b(A)\iso D^b(\BF_2B_0(G))$, corresponding to
the perversity function $(0,0,1,0)$.
Let $L$ be the largest submodule of $\Proj{S_3}$ containing $S_3$ as a 
submodule and such that $L/S_3$ has no composition factor $S_3$. Then,
$\Res^G_N(\Omega^{-1}L)\simeq \Res^H_N S'_3$ in $kN\mstab$. Furthermore,
$\Res_Q S'_3$ is projective, and hence $F(\Omega^{-1}L)\simeq S'_3$.

\smallskip
Let $A'$ be an $\BF_2$-algebra equipped with a standard perverse
equivalence $D^b(A')\iso D^b(\BF_2B_0(H))$, corresponding to
the perversity function $(0,0,0,1)$.
Let $L'$ be the largest submodule of $\Proj{S'_4}$ containing $S'_4$ as a
submodule and such that $L'/S'_4$ has no composition factor $S'_4$. Then,
$L''=\Res_N^H\Omega^{-1}(L')$ is an indecomposable module of dimension
$72$.  There is an exact sequence
$$0\to M\to \Proj{23344}\oplus V\to L''\to 0$$
showing that $C_4\simeq L''$ in $\BF_2N\mstab$.

\smallskip
We have a diagram of standard stable equivalences
$$\xymatrix{
A\mstab\ar[rr]^-{(0,0,1,0)}_-\sim && \BF_2B_0(G)\mstab\ar[dr]^F_-\sim & &
\BF_2B_0(H)\mstab \ar[dl]_-{\Res}^-\sim && A'\mstab \ar[ll]_-{(0,0,0,1)}^-\sim \\
& & & kN\mstab
}$$
The set of images of simple $A$-modules in $kN\mstab$
coincide with that of simple $A'$-modules. It follows that the composite
equivalence $A\mstab\to A'\mstab$ comes from a Morita equivalence
\cite[Theorem 2.1]{Li2}. So,
we have obtained a $2$-step perverse equivalence
$$kB_0(G) \xleftarrow{(0,0,1,0)} A \xrightarrow{(0,0,0,1)} kB_0(H).$$

\subsection{$\PSL_2(\ell^n)$}
Let $G=\PSL_2(\ell^n)$ for some integer $n\ge 1$ and let
$\tG=\Aut(G)$. Okuyama \cite{Ok2} has
constructed a sequence of derived equivalences as in \S \ref{se:Okuyama}.
The sets $I_r$ used by Okuyama are invariant under $\Out(G)$. It follows
that there is a complex $C$ of $kB_0\bigl(N_{\Delta}(G,\tG)\bigr)$-modules
whose restriction to $G\times N_G(P)^\opp$ is a two-sided tilting complex.
It actually induces a splendid Rickard equivalence.

Note that the equivalences defined by Okuyama are not perverse in general.
It is not known if they are compositions of perverse equivalences.

\end{document}